\documentclass[11pt,reqno]{amsart}
\usepackage{amsmath,amsfonts,amsthm,color,amssymb, mathrsfs, soul}
\usepackage{xcolor}
\usepackage[T1]{fontenc}
\usepackage{wasysym}
\usepackage[normalem]{ulem}
\usepackage{stmaryrd}
\usepackage[left=1in, right=1in, top=1.1in,bottom=1.1in]{geometry}
\usepackage{enumitem}
\usepackage{hyperref}
\usepackage{cancel}
\usepackage{comment}
\usepackage{csquotes}
\usepackage{graphicx}
\usepackage{pstricks}
\usepackage{lmodern}
\numberwithin{equation}{section}

\hypersetup{
	colorlinks   = true,
	citecolor    = blue,
	linkcolor=blue
}
\allowdisplaybreaks
\newtheorem{theorem}{Theorem}[section]
\newtheorem{lemma}{Lemma}[section]
\newtheorem{corollary}{Corollary}[section]
\newtheorem{definition}{Definition}[section]
\newtheorem{example}{Example}[section]
\newtheorem{proposition}{Proposition}[section]

\newtheorem{assumption}{Assumption}[section]

\newtheorem{remark}{Remark}[section]

\newtheorem{assumptionalt}{Assumption}[assumption]

\newenvironment{assumptionp}[1]{
	\renewcommand\theassumptionalt{#1}
	\assumptionalt
}{\endassumptionalt}

\def\lla{\left\langle}
\def\rra{\right\rangle}

\def\R{\mathbb R}
\def\E{\mathbb E}

\begin{document}
	\title[ ]{Anticipated backward stochastic evolution equations and maximum
		principle for path-dependent systems in infinite dimensions}
	\author[G. Liu]{Guomin Liu}
	\address{School of Mathematical Sciences, Nankai University, Tianjin 300071,
		China}
	\email{gmliu@nankai.edu.cn}
	\author[J. Song]{Jian Song}
	\address{Research Center for Mathematics and Interdisciplinary Sciences, State Key Laboratory of Cryptography and Digital Economy Security,
		Shandong University, Qingdao 266237, China}
	\email{txjsong@sdu.edu.cn}
	\author[M. Wang]{Meng Wang}
	\address{Department of Mathematical and Statistical Sciences,
		University of Alberta, Edmonton, Alberta,
		Canada}
	\email{mwang13@ualberta.ca}
	
	\begin{abstract}
		For a class of  path-dependent stochastic evolution equations  driven by cylindrical $Q$-Wiener process, we study the Pontryagin's maximum principle for the
		stochastic recursive optimal control problem. In this infinite-dimensional control system, the state process depends on  its  past trajectory, the control is delayed via an integral with respect to a general finite measure, and  the final cost relies on the delayed state.
		To obtain the maximum principle, we introduce a functional  adjoint operator for the non-anticipative path derivative and establish the well-posedness of an anticipated backward stochastic evolution equation in the path-dependent form, which serves as the adjoint equation.

		 \smallskip
		
	\end{abstract}
	
	\date{}
	\maketitle
	
	\tableofcontents

	\vspace{-10pt}
\section{Introduction}

In this paper, we investigate the stochastic recursive optimal control problem
of \emph{path-dependent} stochastic evolution equation (PSEE for short) evolving in a Hilbert space $H$: 
\begin{equation}
	\left\{ \begin{aligned} dx(t)= &\,\, \Big[A(t)x(t)+b \Big(t,x_{t-K,t},\int_{-K
		}^{0}u(t+s)\mu_1(ds)\Big)\Big]dt \\ & +\Big[B(t)x(t)+\sigma
	\Big(t,x_{t-K,t},\int_{-K }^{0}u(t+s)\mu_1(ds)\Big)\Big]dw(t),\quad t\in[0,T],
		\\ x(t)= &\,\, \gamma(t),\,\,u(t)=v(t),\quad t\in [ -K ,0],
	\end{aligned}\right.  \label{EQ-1}
\end{equation}
with the cost functional   given by $J(u(\cdot )):=y(0)$, where $(y(\cdot),z(\cdot))$ solves the
following  backward stochastic differential equation (BSDE for
short) 
\begin{equation}\label{BSDE}
	\left\{ \begin{aligned} -dy(t)= &\,\, f\Big(t,x_{t-K,t},y(t),z(t), \int_{-K
		}^{0}u(t+s)\mu_1(ds)\Big)dt -z(t)dw(t), \quad t\in[0,T], \\ y(T)= &
		\,\,h\Big(\int_{-K}^0x(T+s)\mu_2(ds)\Big). \end{aligned}\right.
\end{equation}
In this control problem \eqref{EQ-1}-\eqref{BSDE}, $K\ge 0$ is a fixed constant, $x_{t-K,t}$ denotes the path of the  state process $x$   on the time interval $[t-K,t]$ (see \eqref{e:x-t-k-t}),
$w(\cdot )$ is a \emph{cylindrical $Q$-Wiener process} on some
Hilbert space $\mathcal{K}$, $A(t)$ and $B(t)$ are random  unbounded linear
operators, 
the coefficient functions $b,\sigma ,f,$ and $h$ are random  functions
taking values in $H$ or $\mathcal{L}(\mathcal{K};H)$ depending on the context, $u(\cdot)$ is
a control process with values in $U$ which is a convex subset of a Hilbert
space $H_{1}$, and $\mu _{1},\mu _{2}$ are finite measures on $[-K,0]$. 

In the classical optimal control theory, the performance of a control  is usually evaluated by a cost functional (utility function) which consists of  a final cost and a running cost. Duffie and Epstein~\cite
{duffie1992stochastic} introduced the notion of stochastic differential
recursive utility, which was later extended to the form of backward
stochastic differential equation (BSDE for short) by Peng \cite%
{peng1993backward}, El Karoui, Peng and Quenez \cite{el1997backward}. An
optimal control problem with cost functional described by a BSDE is then
called a stochastic recursive optimal control problem.
The Pontryagin's maximum principle is widely recognized as an effective
approach in solving optimal control problems (see \cite%
{peng1990general,doi:10.1137/S0363012996313100,du2013maximum,lu2014general,fuhrman2013stochastic,liu2021maximum}
and the references therein). In particular, Peng \cite{peng1993backward} 
derived a local form of the stochastic maximum principle for finite-dimensional  
stochastic recursive optimal control problems.

Path-dependent differential equations describe a class of systems whose  evolution 
depends not only on the current state but also on the entire past
trajectories. In literature,  studies on the maximum principle for path-dependent stochastic
systems have been focusing on systems with an integral delay with respect to some finite measure.  For instance,  a pointwise delay is an  integral delay with respect to a Dirac delta measure,  a moving average delay is an integral
delay with respect to the Lebesgue measure, and there have been fruitful results on maximum principles for  control systems with such delays. In particular, for the finite-dimensional case, one may refer to   Chen and Wu \cite{10cw}, {\O}ksendal, Sulem and Zhang \cite{11osz} for   pointwise delay  and moving average delay, and to Guatteri and Masiero~\cite%
{guatteri2021stochastic} for an integral delay with respect to a {general finite}
measure; for  the infinite-dimensional case, one can refer to {\O}ksendal, Sulem and Zhang \cite{oksendal2012optimal} and Meng and Shen~\cite{16ms} for the  pointwise delay and  moving average delay, and to Guatteri, Masiero and Orrieri \cite{guatteri2017stochastic} for 
integral delay with respect to {a general finite measure}, {in the state equation of which  the drift term does not involve the control delay,
	and the diffusion is independent of state and control.} 
We also refer the reader to \cite{FMT10,GM23,20lww,MSWZ23,YU20122420,ZX17} and references therein for more results on maximum principle for delay systems. The adjoint equations of
systems with delay,  as derived in the above-mentioned works, are now known as \emph{anticipated BSDEs} (ABSDEs for short), the theory of
which was established by Peng and Yang \cite{10py}.

In contrast, the only work on the maximum principle for general path-dependent
control systems, to the best of  our knowledge, is due to Hu and Peng \cite{hu1996maximum} for finite-dimensional systems,  where a backward  stochastic integral equation of Volterra type  was derived as the adjoint equation. We remark that the system considered in \cite{hu1996maximum} does not contain control delay and all the coefficient functions are deterministic; see Remark \ref{rem:time-state dual}.

In this paper, we aim to derive the maximum principle for
the recursive optimal control problem  \eqref{EQ-1}-\eqref{BSDE} of an infinite-dimensional
path-dependent stochastic system (see Theorem~\ref{SMP}). In
our control system, the past trajectories (of the control and the state) and
the unbounded operators are involved in both drift and diffusion terms, and the
final cost term can also depend on the past of the state.   In view of  the general form of path dependence in our control system, existing methodologies seem insufficient to achieve the desired result. Below, we briefly outline  two critical components of our proof, which also  represent two  main contributions of this work: the dual analysis  of the path derivative operator  and the establishment of the well-posedness for the ABSEE as adjoint equations.

For the non-anticipative  path derivative operator (see \eqref{e:extension}) in the system,  we make use of its operator-valued Dinculeanu-Singer  representing measure  to derive its  adjoint operator in the functional sense, which turns out to be anticipative or non-adapted (see Proposition~\ref{thm:rho*} and  Remark~\ref{rem:rho*}).  This enables us to obtain a BSEE involving  anticipative operators  to  serve as the adjoint equation (see {equation} \eqref{e-adjoint}) in the stochastic maximum principle (see Theorem \ref{SMP}). As a comparison,  a direct functional analytic method is utilized and an adjoint BSDE of Volterra type was derived in \cite{hu1996maximum}.

In our setting, the adjoint equation is a  path-dependent ABSEE  with a running terminal condition on an
interval (see~\eqref{absee-0} and \eqref{e-adjoint}), of which the well-posedness needs to be established. The ABSEE with a running terminal, to our best knowledge, was introduced in \cite{guatteri2017stochastic}. Given that the state equation of \cite{guatteri2017stochastic} incorporates the state's history via  an integral delay with respect to a prescribed finite measure and 
includes neither control delay in the drift term nor state and control in the diffusion term, the generator of the
corresponding ABSEE is linear  and independent of $q$, and depends on the future information of $p$ through an integral with respect to the delay measure. Moreover, in their ABSEE the running datum   $\zeta$ is assumed to be continuous, and the  {$dF$ is assumed to  be a deterministic  finite  
measure}. As a comparison, we establish a well-posedness  result for path-dependent ABSEE~\eqref{absee-0} in a general form, where $F$ is a random process with bounded variation, the datum $\zeta $ is
measurable, and the generator  is nonlinear
in both $p$ and $q$ (see also Remark \ref{rem:GM21}).  The
well-posedness of \eqref{absee-0} is obtained by a combination of  the continuation method,
solution translation, and an approximation argument (see Theorem \ref{ABSEE-thm}), after we establish some  {\it a
	priori} estimates  (see Theorem~\ref{es-p})  by using an infinite-dimensional It\^{o}'s formula (see Lemma~\ref{Ito-lemma}).

The rest of this paper is organized as follows. We collect some preliminaries on infinite-dimensional stochastic analysis in Section~\ref{sec:pre}. In Section~\ref{sec:SDEE}, we prove the well-posedness results for 
path-dependent SEEs and anticipated BSEEs. In Section~\ref{sec:path-derivative},  we investigate the non-anticipative path derivative and its functional adjoint operator.  In Section~\ref{sec:SMP}, we formulate our stochastic recursive optimal
control problem and derive the maximum principle.
Finally, in Section~\ref{sec:APP} we apply our result to  controlled path-dependent parabolic SPDEs and linear
quadratic (LQ)  problems.

\section{Preliminaries}
\label{sec:pre}

In this section, we provide some preliminaries on stochastic calculus in
infinite-dimensional spaces. We refer to \cite{DZ92,PR07} for more details.

Let $X$ and $Y$ be generic Banach spaces. We denote by $\mathcal{L}(X,Y)$
the space of bounded linear operators mapping from $X$ to $Y$, and we write $%
\mathcal{L}(X)$ for $\mathcal{L}(X,X)$ and denote by $I_{X}$ the identity
operator on $X$. Assume $X$ is a separable Hilbert space with an orthonormal basis $%
\{e_j\}_{j=1}^N$, where $N\in \mathbb{N}\cup\{\infty\}$ is a finite number
or infinity depending on the dimension of $X$. In the remainder of this paper, we
focus on the case $N=\infty$, noting that all results also hold for the finite-dimensional setting. We denote by $\mathcal{L}_{2}(X,Y)$ the space of \emph{%
	Hilbert-Schmidt operators} mapping from $X$ to $Y,$ i.e., $\mathcal{L}%
_{2}(X,Y)$ consists of $T\in \mathcal{L}(X,Y)$ satisfying 
\begin{equation*}  \label{e:norm-L2}
	\Vert T\Vert _{\mathcal{L}_{2}(X,Y)}^{2}:=\sum_{j=1}^{\infty}\Vert
	Te_{j}\Vert _{Y}^{2}<\infty.
\end{equation*}
If we assume further that $Y$ is a separable Hilbert space, the space $%
\mathcal{L}_{2}(X,Y)$ of Hilbert-Schmidt operators becomes a separable
Hilbert space with the inner product 
\begin{equation*}
	\left\langle T,G\right\rangle _{\mathcal{L}_{2}(X,Y)}:=\sum_{j=1}^{\infty}%
	\left \langle Te_{j},Ge_{j}\right\rangle _{Y}.
\end{equation*}

Assume on some complete probability space $(\Omega,\mathcal{F},P)$,  $w=\{w(t)\}_{t\in \lbrack 0,T]}$ is a $\mathcal K$-valued  \emph{cylindrical $Q$-Wiener process}, for some separable Hilbert space $\mathcal K$ and symmetric, nonnegative-definite (i.e., self-adjoint) operator $Q\in \mathcal{L}(\mathcal K)$.  
 More specifically, \begin{equation*}
		w(t)=\sum_{j=1}^{\infty }\beta ^{j}(t)Q^{\frac{1}{2}}e_{j}, ~ t\in[0,T],
	\end{equation*} where  $\big\{\beta ^{j}(t), t\in[0,T]\big\}_{j\in \mathbb{N}}$ is a family of independent
	one-dimensional standard Brownian motions on $(\Omega,\mathcal{F},P)$, $Q^{\frac{1}{2}}$ is the nonnegative square root of $Q$, and $\{e_{j}\}_{j=1}^{\infty }$ is an orthonormal basis diagonalizing $Q$, i.e.,
$
		Qe_{j}=\lambda _{j}e_{j},~j\in \mathbb{N},
$
	with $\lambda _{j}\geq 0$ being the eigenvalues of $Q$. Note that if $Q$ has a finite trace, $w$ is   a \emph{standard} $\mathcal K$-valued Wiener process of trace class; if   $Q=I_{\mathcal K}$,  $w$ is a \emph{cylindrical Wiener process}.

Let $\mathbb{F}=\{\mathcal{F}_{t}\}_{t\geq 0}$ be the filtration generated
by the Wiener process $\{w(t)\}_{t\in \lbrack 0,T]}$ and augmented by the
class of all $P$-null sets of $\mathcal{F}$. Let $E$ denote a generic
separable Hilbert space with norm $\Vert \cdot \Vert _{E}$. We introduce the
following spaces that will be used in the paper.

\begin{itemize}
	\item For any $\sigma $-algebra $\mathcal{G}$, $L^{2}(\mathcal{G};E)$ is the
	set of all $\mathcal{G}$-measurable random variables $\xi $ taking values in 
	$E$ such that 
	\begin{equation*}
		\mathbb{E}\left[ \Vert \xi \Vert _{E}^{2}\right] <\infty.
	\end{equation*}
	
	\item $L^{2}(0,T;E)$ denotes the set of all $E$-valued deterministic
	processes $\phi =\{\phi (t),\ t\in \lbrack 0,T]\}$ such that 
	\begin{equation*}
		\int_{0}^{T}\Vert \phi (t)\Vert _{E}^{2}dt<\infty .
	\end{equation*}
	
	\item $L_{\mathbb{F}}^{2}(0,T;E)$ denotes the set of all $E$-valued $\mathbb{%
		F}$-adapted processes $\phi =\{\phi (t,\omega ),\ (t,\omega )\in \lbrack
	0,T]\times \Omega \}$ such that 
	\begin{equation*}
		\mathbb{E}\Big[ \int_{0}^{T}\Vert \phi (t)\Vert _{E}^{2}dt\Big] <\infty .
	\end{equation*}
	
	\item $C_{\mathbb{F}}^{2}(0,T;E)\ $($D_{\mathbb{F}}^{2}(0,T;E),$ resp.)  is
	the set of all $E$-valued $\mathbb{F}$-adapted continuous (c\`{a}dl\`{a}g$,$
	resp.) processes $\phi =\{\phi (t,\omega ),\ (t,\omega )\in \lbrack
	0,T]\times \Omega \}$ such that 
	\begin{equation*}
		\mathbb{E}\Big[\sup_{t\in[0, T]}\Vert \phi (t)\Vert _{E}^{2}\Big]<\infty .
	\end{equation*}
	
	\item Given an $\mathbb{F}$-adapted finite-variation process  $F$ on $[0,T]$, $L_{\mathbb{F}%
		,F}^{2}(0,T;E)$ denotes the set of all $E$-valued progressively measurable
	processes $\phi$ satisfying 
	\begin{equation*}
		\mathbb{E}\Big[ \int_{0}^{T}\Vert \phi (t)\Vert _{E}^{2}d|F|_v(t)\Big] %
		<\infty,
	\end{equation*}
	where $|F|_v$ is the total variation process of $F$. In particular, when $%
	F(t)=t$, $L_{\mathbb{F},F}^{2}(0,T;E)$ coincides with $L_{\mathbb{F}%
	}^{2}(0,T;E).$
\end{itemize}

Let $V$ and $H$ be two  separable Hilbert spaces such that $V$ is
densely embedded in $H$. Identify $H$ with its dual space $H^{\ast }$ and
denote by $V^{\ast }$ the dual space of $V$. Then we have $V\subset
H=H^{\ast }\subset V^{\ast }$. Denote by $\left\langle \cdot ,\cdot
\right\rangle _{H}$ (resp. $\left\langle \cdot ,\cdot \right\rangle _{\ast }$%
) the scalar product in $H$ (resp.  the duality product between $V^{\ast }$ and $V$%
). We call $(V,H,V^{\ast })$ a \emph{Gelfand triple}.

Recall that $\mathcal K$ is the Hilbert space where the Wiener process $w$ takes
values. Then its subspace $\mathcal K_{0}:=Q^{\frac{1}{2}}(\mathcal K)$ is a Hilbert space
endowed with the inner product 
\begin{equation*}
	\left\langle u,v\right\rangle _{0}=\langle Q^{-\frac{1}{2}}u,Q^{-\frac{1}{2}%
	}v\rangle _{\mathcal K},\text{ }u,v\in \mathcal K_{0}.
\end{equation*}%
Denote  $\mathcal{L}_{2}^{0}(\mathcal K,H):=\mathcal{L}_{2}(\mathcal K_{0},H)=\mathcal{L}%
_{2}(Q^{\frac{1}{2}}(\mathcal K),H)$, of which the norm is given by 
\begin{equation*}
	\Vert F\Vert _{\mathcal{L}_{2}^{0}(\mathcal K,H)}:=\Vert F\Vert _{\mathcal{L}%
		_{2}(\mathcal K_{0},H)}=\Vert FQ^{\frac{1}{2}}\Vert _{\mathcal{L}_{2}(\mathcal K,H)}.
\end{equation*}%
We also write $\mathcal{L}_{2}^{0}$ for $\mathcal{L}_{2}^{0}(\mathcal K,H)$ for
notation simplicity. For $f\in L_{\mathbb{F}}^{2}(0,T;\mathcal{L}_{2}^{0}),$
we define the stochastic integral with respect to $w$ as follows:%
\begin{equation*}
	\int_{0}^{T}f(t)dw(t):=\sum_{k=1}^{\infty }\int_{0}^{T}f(t)Q^{\frac{1}{2}%
	}e_{k}d\beta ^{k}(t),
\end{equation*}%
where the right-hand side is understood as a limit in $L^{2}(\mathcal{F}%
_{T};H)$. The process $\int_{0}^{t}f(s)dw(s)$ is an $H$-valued continuous
martingale satisfying the It\^o isometry 
\begin{equation*}
	\mathbb{E}\Big[ \Big\Vert \int_{0}^{t}f(s)dw(s)\Big\Vert _{H}^{2}\Big]
	=\mathbb{E}\Big[\int_{0}^{t}\Vert f(s)\Vert _{\mathcal{L}_{2}^{0}}^{2}ds%
	\Big]=\mathbb{E}\Big[\int_{0}^{t}\Vert f(s)Q^{\frac{1}{2}}\Vert _{\mathcal{L}
		_{2}(\mathcal K,H)}^{2}ds\Big],
\end{equation*}
and the  Burkholder-Davis-Gundy inequality: for some constant $C>0$, 
\begin{align*}
	\mathbb{E}\Big[ \sup_{t\in[0,T]}\Big\Vert
	\int_{0}^{t}f(s)dw(s)\Big\Vert _{H}^{2}\Big] &\leq C\mathbb{E}\Big[
	\int_{0}^{T}\Vert f(t)\Vert _{\mathcal{L}_{2}^{0}}^{2}dt\Big]\\&=C\mathbb{E%
	}\Big[\int_{0}^{T}\Vert f(t)Q^{\frac{1}{2}}\Vert _{\mathcal{L}%
		_{2}(\mathcal K,H)}^{2}dt\Big].
\end{align*}

For $f \in L^2_{\mathbb{F}}(0,T;\mathcal{L}_2^0)$ and $g \in
L^2_{\mathbb{F}}(0,T;H)$, \begin{equation*}
	M(t): = \int_{0}^{t} \big< f(s) dw(s), g(s)\big>e_H=\sum_{k=1}^{\infty}
	\int_{0}^{t} \big< f(s) Q^{\frac12}e_{k}, g(s) \big>_H d\beta ^{k}(s),
\end{equation*}
is a real-valued martingale with quadratic variation 
\begin{equation}  \label{e:QV}
	\begin{aligned} & \langle M \rangle(t) = \sum_{k=1}^\infty
		\int_{0}^{t}\langle f(s) Q^{\frac12}e_{k}, g(s)\rangle_H^2 ds \\& \le
		\int_{0}^{t} \sum_{k=1}^\infty \|f(s) Q^{\frac12}e_{k}\|_H^2\| g(s)\|_H^2 ds
		=\int_0^t \|f(s)\|^2_{\mathcal L_2^0} \|g(s)\|_H^2ds. \end{aligned}
\end{equation}

Consider three processes $\{v(t,\omega ),\ (t,\omega )\in \lbrack 0,T]\times
\Omega \}$, $\{M(t,\omega ),\ (t,\omega )\in \lbrack 0,T]\times \Omega \}$
and $\{v^{\ast }(t,\omega ),\ (t,\omega )\in \lbrack 0,T]\times \Omega \}$
with values in $V$, $H$ and $V^{\ast }$, respectively. Assume that $v(t,\omega )$ is
measurable with respect to $(t,\omega )$ and   $\mathcal{F}_{t}$-measurable
with respect to $\omega $ for $t\in \lbrack 0,T]$, and for any $\eta \in V$, the quantity $%
\left\langle v^{\ast }(t,\omega ),\eta \right\rangle _{\ast } $ is
measurable with respect to $(t,\omega )$ and $\mathcal{F}_{t}$-measurable
with respect to $\omega $ for $t\in \lbrack 0,T]$. Let
$M$ be a continuous local martingale and  $\left\langle
M\right\rangle $ be the increasing process part for $\Vert M\Vert
_{H}^{2}$ in the Doob-Meyer decomposition. Suppose $F$ is a real-valued
adapted c\`{a}dl\`{a}g finite-variation process on $[0,T]$, and $\zeta\in L_{\mathbb{F}%
	,F}^{2}(0,T;H)$, $v^{\ast }\in L_{\mathbb{F}}^{2}(0,T;V^{\ast })$, $v\in L_{%
	\mathbb{F}}^{2}(0,T;V)$.

The following It\^o's formula is an extension of  \cite[Theorem 1]{82gk}.
	\begin{lemma}
		\label{Ito-lemma} Suppose that for each $\varphi \in V$, it holds  that
		\begin{equation*}
			\left\langle v(t),\varphi \right\rangle _{H}=\int_{0}^{t}\left\langle
			v^{\ast }(s),\varphi \right\rangle _{\ast }ds+\int_{(0,t]}\left\langle
			\zeta(s),\varphi \right\rangle dF(s)+\left\langle M(t),\varphi \right\rangle
			_{H},
		\end{equation*}
		for $dt\times dP$-almost all $(t,\omega )\in \lbrack 0,T]\times \Omega $.
		Then there exists an adapted c\`{a}dl\`{a}g $H$-valued process $h(\cdot )$
		such that
		
		\begin{itemize}
			\item[(i)] for $dt\times dP$-almost all $(t,\omega )\in \lbrack 0,T]\times
			\Omega $, $h(t, \omega)=v(t,\omega)$;
			
			\item[(ii)] for $t\in \lbrack 0,T]$, it holds almost surely
			\begin{equation}  \label{e:ito-lemma}
				\begin{aligned} \Vert h(t)\Vert _{H}^{2}=&\Vert h(0)\Vert
					_{H}^{2}+2\int_{0}^{t}\left\langle v^{\ast }(s),v(s)\right\rangle _{\ast
					}ds+2\int_{(0,t]}\left\langle h(s),\zeta(s)\right\rangle dF(s)\\
					&+2\int_{0}^{t}\left\langle h(s),dM(s)\right\rangle _{H}+\left\langle
					M\right\rangle (t)-\int_{(0,t]}\Vert \zeta(s)\Vert _{H}^{2}\Delta F(s)dF(s),
				\end{aligned}
			\end{equation}
			where $\Delta F(s) = F(s) -F(s^-)$.
		\end{itemize}
	\end{lemma}
	
	\begin{proof}
		When $v(\cdot)$ is a $V$-valued process such that for each $\varphi\in V$,
		it holds for $dt\times dP$-almost all $(t,\omega)\in[0,T]\times \Omega$  that
		\begin{equation*}
			\left\langle v(t),\varphi \right\rangle _{H}=\int_{0}^{t}\left\langle
			v^{\ast }(s),\varphi \right\rangle _{\ast }ds+\left\langle N(t),\varphi
			\right\rangle _{H},
		\end{equation*}
		where $v^*(\cdot)$ is a $V^*$-valued process and $N(\cdot)$ is an $H$-valued
		c\`adl\`ag local martingale, the It\^o's formula was
		proved in \cite[Theorem 1]{82gk}. The desired result can be obtained by the same
		argument, with the $H$-valued
		c\`adl\`ag martingale $N(t)$ being replaced by the $H$-valued c\`adl\`ag
		semi-martingale $\int_{(0,t]} \zeta(s) dF(s) + M(t)$.
	\end{proof}
	
	\section{PSEEs and anticipated BSEEs}
	
	\label{sec:SDEE}
	
	\subsection{Path-dependent stochastic evolution equations}
	
	Let $K\geq 0$ be a fixed constant. For $t\in \lbrack -K ,0)$,
	we define $\mathcal{F}_{t}:=\mathcal{F}_{0}$. For a process $x(\cdot
	):[-K,T]\rightarrow H$ and $t\in[-K, T]$, its value at time $t$ is denoted by $x(t)$, and we denote $$
	x_t:=\big\{x (t\wedge r),\, r\in[-K, T]\big\}.$$
	
 Let
	\begin{equation*}
		A:[0,T]\times \Omega \rightarrow \mathcal{L}(V,V^{\ast }),\quad
		B:[0,T]\times \Omega \rightarrow \mathcal{L}(V,\mathcal{L}_{2}^{0})
	\end{equation*}	be (random) unbounded linear operators 
	and 
	\begin{equation*}
		b:[0,T]\times \Omega \times { C(-K,T;H)}\rightarrow H,\quad \sigma :[0,T]\times
		\Omega \times C(-K,T;H)\rightarrow \mathcal{L}_{2}^{0}
	\end{equation*}
be	nonlinear functions,  
	where $C(-K,T;H)$ denotes the space of continuous functions
	from $[-K,T]$ to $H$, endowed with the uniform norm $\|x\|_{C(-K,T;H)}=\sup_{t\in[-K,T]}\|x(t)\|_H$. 
	We consider the following path-dependent stochastic  evolution equation (PSEE) in $%
	(V,H,V^{\ast })$:
	\begin{equation}
		\left\{ \begin{aligned} dx(t)=& \big[A(t)x(t)+b(t,x_t)\big]dt
			+\big[B(t)x(t)+\sigma (t,x_t)\big]dw(t),\quad t\in \lbrack 0,T], \\ x(t)=&
			\gamma(t),\quad t\in \lbrack -K ,0], \end{aligned}\right.  \label{sdee-0}
	\end{equation}
	where $\gamma:[-K,0]\rightarrow H$ is
	the initial path. 
	
 Denote
	\[
	\mathcal X :=\Big\{x \ \text{is a process on}\ [-K,T]: x|_{[0,T]}\in L_{\mathbb{F}}^{2}(0 ,T;V) \ \text{and}\ x \in   C_{\mathbb{F}}^{2}(-K,T;H)\Big\},\]
	with norm
	\[
	\Vert x\Vert _{\mathcal X}:=\Big(\big\Vert x|_{[0,T]}\big\Vert^2_{L_{\mathbb{F}}^{2}(0,T;V)}
	+\big\Vert x\big\Vert^2_{C_{\mathbb{F}}^{2}(-K,T;H)}\Big)^{\frac12}
	\] and $x|_{[0,T]}$ denoting  the restriction of $x$ on $[0,T]$.

		Throughout the rest of the paper, we denote by $C$ a generic positive
	constant which may differ line by line.

	To get the existence and uniqueness of the solution, we
	impose the following conditions.
	
	\begin{itemize}
		\item[(A1)] For each $x\in C(-K,T;H)$, $b(\cdot ,\cdot ,x)$, $%
		\sigma (\cdot ,\cdot ,x)$ are progressively measurable. $b(\cdot ,\cdot
		,0)\in L_{\mathbb{F}}^{2}(0,T;H)$, $\sigma (\cdot ,\cdot ,0)\in L_{\mathbb{F}%
		}^{2}(0,T;\mathcal{L}_{2}^{0})$ and {$ \gamma(\cdot )\in 
		C(-K,0;H)$}.
		
		\item[(A2)] For each $u\in V,$ $A(\cdot ,\cdot )u$ and $B(\cdot
		,\cdot )u$ are progressively measurable. There exist $\alpha >0$ and $%
		\lambda \in \mathbb{R}$ such that for each $(t,\omega )\in \lbrack
		0,T]\times \Omega $, 
		\begin{equation*}
			2\left\langle A(t)u,u\right\rangle _{\ast }+\Vert B(t)u\Vert _{\mathcal{L}%
				_{2}^{0}}^{2}\leq -\alpha \Vert u\Vert _{V}^{2}+\lambda \Vert u\Vert
			_{H}^{2},\,\,\,\text{ for all }u\in V.
		\end{equation*}
		
		\item[(A3)] There exists a constant $K_{1}>0$ such that, for each $(t,\omega )\in \lbrack 0,T]\times \Omega $, 
		\begin{equation*}
			\Vert A(t)u\Vert _{\ast }\leq K_{1}\Vert u\Vert _{V},\,\,\,\text{ for all }%
			u\in V.
		\end{equation*}
		
		\item[(A4)] There exists a constant $L_1>0$ such that, for each
		$(t,\omega )\in \lbrack 0,T]\times \Omega $,
		\begin{align*}
			&  \Vert b(t,x_{t})-b(t,x_{t}^{\prime })\Vert
			_{H}^{2}+\Vert \sigma (t,x_{t})-\sigma (t,x_{t}^{\prime })\Vert _{\mathcal{L}
				_{2}^{0}}^{2}\\
			&\leq L_1\sup_{s\in[-K,t]}\Vert x(s)-x^{\prime }(s)\Vert _{H}^{2}, \,\,\, \text{ for all } x,x^{\prime } \in  C(-K,T;H).
		\end{align*}
	\end{itemize}
	
	Note that {(A2)} and {(A3)} yield 
	\begin{equation}
		\Vert B(t)u\Vert _{\mathcal{L}_{2}^{0}}\leq C_{1}\Vert u\Vert _{V},\text{
			for all }u\in V,  \label{Myeq2-41}
	\end{equation}
	where $C_{1}$ is a constant  depending only 
	on $\lambda$ and $K_{1}$. 
	
		\begin{definition}
		\label{def:SDEE} A process $x(\cdot )\in \mathcal X$ is
		called a solution to \eqref{sdee-0}, if for $dt\times dP$-almost all $%
		(t,\omega )\in \lbrack -K,T]\times \Omega ,$ it holds in $V^{\ast }$ that:
		\begin{equation*}
			\left\{ \begin{aligned} x(t)=&\,\,
				\gamma(0)+\int_{0}^{t}A(s)x(s)ds+\int_{0}^{t}b(s,x_s)ds \\ &
				+\int_{0}^{t}\big[B(s)x(s)+\sigma (s,x_s)\big]dw(s),\quad t\in \lbrack 0,T], \\ x(t)=&
				\,\,\gamma(t),\quad t\in \lbrack -K,0), \end{aligned}\right.
		\end{equation*}
		or equivalently, for $dt\times dP$-almost all $(t,\omega )\in \lbrack -K
		,T]\times \Omega $ and all $\varphi \in V$, the following holds
		\begin{equation*}
			\left\{ \begin{aligned} \left\langle x(t),\varphi \right\rangle _{H}=&\,\,
				\left\langle \gamma(0),\varphi \right\rangle _{H}+\int_{0}^{t}\left\langle
				A(s)x(s),\varphi \right\rangle _{\ast }ds+\int_{0}^{t}\left\langle
				b(s,x_s),\varphi \right\rangle _{H}ds \\ & +\int_{0}^{t}\big< \big[
				B(s)x(s)+\sigma (s,x_s)\big]dw(s),\varphi \big> _{H},\quad t\in \lbrack 0,T], \\
				x(t)=&\,\, \gamma(t),\quad t\in \lbrack -K ,0). \end{aligned}\right.
		\end{equation*}
	\end{definition}
	
	We  have the following \emph{a priori} estimate on the solution of PSEE.
	
	\begin{theorem}
		\label{estimate} Assume conditions {(A1)}-{(A4)} hold. Suppose that $x(\cdot )$
		is a solution to PSEE~\eqref{sdee-0}. Then
		\begin{equation} \label{es-x}
			\begin{aligned}
				& \mathbb{E}\Big[\sup\limits_{t\in[0,T]}\Vert x(t)\Vert _{H}^{2}\Big]+%
				\mathbb{E}\int_{0}^{T}\Vert x(t)\Vert _{V}^{2}dt   \\
				& \leq C\Big\{\E\Big[\sup_{t\in[-K,0]}\|\gamma(t)\|^2_H\Big]+
				\mathbb{E}\int_{0}^{T}\big( \Vert b(t,0)\Vert _{H}^{2}+\Vert
				\sigma (t,0)\Vert _{\mathcal{L}_{2}^{0}}^{2}\big) dt\Big\},
			\end{aligned}
		\end{equation}
		for some constant $C>0$ depending on $\lambda ,\alpha ,K_{1}$ and $L_1$.  Moreover, if $x^{\prime
		}(\cdot )$ is a solution to \eqref{sdee-0} with $(b,\sigma,\gamma)$ replaced by $(b^{\prime },\sigma ^{\prime
		},\gamma')$, then 
		\begin{equation} \label{es-difference}
			\begin{aligned}
				& \mathbb{E}\Big[\sup\limits_{t\in[0,T]}\Vert x(t)-x^{\prime }(t)\Vert
				_{H}^{2}\Big]+\mathbb{E}\int_{0}^{T}\Vert x(t)-x^{\prime }(t)\Vert _{V}^{2}dt
				\\
				& \leq C\Big\{\E\Big[\sup_{t\in[-K,0]}\|\gamma(t)-\gamma'(t)\|^2_H\Big]+\mathbb{E}\int_{0}^{T}\Vert b(t,x_{t}^{\prime })-b^{\prime
				}(t,x_{t}^{\prime })\Vert _{H}^{2}dt \\&\hspace{4em}+\mathbb{E}\int_{0}^{T}\Vert \sigma
				(t,x_{t}^{\prime })-\sigma ^{\prime }(t,x_{t}^{\prime })\Vert _{\mathcal{L}
					_{2}^{0}}^{2}dt\Big\}. 
			\end{aligned}
		\end{equation}
	\end{theorem}
\begin{proof}
	We shall prove the estimate \eqref{es-difference},  and
	\eqref{es-x} follows from \eqref{es-difference} with $\gamma ^{\prime
	}\equiv 0,$ $b^{\prime }\equiv 0,\sigma ^{\prime }\equiv 0$. 
	
	To simplify the
	notations, we denote 
	\begin{equation*}
		\hat{x}(t)=x(t)-x^{\prime }(t),\text{ for }t\in \lbrack -K,T].
	\end{equation*}%
	It is easy to show that  $\big\{b(t,\tilde{x}_{t})\big\}_{t\in \lbrack 0,T]}\in L_{\mathbb{F}}^{2}(0,T;H)$ and $
	\big\{\sigma (t,\tilde{x}_{t})\big\}_{t\in \lbrack 0,T]}\in L_{\mathbb{F}
	}^{2}(0,T;\mathcal{L}_{2}^{0})$ for $\tilde{x}=x,x^{\prime }\in {\mathcal X}$, by the assumptions on $b$ and $\sigma$. Applying It\^{o}'s formula 
	\eqref{e:ito-lemma} to $\Vert \hat{x}(t)\Vert _{H}^{2}$ on $[0,T]$, we have 
	for $t\in \lbrack 0,T]$, 
	\begin{align*}
		& \Vert \hat{x}(t)\Vert _{H}^{2}-\Vert \hat{x}(0)\Vert
		_{H}^{2}=2\int_{0}^{t}\left\langle A(s)\hat{x}(s),\hat{x}(s)\right\rangle
		_{\ast }ds \\
		& \hspace{2em}+2\int_{0}^{t}\left\langle b(s,x_{s})-b^{\prime
		}(s,x_{s}^{\prime }),\hat{x}(s)\right\rangle _{H}ds \\
		& \hspace{2em}+2\int_{0}^{t}\left\langle \big[B(s)\hat{x}(s)+\sigma
		(s,x_{s})-\sigma ^{\prime }(s,x_{s}^{\prime })\big]dw(s),\hat{x}(s)\right\rangle
		_{H} \\
		& \hspace{2em}+\int_{0}^{t}\big\Vert B(s)\hat{x}(s)+\sigma (s,x_{s})-\sigma
		^{\prime }(s,x_{s}^{\prime })\big\Vert _{\mathcal{L}_{2}^{0}}^{2}ds.
	\end{align*}
	
	By  (A2), we get
	\begin{align*}
		\Vert \hat{x}(t)\Vert _{H}^{2}& \leq \Vert \hat{x}(0)\Vert _{H}^{2}-\alpha
		\int_{0}^{t}\Vert \hat{x}(s)\Vert _{V}^{2}ds+(\lambda +1)\int_{0}^{t}\Vert 
		\hat{x}(s)\Vert _{H}^{2}ds \\
		& \hspace{1em}+\int_{0}^{t}\Vert b(s,x_{s})-b^{\prime }(s,x_{s}^{\prime
		})\Vert _{H}^{2}ds+\int_{0}^{t}\Vert \sigma (s,x_{s})-\sigma ^{\prime
		}(s,x_{s}^{\prime })\Vert _{\mathcal{L}_{2}^{0}}^{2}ds \\
		& \hspace{1em}+2\int_{0}^{t}\big<B(s)\hat{x}(s),\sigma (s,x_{s})-\sigma
		^{\prime }(s,x_{s}^{\prime })\big>_{\mathcal{L}_{2}^{0}}ds \\
		& \hspace{1em}+2\int_{0}^{t}\left\langle \big[B(s)\hat{x}(s)+\sigma
		(s,x_{s})-\sigma ^{\prime }(s,x_{s}^{\prime })\big]dw(s),\hat{x}(s)\right\rangle
		_{H}.
	\end{align*}
	Then by (A4), \eqref{Myeq2-41}, the  triangular
	inequality, and the fact $2ab\leq pa^{2}+b^{2}/p$ for $p>0$, we get
	\begin{equation} \label{Myeq2-3}
		\begin{aligned}
			& \Vert \hat{x}(t)\Vert _{H}^{2}\leq \Vert \hat{x}(0)\Vert _{H}^{2}-\alpha
			\int_{0}^{t}\Vert \hat{x}(s)\Vert _{V}^{2}ds+(\lambda +1)\int_{0}^{t}\Vert 
			\hat{x}(s)\Vert _{H}^{2}ds   \\
			& \hspace{1em}+4L_{1}\int_{0}^{t}\sup_{r\in \lbrack -K,s]}\Vert \hat{x}(r)\Vert
			_{H}^{2}ds+2\int_{0}^{t}\Vert b(s,x_{s}^{\prime })-b^{\prime
			}(s,x_{s}^{\prime })\Vert _{H}^{2}ds \\
			& \hspace{1em}+\Big(2+\frac{4(C_{1})^{2}}{\alpha }\Big)\int_{0}^{t}\Vert
			\sigma (s,x_{s}^{\prime })-\sigma ^{\prime }(s,x_{s}^{\prime })\Vert _{%
				\mathcal{L}_{2}^{0}}^{2}ds+\frac{\alpha }{2}\int_{0}^{t}\Vert \hat{x}%
			(s)\Vert _{V}^{2}ds  \\
			& \hspace{1em}+2\int_{0}^{t}\left\langle \big[B(s)\hat{x}(s)+\sigma
			(s,x_{s})-\sigma ^{\prime }(s,x_{s}^{\prime })\big]dw(s),\hat{x}(s)\right\rangle
			_{H}  \\
			& \leq \Vert \hat{x}(0)\Vert _{H}^{2}-\frac{\alpha }{2}\int_{0}^{t}\Vert 
			\hat{x}(s)\Vert _{V}^{2}ds +C\int_{0}^{t}\sup_{r\in \lbrack -K,s]}\Vert \hat{x}(r)\Vert
			_{H}^{2}ds \\
			& \hspace{1em}+C\int_{0}^{t}\Vert b(s,x_{s}^{\prime })-b^{\prime
			}(s,x_{s}^{\prime })\Vert _{H}^{2}ds   +C\int_{0}^{t}\Vert \sigma (s,x_{s}^{\prime })-\sigma ^{\prime
			}(s,x_{s}^{\prime })\Vert _{\mathcal{L}_{2}^{0}}^{2}ds\\
			& \hspace{1em}+2\int_{0}^{t}\left\langle \big[B(s)\hat{x}(s)+\sigma (s,x_{s})-\sigma
			^{\prime }(s,x_{s}^{\prime })\big]dw(s),\hat{x}(s)\right\rangle _{H}.
		\end{aligned}
	\end{equation}
	Taking expectation on both sides of \eqref{Myeq2-3}, we  get 
	\begin{equation}\label{es:normV}
		\begin{aligned} 
			&\E\int_0^t\|\hat x(s)\|_V^2ds\leq C\E\Big\{\Vert\hat{x}(0)\Vert _{H}^{2} +\int_{0}^{t}\sup_{r\in \lbrack -K,s]}\Vert \hat{x}(r)\Vert
			_{H}^{2}ds\\&\hspace{3em} +\int_{0}^{t}\Vert b(s,x_{s}^{\prime })-b^{\prime
			}(s,x_{s}^{\prime })\Vert _{H}^{2}ds  +\int_{0}^{t}\Vert \sigma (s,x_{s}^{\prime })-\sigma ^{\prime
			}(s,x_{s}^{\prime })\Vert _{\mathcal{L}_{2}^{0}}^{2}ds\Big\}\\&\leq C\E\Big\{\sup_{t\in \lbrack -K,0]}\Vert \hat{\gamma}
			(t)\Vert _{H}^{2}+\int_{0}^{t}\sup_{r\in \lbrack 0,s]}\Vert \hat{x}
			(r)\Vert _{H}^{2}ds+\int_{0}^{t}\Vert b(s,x_{s}^{\prime
			})-b^{\prime }(s,x_{s}^{\prime })\Vert _{H}^{2}ds \\
			& \hspace{6em}+\int_{0}^{t}\Vert \sigma (s,x_{s}^{\prime })-\sigma
			^{\prime }(s,x_{s}^{\prime })\Vert _{\mathcal{L}_{2}^{0}}^{2}ds\Big\}.
		\end{aligned}
	\end{equation}
	Moreover, it follows from \eqref{e:QV} and Burkholder-Davis-Gundy inequality that 
	\begin{equation}\label{e:3}
		\begin{aligned}
			& \mathbb{E}\Big[\sup\limits_{t\in[0,T]}\int_{0}^{t}\left\langle \big[B(s)%
			\hat{x}(s)+\sigma (s,x_{s})-\sigma ^{\prime }(s,x_{s}^{\prime })\big]dw(s),\hat{x%
			}(s)\right\rangle _{H}\Big] \\
			& \leq C\mathbb{E}\Big( \int_{0}^{T}\big\Vert B(t)\hat{x}(t)+\sigma
			(t,x_{t})-\sigma ^{\prime }(t,x_{t}^{\prime })\big\Vert _{\mathcal{L}%
				_{2}^{0}}^{2}\Vert \hat{x}(t)\Vert _{H}^{2}dt\Big) ^{\frac{1}{2}} \\
			& \leq \frac{1}{4}\mathbb{E}\Big[\sup\limits_{t\in[0,T]}\Vert \hat{x}%
			(t)\Vert _{H}^{2}\Big]+{C}\mathbb{E}\int_{0}^{T}\big(\Vert B(t)\hat{x}%
			(t)\Vert _{\mathcal{L}_{2}^{0}}^{2}+\Vert \sigma (t,x_{t})-\sigma ^{\prime
			}(t,x_{t}^{\prime })\Vert _{\mathcal{L}_{2}^{0}}^{2}\big)dt\\
			& \leq \frac{1}{4}\mathbb{E}\Big[\sup\limits_{t\in[0,T]}\Vert \hat{x}%
			(t)\Vert _{H}^{2}\Big]+{C}\mathbb{E}\int_{0}^{T}\big(\Vert \hat{x}
			(t)\Vert _{V}^{2}+\Vert \sigma (t,x_{t})-\sigma ^{\prime
			}(t,x_{t}^{\prime })\Vert _{\mathcal{L}_{2}^{0}}^{2}\big)dt,
		\end{aligned}
	\end{equation}
	where the last step follows from \eqref{Myeq2-41}. 
	Then, taking supremum over $t\in \lbrack 0,\tau]$ for $\tau\in(0,T]$ and  taking expectation on both sides
	of (\ref{Myeq2-3}),  we have, in view of the estimates  \eqref{es:normV} and \eqref{e:3}, 
	\begin{align*}
		& \mathbb{E}[\sup_{t\in \lbrack 0,\tau]}\Vert \hat{x}(t)\Vert _{H}^{2}]+\mathbb{
			E}\int_{0}^{\tau}\Vert \hat{x}(s)\Vert _{V}^{2}ds \\
		& \leq C\Big\{\mathbb{E}\Big[\sup_{r\in \lbrack -K,0]}\Vert \hat{\gamma}
		(r)\Vert _{H}^{2}\Big]+\int_{0}^{\tau}\E\Big[\sup_{r\in \lbrack 0,s]}\Vert \hat{x}
		(r)\Vert _{H}^{2}\Big]ds+\mathbb{E}\int_{0}^{\tau}\Vert b(s,x_{s}^{\prime
		})-b^{\prime }(s,x_{s}^{\prime })\Vert _{H}^{2}ds \\
		& \hspace{6em}+\mathbb{E}\int_{0}^{\tau}\Vert \sigma (s,x_{s}^{\prime })-\sigma
		^{\prime }(s,x_{s}^{\prime })\Vert _{\mathcal{L}_{2}^{0}}^{2}ds\Big\}, 
	\end{align*}
	and the desired \eqref{es-difference} follows from  the Gr\"{o}nwall's inequality.

\end{proof}
Now we are ready to prove the well-posedness of PSEE \eqref{sdee-0}. 

\begin{theorem}
	\label{exu} Assuming (A1)-(A4), PSEE~\eqref{sdee-0} admits a unique
	solution in {$\mathcal X$} in the sense of Definition \ref{def:SDEE}. 
\end{theorem}

\begin{proof}
	Given  any fixed $X(\cdot)\in \mathcal X$ satisfying $X(t)=\gamma(t), t\in[-K,0]$, the following   linear SEE without delay
	\begin{equation*}
		\left\{ \begin{aligned}
			dx(t)=&\big[A(t)x(t)+b(t, X_t)\big]dt+\big[B(t)x(t)+\sigma(t,X_t)\big]dw(t),\quad t\in[0,T],
			\\ x(t)=&\gamma(t),\quad t\in[-K,0],\end{aligned}\right. 
	\end{equation*}
	has a unique solution  in {$\mathcal X$} by \cite{81k}. Thus, this  equation  defines a mapping $
	\mathbb{I}:  \mathcal X\to {\mathcal X}$  by $\mathbb{I}(X)=x$.

	For $X(\cdot),X^{\prime }(\cdot)\in {\mathcal X}$, we denote, for $t\in [-K,T]$,
	 \[
	 \mathbb{I}
	 (X^{\prime })=x^{\prime },\ \mathbb{I}(X)=x\ \text{and} \ \hat X(t)=X(t)-X'(t),\ \hat x(t)=x(t)-x'(t).
	 \]
	Obviously, $\hat{x}(\cdot)$ satisfies the following equation 
	\begin{equation*}
		\left\{ \begin{aligned} d\hat x(t)=&\big[A(t)\hat
			x(t)+b(t,X_t)-b(t,X'_t)\big]dt\\&+\big[B(t)\hat
			x(t)+\sigma(t,X_t)-\sigma(t,X'_t)\big]dw(t),\quad t\in[0,T],\\\hat x(t)=&0,\quad t\in[-K,0]. \end{aligned}%
		\right. 
	\end{equation*}%
	Then it follows from the \emph{a priori} estimate \eqref{es-difference} and (A4) that 
	\begin{align*}
		& \mathbb{E}\Big[\sup\limits_{t\in[0,T]}\Vert \hat{x}(t)\Vert _{H}^{2}%
		\Big]+\E\int_0^T\|\hat x(t)\|^2_Vdt\\&\leq C\mathbb{E}\int_{0}^{T}\big(\Vert b(t,X_{t})-b(t,X_{t}^{\prime
		})\Vert _{H}^{2}+\Vert \sigma (t,X_{t})-\sigma (t,X_{t}^{\prime })\Vert _{%
			\mathcal{L}_{2}^{0}}^{2}\big)dt \\
		& \leq 2CL_1\int_{0}^{T}\Vert \hat{X}(t)\Vert _{H}^{2}dt \\
		& \leq 2CL_1T\mathbb{E}\Big[\sup\limits_{t\in[0,T]}\Vert \hat{X}(t)\Vert
		_{H}^{2}\Big]\\&\leq 2CL_1T\Big\{\mathbb{E}\Big[\sup\limits_{t\in[0,T]}\Vert \hat{X}(t)\Vert
		_{H}^{2}\Big]+\E\int_0^T \|\hat {X}(t)\|_V^2dt\Big\},
	\end{align*}%
	where $C>0$ depends only on $\lambda ,\alpha ,K_{1}$ and $K$. Then for  $T <
	\frac{1}{2CL_1}$, $\mathbb{I}$ is a contraction  on $\mathcal X$, and hence has a unique fixed point $x(\cdot)\in \mathcal X$ which is the unique solution to
	\begin{equation*}\label{SDEE-6}
		\left\{ \begin{aligned}
			dx(t)=&\big[A(t)x(t)+b(t, x_t)\big]dt+\big[B(t)x(t)+\sigma(t,x_t)\big]dw(t), \quad t\in[0,T],
			\\ x(t)=&\gamma(t),\quad t\in[-K,0]. \end{aligned}\right. 
	\end{equation*}
	For general $T>0$, we may repeat the above procedure to obtain the well-posedness.  
\end{proof}
\begin{remark}\label{rem:A4'}
	Via similar arguments, Theorem \ref{exu} remains valid if condition {(A4)} is replaced by 
	\begin{itemize}
		\item[({A4$\hspace{0.3mm}'$})] There exists a constant $L_1>0$ such that, for each $(t,\omega )\in \lbrack 0,T]\times \Omega$, 
		\begin{align*}
			& \int_0^t\Big( \Vert b(s,x_{s})-b(s,x_{s}^{\prime })\Vert
			_{H}^{2}+\Vert \sigma (s,x_{s})-\sigma (s,x_{s}^{\prime })\Vert _{\mathcal{L}
				_{2}^{0}}^{2}\Big)ds\\
			& \leq L_1\int_{-K}^t\Vert x(s)-x^{\prime }(s)\Vert _{H}^{2}ds,
		\end{align*}
		holds for any 
		$x,x^{\prime } \in { C(-K,T;H)}$.
	\end{itemize}
\end{remark}

\begin{remark}
	Compared with the results in \cite{TGR02,2019m}, our SEE \eqref{sdee-0} contains an unbounded operator $B$ in the diffusion
	term.
\end{remark}

\subsection{Anticipated backward stochastic evolution equations}
In this subsection, we study the well-posedness of \emph{anticipated
	backward stochastic evolution equations} (ABSEEs) with a
running terminal.  It will be used to
describe the adjoint equation in the derivation of the maximum
principle.

Let $\mathcal{M}:[0,T]\times \Omega \rightarrow \mathcal{L}(V,V^{\ast
}),\,\,\,\mathcal{N}:[0,T]\times \Omega \rightarrow \mathcal{L}(\mathcal{L}
_{2}^{0},V^{\ast })$ be unbounded linear
 operators  and  $g: \lbrack
0,T] \times\Omega \times {L_{\mathbb{F}}^{2}(0,T+K;H)\times L_{\mathbb{F}}^{2}(0,T+K;%
	\mathcal{L}_{2}^{0})}\rightarrow H$ be a generator function. Let $F$ be a real-valued adapted process on $[0,T]$ with
finite variation (and hence $dF$ induces a random signed measure on $[0,T]$).  For a function $a(\cdot
): [0,T+K]\rightarrow E,$  we denote, for $t\in[0,T]$,
$$a_{t+}=\big\{a (t\vee r),\, r\in [0, T+K]\big\}.$$

We aim to study the following ABSEE 
\begin{equation}
	\left\{ \begin{aligned} p(t)=&\,\,
		\xi(T)+\int_{(t,T]}\zeta(s)dF(s)+\int_t^T\Big\{\mathcal{M}(s)p(s)+%
		\mathcal{N}(s)q(s)\\
		&+\mathbb{E}^{\mathcal{F}_{s}}\big[g(s,p_{s+},q_{s+})\big]\Big\}ds
		-\int_t^Tq(s)dw(s),\quad t\in \lbrack 0,T], \\ p(t)=&\,\,
		\xi(t),\,q(t)=\eta(t),\quad {t\in (T,T+K]}, \end{aligned}\right.
	\label{absee-0}
\end{equation}
where processes $\xi $, $\zeta $ and $\eta $ are terminal conditions acting on $
[T,T+K]$, $(0,T]$ and $(T,T+K]$, respectively. 
The term  $\int_{(t,T]} \zeta(s) dF(s)$, known as the running terminal condition, makes  ABSEE \eqref{absee-0} distinct from the classical situation, in particular when $dF$ is not absolutely continuous with  respect to the Lebesgue measure. 

 We denote
\begin{align*}	\mathscr P:=&\Big\{p \ \text{is a process on}\ [0,T+K]: p|_{[0,T]}\in L_{\mathbb{F}}^{2}(0 ,T;V) \\&\hspace{2em} \text{and}\ p|_{[T,T+K]} \in   L_{\mathbb{F}}^{2}(T,T+K;H)\Big\},
\end{align*}
with norm
\[
\Vert p\Vert_{\mathscr P}:=\Big(\big\Vert p|_{[0,T]}\big\Vert^2_{L_{\mathbb{F}}^{2}(0,T;V)}
+\big\Vert p|_{[T,T+K]}\big\Vert^2_{L_{\mathbb{F}}^{2}(T,T+K;H)}\Big)^{\frac12}.
\]

To obtain the existence and uniqueness of {the solution} to \eqref{absee-0}, we
impose the following conditions.

\begin{itemize}
	\item[({B1})] For each $(p,q)\in {L^{2}(0,T+K; H)\times L^{2}(0,T+K;%
		\mathcal{L}_{2}^{0})},$ $g(\cdot ,\cdot ,p,q)$ is a measurable function; $g(\cdot ,\cdot ,0,0)\in L_{\mathbb{F}}^{1,2}(0,T;H)$.  ${\xi \in  L^{2}_{\mathbb{F}}(T,T+K;H})$ and $\xi(T)\in L^{2}(\mathcal{F}_T;H)$, $\zeta \in
	L_{\mathbb{F},F}^{2}(0,T;H)$ and $\eta \in {L^{2}_{\mathbb{F}}(T,T+K;\mathcal{L}_{2}^{0})}
	$ with $L_{\mathbb{F}}^{1,2}(0,T;H)$ being the space of $H$-valued
	progressively measurable processes $\phi (\cdot )$ with norm 
	\begin{equation*}
		\Vert \phi \Vert _{L_{\mathbb{F}}^{1,2}(0,T;H)}=\Big( \mathbb{\mathbb{E}}%
		\Big[ \Big( \int_{0}^{T}\Vert \phi (t)\Vert _{H}{d}t\Big) ^{2}\Big]
		\Big) ^{\frac{1}{2}}.
	\end{equation*}
	
	\item[({B2})] For each $v\in V,$ $\mathcal{M}(\cdot ,\cdot )v$ and $%
	\mathcal{N}(\cdot ,\cdot )v$ are progressively measurable. There exist
	constants $\alpha >0$ and $\lambda \in \mathbb{R}$ such that for each $%
	(t,\omega )\in \lbrack 0,T]\times \Omega $, 
	\begin{equation*}
		2\left\langle \mathcal{M}(t)v,v\right\rangle _{\ast }+\Vert \mathcal{N}
		^{\ast }(t)v\Vert _{\mathcal{L}_{2}^{0}}^{2}\leq -\alpha \Vert v\Vert
		_{V}^{2}+\lambda \Vert v\Vert _{H}^{2},\ \text{for all }v\in V,
	\end{equation*}
	where $\mathcal{N}^{\ast }\in \mathcal{L}(V,\mathcal{L}_{2}^{0})$ is the
	adjoint operator of $\mathcal{N}\in \mathcal{L}(\mathcal{L}_{2}^{0},V^{\ast
	})$.
	
	\item[({B3})] There exists a constant $K_{2}>0$ such that for each $(t,\omega )\in \lbrack 0,T]\times \Omega $, 
	\begin{equation*}
		\Vert \mathcal{M}(t)v\Vert _{\ast }\leq K_{2}\Vert v\Vert _{V},\ \text{for all }v\in V.
	\end{equation*}
	
	\item[({B4})] There exists a positive constant $L_2$ such that for each $(t,\omega )\in \lbrack 0,T]\times \Omega $,
	\begin{align*}
		& \int_{t}^{T}\big\Vert g(s,p_{s+},q_{s+})-g(s,p_{s_{+}}^{\prime
		},q_{s_{+}}^{\prime })\big\Vert _{H}^{2}ds \\
		& \leq  L_2\Big\{\int_{t}^{T+K}\Vert p(s)-p^{\prime }(s)\Vert^{2} _{H}ds+\int_{t}^{T+K}\Vert
			q(s)-q^{\prime }(s)\Vert _{\mathcal{L}_{2}^{0}}^{2}ds\Big\},
	\end{align*}
	for all $(p,q),(p^{\prime },q^{\prime })\in  L^2(0,T+K;H)\times L^2(0,T+K;\mathcal{L}_{2}^{0}).$

	\item[({B5})] The total variation $|F|_{v}$ of $F$ on $[0,T]$ is 
	bounded by  a constant $K_{F}$.
\end{itemize}

Similar to \eqref{Myeq2-41}, (B2) and (B3) yield 
\begin{equation}
	\Vert \mathcal{N}(t)v\Vert_{V^\ast}\leq C_{2}\Vert v\Vert _{\mathcal{L}%
		_{2}^{0}},\ \text{for}\  v\in \mathcal{L}_{2}^{0},  \label{Myeq2-41(2)}
\end{equation}
where $C_{2}$ is a constant depending on $\lambda $ and $K_2$.

\begin{remark}
	If $\mathcal{M}$ and $\mathcal{N}$ are the adjoint operators of $A$ and $B$
	respectively which satisfy the conditions (A2)-(A3), then $\mathcal{M}$ and $%
	\mathcal{N}$ satisfy (B2)-(B3) accordingly.
\end{remark}

\begin{definition}
	\label{def:ABSEE} A process $(p(\cdot ),q(\cdot ))\in  {\mathscr P}\times L_{\mathbb{F}}^{2}(0,T+K ;\mathcal{L}_{2}^{0})$ is called a
	solution to ABSEE~\eqref{absee-0}, if for $dt\times dP$-almost all $(t,\omega )\in
	\lbrack 0,T+K ]\times \Omega$,  it holds in $V^{\ast}$ that:
	\begin{equation}\label{bsde-def1}
		\left\{ \begin{aligned} p(t)=&\,\, \xi(T)
			+\int_{(t,T]}\zeta(s)dF(s)+\int_{t}^{T}\Big\{\mathcal{M}(s)p(s)+%
			\mathcal{N}(s)q(s)\\
			&+\mathbb{E}^{\mathcal{F}_{s}}[g(s,p_{s{+}},q_{s{+}})]\Big\}ds
			-\int_{t}^{T}q(s)dw(s),\quad t\in \lbrack 0,T], \\ p(t)=&\,\, \xi(t),\
			q(t)=\eta(t),\quad t\in (T,T+K ], \end{aligned}\right.
	\end{equation}%
	or equivalently, for $dt\times dP$-almost all $(t,\omega )\in \lbrack
	0,T+K]\times \Omega $ and every $\varphi \in V$, 
	\begin{equation}
		\left\{ \begin{aligned} \left\langle p(t),\varphi \right\rangle _{H}=&\,\,
			\left\langle \xi(T) ,\varphi \right\rangle
			_{H}+\int_{(t,T]}\langle\zeta(s),\varphi\rangle_H
			dF(s)+\int_{t}^{T}\left\langle \mathcal{M}(s)p(s),\varphi \right\rangle
			_{\ast }ds\\ & +\int_{t}^{T}\left\langle \mathcal{N}(s)q(s),\varphi
			\right\rangle _{\ast }ds
			+\int_{t}^{T}\big<\mathbb{E}^{\mathcal{F}_{s}}[g(s,p_{s+},q_{s+})],\varphi
			\big>_{H}ds \\ & -\int_{t}^{T}\left\langle q(s)dw(s),\varphi \right\rangle
			_{H},\quad t\in \lbrack 0,T], \\ p(t)=&\,\, \xi(t),\ q(t)=\eta(t),\quad t\in
			(T,T+K ]. \end{aligned}\right.
	\end{equation}
\end{definition}
\begin{remark}
	If $(p,q)$   is   a
	solution of ~\eqref{absee-0}, then from Lemma \ref{Ito-lemma}, we know that $p|_{[0,T]}\in D_{\mathbb{F}}^{2}(0,{T};H)$.
\end{remark}

In parallel to Theorem \ref{estimate}, we have the following \emph{a priori} estimate for ABSEE~(\ref{absee-0}), of which the proof is more involved, due to the  nature of backward SEEs  and the presence of a running terminal term.

\begin{theorem}
	\label{es-p} Assume the assumptions (B1)-(B4) hold. Suppose $(p(\cdot ),q(\cdot
	)) $ is a solution to ABSEE~(\ref{absee-0}) associated with $(\xi,\eta,g,\zeta )$, then there exists a positive
	constant $C$ depending on $\lambda ,\alpha ,L_2$ and $K_{2}$ such that 
	\begin{equation}
		\label{EQ3-7}
		\begin{aligned}
			& \mathbb{E}\big[\sup\limits_{t\in[0,T]}\Vert p(t)\Vert _{H}^{2}\big]+
			\mathbb{E}\int_{0}^{T}\Vert q(t)\Vert _{\mathcal{L}_{2}^{0}}^{2}dt+\mathbb{E}
			\int_{0}^{T}\Vert p(t)\Vert _{V}^{2}dt \\
			& \leq C\Big\{\mathbb{E}[\Vert \xi(T) \Vert _{H}^{2}]+\E\int_T^{T+K}\big(\|\xi(t)\|^2_{ H}+\|\eta(t)\|^2_{\mathcal L_2^0}\big)dt\\
			&\hspace{2em}+\mathbb{E}
			\int_{(0,T]}\Vert \zeta (t)\Vert _{H}^{2}\Delta F(t)dF(t)+\mathbb{E}\Big(
			\int_{(0,T]}\Vert \zeta (t)\Vert _{H}d|F|_{v}(t)\Big)^{2} \\
			& \hspace{2em}+\Big(\mathbb{E}\int_{0}^{T}\Vert g(t,0,0)\Vert _{H}dt\Big)^{2}%
			\Big\}.
		\end{aligned}
	\end{equation}
	Moreover, let $(p^{\prime }(\cdot ),q^{\prime }(\cdot ))$ be a solution to %
	\eqref{absee-0} with $(\xi ^{\prime },\eta', g^{\prime },\zeta ^{\prime })$. Then
	the following estimate holds: 
	\begin{equation}\label{bes-difference}
		\begin{aligned}
			& \mathbb{E}\Big[\sup\limits_{t\in[0,T]}\Vert p(t)-p^{\prime }(t)\Vert
			_{H}^{2}\Big]+\mathbb{E}\int_{0}^{T}\Vert q(t)-q^{\prime }(t)\Vert _{
				\mathcal{L}_{2}^{0}}^{2}dt+\mathbb{E}\int_{0}^{T}\Vert p(t)-p^{\prime
			}(t)\Vert _{V}^{2}dt\\&\leq C\Big\{\mathbb{E}[\Vert \xi(T) -\xi ^{\prime }(T)\Vert
			_{H}^{2}]+\E\int_{T}^{T+K}\big(\Vert {\xi}(t)-\xi'(t)\Vert^{2} _{H}+\Vert {\eta}(t)-\eta'(t)\Vert _{\mathcal{L}
				_{2}^{0}}^{2}\big)dt  \\
			&\hspace{2em}+\mathbb{E}\int_{(0,T]}\Vert \zeta (t)-\zeta ^{\prime
			}(t)\Vert _{H}^{2}\Delta F(t)dF(t)+\mathbb{E}\Big(\int_{(0,T]}\Vert \zeta
			(t)-\zeta ^{\prime }(t)\Vert _{H}d|F|_{v}(t)\Big)^{2} \\
			&\hspace{2em}+\mathbb{E}\Big(\int_{0}^{T}\big\Vert g(t,p_{t{+}}^{\prime
			},q_{t{+}}^{\prime })-g^{\prime }(t,p_{t{+}}^{\prime },q_{t{+}}^{\prime
			})\big\Vert _{H}dt\Big)^{2}\Big\},  
		\end{aligned}
	\end{equation}
	where $C$ is a positive constant depending on $\lambda ,\alpha ,L_2$ and $%
	K_{2}.$
\end{theorem}

\begin{proof}
	It suffices to prove \eqref{bes-difference} which implies (\ref{EQ3-7}).
	Set 
	\begin{equation*}
		\hat{p}(t)=p(t)-p^{\prime }(t),\ \hat{q}(t)=q(t)-q^{\prime }(t),\ \hat{\zeta}%
		(t)=\zeta (t)-\zeta ^{\prime }(t),\text{ for } t\in[0,T],
	\end{equation*}
	and
	\begin{equation*}
		\hat{\xi}(t)=\xi(t) -\xi ^{\prime }(t),\  \ \text{for}\ t\in[T,T+K]; \  \hat{\eta}(t)=\eta(t) -\eta ^{\prime }(t),  \ \text{for}\ t\in(T,T+K].
	\end{equation*}
	We first note that, from (\ref{bsde-def1}),   assumptions (B3)-(B5), (\ref{Myeq2-41(2)}) and Burkholder-Davis-Gundy inequality, \begin{equation}\label{sup-finite}\mathbb{E}\Big[\sup\limits_{t\in[0,T]}\Vert \hat{p}(t)\Vert _{H}^{2}
		\Big]<\infty.\end{equation}
	Applying It\^{o}'s formula \eqref{e:ito-lemma} to $\Vert \hat{p}(t)\Vert
	_{H}^{2}$ on $[t,T],$ we have 
	\begin{align*}
		& \Vert \hat{p}(t)\Vert _{H}^{2}+\int_{t}^{T}\Vert \hat{q}(s)\Vert _{%
			\mathcal{L}_{2}^{0}}^{2}ds=\Vert \hat{\xi}(T)\Vert _{H}^{2}+2\int_{t}^{T}\Big\{%
		\big<\mathcal{M}(s)\hat{p}(s),\hat{p}(s)\big>_{\ast }+\left\langle \mathcal{N%
		}(s)\hat{q}(s),\hat{p}(s)\right\rangle _{\ast } \\
		&\hspace{3em} +\big<\mathbb{E}^{\mathcal{F}_{s}}\big[g(s,p_{s+},q_{s+})-g^{\prime
		}(s,p_{s+}^{\prime },q_{s+}^{\prime })\big],\hat{p}(s)\big>_{H}\Big\}ds \\
		&\hspace{3em}  +2\int_{(t,T]}\big<\hat{p}(s),\hat{\zeta}(s)\big>_{H}dF(s)+\int_{(t,T]}%
		\Vert \hat{\zeta}(s)\Vert _{H}^{2}\Delta F(s)dF(s) \\
		& \hspace{3em} -2\int_{t}^{T}\left\langle \hat{p}(s),\hat{q}(s)dw(s)\right\rangle _{H}.
	\end{align*}
	By conditions  (B1)-(B3), we obtain that, for some positive constant $%
	\varepsilon $ to be determined, 
	\begin{equation}\label{EQ-5} 
		\begin{aligned}
			& \Vert \hat{p}(t)\Vert _{H}^{2}+\int_{t}^{T}\Vert \hat{q}(s)\Vert _{
				\mathcal{L}_{2}^{0}}^{2}ds   \\
			& \leq  2\int_{(t,T]}\big<\hat{p}(s),\hat{\zeta}%
			(s)\big>_{H}dF(s)+\int_{(t,T]}\Vert \hat{\zeta}(s)\Vert _{H}^{2}\Delta
			F(s)dF(s)  \\
			& \hspace{1em}+\int_{t}^{T}\Big\{-2\varepsilon \big<\mathcal{M}(s)\hat{p}(s),%
			\hat{p}(s)\big>_{\ast }+2(1+\varepsilon )\big<\mathcal{M}(s)\hat{p}(s),\hat{p%
			}(s)\big>_{\ast }\\&\hspace{4em}+(1+\varepsilon )\Vert \mathcal{N}^{\ast }(s)\hat{p}%
			(s)\Vert _{\mathcal{L}_{2}^{0}}^{2}+\tfrac{1}{1+\varepsilon }\Vert \hat{q}%
			(s)\Vert _{\mathcal{L}_{2}^{0}}^{2}   \\
			& \hspace{4em}+2\big|\big<\mathbb{E}^{\mathcal{F}%
				_{s}}\big[g(s,p_{s+},q_{s+})-g(s,p_{s+}^{\prime },q_{s+}^{\prime })\big],\hat{p}(s)%
			\big>_{H}\big|   \\
			& \hspace{4em}+2\big|\big<\mathbb{E}^{\mathcal{F}_{s}}\big[g(s,p_{s+}^{\prime
			},q_{s+}^{\prime })-g^{\prime }(s,p_{s+}^{\prime },q_{s+}^{\prime })\big],\hat{p}%
			(s)\big>_{H}\big|\Big\}ds   \\
			& \hspace{1em}+\Vert\hat{\xi}(T)\Vert _{H}^{2}-2\int_{t}^{T}\left\langle \hat{q}(s)dw(s),\hat{p}
			(s)\right\rangle _{H}\\
			& \leq 2\sup_{s\in[t, T]}\Vert \hat{p}
			(s)\Vert _{H}\int_{(t,T]}\Vert \hat{\zeta}(s)\Vert
			_{H}d|F|_{v}(s)+\int_{(t,T]}\Vert \hat{\zeta}(s)\Vert _{H}^{2}\Delta
			F(s)dF(s) \\
			& \hspace{1em}+\int_{t}^{T}\Big\{2\varepsilon K_{2}\Vert \hat{p}(s)\Vert
			_{V}^{2}+(1+\varepsilon )(-\alpha \Vert \hat{p}(s)\Vert _{V}^{2}+\lambda
			\Vert \hat{p}(s)\Vert _{H}^{2})+\tfrac{1}{1+\varepsilon }\Vert \hat{q}
			(s)\Vert _{\mathcal{L}_{2}^{0}}^{2}   \\
			& \hspace{4em}+\tfrac{4L_2}{\varepsilon }\Vert \hat{p}(s)\Vert _{H}^{2}+%
			\tfrac{\varepsilon }{4L_2}\mathbb{E}^{\mathcal{F}_{s}}\big[\big\Vert 
			g(s,p_{s+},q_{s+}) -g(s,p_{s+}^{\prime },q_{s+}^{\prime })\big\Vert_{H}^{2}\big]
			\Big\}ds \\
			& \hspace{1em}+2\sup_{s\in[t,T]}\Vert \hat{p}(s)\Vert _{H}\int_{t}^{T}
			\big \Vert\mathbb{E}^{\mathcal{F}_{s}}\big[g(s,p_{s+}^{\prime },q_{s+}^{\prime })
			-g^{\prime }(s,p_{s+}^{\prime },q_{s+}^{\prime })\big]\big\Vert%
			_{H}ds\\
			& \hspace{1em}+\Vert \hat{\xi}(T)\Vert _{H}^{2}-2\int_{t}^{T}\left\langle \hat{q}(s)dw(s),\hat{p}(s)\right\rangle
			_{H}.  
		\end{aligned}
	\end{equation}
	Taking expectation on both sides and using the condition (B4), we get 
	\begin{align*}
		& \mathbb{E}[\Vert \hat{p}(t)\Vert _{H}^{2}]+\mathbb{E}\int_{t}^{T}\Vert 
		\hat{q}(s)\Vert _{\mathcal{L}_{2}^{0}}^{2}ds \\
		& \leq \mathbb{E}\big[\Vert \hat{\xi}(T)\Vert _{H}^{2}\big]+\frac{\varepsilon}{4}\E\int_{T}^{T+K}\big\{\Vert \hat{\xi}(s)\Vert^2_{ H}+\Vert \hat{\eta}(s)\Vert
		_{\mathcal{L}_{2}^{0}}^{2} \big\}ds\\&\hspace{1em}+2\mathbb{E}\Big[\sup_{s\in[t,T]}\Vert \hat{p}(s)\Vert _{H}\int_{(t,T]}\Vert \hat{\zeta}(s)\Vert
		_{H}d|F|_{v}(s)\Big]+\mathbb{E}\int_{(t,T]}\Vert \hat{\zeta}(s)\Vert
		_{H}^{2}\Delta F(s)dF(s) \\
		& \hspace{1em}+\mathbb{E}\int_{t}^{T}\Big\{2\varepsilon K_{2}\Vert \hat{p}
		(s)\Vert _{V}^{2}+(1+\varepsilon )(-\alpha \Vert \hat{p}(s)\Vert
		_{V}^{2}+\lambda \Vert \hat{p}(s)\Vert _{H}^{2})+\tfrac{1}{1+\varepsilon }
		\Vert \hat{q}(s)\Vert _{\mathcal{L}_{2}^{0}}^{2} \\
		& \hspace{5em}+\tfrac{4L_2}{\varepsilon }\Vert \hat{p}(s)\Vert _{H}^{2}+\tfrac{
			\varepsilon }{4L_2}\times L_2\big(\Vert \hat{p}(s)\Vert^{2} _{ H}+\Vert \hat{q}(s)\Vert
		_{\mathcal{L}_{2}^{0}}^{2} \big)\Big\}ds \\
		& \hspace{1em}+2\mathbb{E}\Big[\sup_{s\in[t,T]}\Vert \hat{p}(s)\Vert
		_{H}\int_{t}^{T}\big\Vert\mathbb{E}^{\mathcal{F}_{s}}\big[g(s,p_{s+}^{\prime
		},q_{s+}^{\prime }) -g^{\prime }(s,p_{s+}^{\prime },q_{s+}^{\prime })\big]\big\Vert%
		_{H}ds\Big] \\
		& \leq \mathbb{E}\big[\Vert \hat{\xi}(T)\Vert _{H}^{2}\big]+\frac{\varepsilon}{4}\E\int_{T}^{T+K}\big(\Vert \hat{\xi}(s)\Vert^2_{ H}+\Vert \hat{\eta}(s)\Vert
		_{\mathcal{L}_{2}^{0}}^{2} \big)ds\\&\hspace{1em}+2\mathbb{E}\big[\sup_{s\in[t,T]}\Vert \hat{p}(s)\Vert _{H}\int_{(t,T]}\Vert \hat{\zeta}(s)\Vert
		_{H}d|F|_{v}(s)\big]+\mathbb{E}\int_{(t,T]}\Vert \hat{\zeta}(s)\Vert
		_{H}^{2}\Delta F(s)dF(s) \\
		& \hspace{1em}+\mathbb{E}\int_{t}^{T}\Big\{{ \big(2\varepsilon K_{2}-\alpha (1+\varepsilon )\big)}\Vert \hat{p}(s)\Vert _{V}^{2}+%
		\big(\tfrac{1}{1+\varepsilon }+\tfrac{\varepsilon }{4}\big)\Vert \hat{q}%
		(s)\Vert _{\mathcal{L}_{2}^{0}}^{2} \\
		& \hspace{6em}+\big(\tfrac{4L_2}{\varepsilon }+{ \tfrac{%
			\varepsilon }{4}}+\lambda (1+\varepsilon )\big)%
		\Vert \hat{p}(s)\Vert _{H}^{2}\Big\}ds \\
		& \hspace{1em}+2\mathbb{E}\Big[\sup_{s\in[t,T]}\Vert \hat{p}(s)\Vert
		_{H}\int_{t}^{T}\big\Vert \mathbb{E}^{\mathcal{F}_{s}}\big[g(s,p_{s+}^{\prime
		},q_{s+}^{\prime }) -g^{\prime }(s,p_{s+}^{\prime },q_{s+}^{\prime })\big]\big\Vert _{H}ds%
		\Big].
	\end{align*}%
	Choosing $\varepsilon$ small enough such that
	\[
{	2\varepsilon  K_{2}-\alpha(1+\varepsilon)}<0\ \ \text{and
	}\ \ \frac{1}{1+\varepsilon}+\frac{\varepsilon}{4}=\frac{4+\varepsilon
		+\varepsilon^{2}}{4+4\varepsilon}<1,
	\]
	we can get 
	\begin{align*}
		& \mathbb{E}[\Vert \hat{p}(t)\Vert _{H}^{2}]+\mathbb{E}\int_{t}^{T}\Vert 
		\hat{q}(s)\Vert _{\mathcal{L}_{2}^{0}}^{2}ds+\mathbb{E}\int_{t}^{T}\Vert 
		\hat{p}(s)\Vert _{V}^{2}ds \\
		& \leq C\mathbb{E}\Big\{\Vert \hat{\xi}(T)\Vert _{H}^{2}+\int_{t}^{T}\Vert \hat{p}(s)\Vert _{H}^{2}ds+\int_{T}^{T+K}\big(\Vert \hat{\xi}(s)\Vert^{2}_{ H}+\Vert \hat{\eta}(s)\Vert
		_{\mathcal{L}_{2}^{0}}^{2} \big)ds\\&\hspace{5em}+\sup_{s\in[t,T]}\Vert \hat{p}(s)\Vert
		_{H}\int_{(t,T]}\Vert \hat{\zeta}(s)\Vert _{H}d|F|_{v}(s)+\int_{(t,T]}\Vert 
		\hat{\zeta}(s)\Vert _{H}^{2}\Delta F(s)dF(s) \\
		& \hspace{5em}+\sup_{s\in[t,T]}\Vert \hat{p}(s)\Vert _{H}\int_{t}^{T}%
		\big\Vert\mathbb{E}^{\mathcal{F}_{s}}\big[g(s,p_{s+}^{\prime },q_{s+}^{\prime })
		-g^{\prime }(s,p_{s+}^{\prime },q_{s+}^{\prime })\big]\big\Vert%
		_{H}ds\Big\},
	\end{align*}%
	where $C$ is a positive constant depending on $\lambda ,\alpha ,L_2,K_{2}$.
	Applying Gr\"{o}nwall's inequality to $\mathbb{E}[\Vert \hat{p}(t)\Vert
	_{H}^{2}]$ yields that, for some undetermined $a>0$, 
	\begin{equation}\label{supout}
		\begin{aligned}
			& \mathbb{E}[\Vert \hat{p}(t)\Vert _{H}^{2}]+\mathbb{E}\int_{t}^{T}\Vert 
			\hat{q}(s)\Vert _{\mathcal{L}_{2}^{0}}^{2}ds+\mathbb{E}\int_{t}^{T}\Vert 
			\hat{p}(s)\Vert _{V}^{2}ds   \\
			& \leq C\mathbb{E}\Big\{\Vert \hat{\xi}(T)\Vert _{H}^{2}+\int_{T}^{T+K}\big\{\Vert \hat{\xi}(s)\Vert^{2} _{ H}+\Vert \hat{\eta}(s)\Vert
			_{\mathcal{L}_{2}^{0}}^{2} \big\}ds\\&\hspace{3em}+\sup_{s\in[t,T]}\Vert \hat{p}(s)\Vert _{H}\int_{(t,T]}\Vert \hat{\zeta}(s)\Vert
			_{H}d|F|_{v}(s)+\int_{(t,T]}\Vert \hat{\zeta}(s)\Vert _{H}^{2}\Delta
			F(s)dF(s)   \\
			& \hspace{3em}+\sup_{s\in[t,T]}\Vert \hat{p}(s)\Vert _{H}\int_{t}^{T}
			\big \Vert\mathbb{E}^{\mathcal{F}_{s}}\big[g(s,p_{s+}^{\prime },q_{s+}^{\prime })
			-g^{\prime }(s,p_{s+}^{\prime },q_{s+}^{\prime })\big]\big\Vert
			_{H}ds\Big\} \\
			& \leq C\mathbb{E}\Big\{\Vert \hat{\xi}(T)\Vert _{H}^{2}+\int_{T}^{T+K}\big\{\Vert \hat{\xi}(s)\Vert^{2} _{ H}+\Vert \hat{\eta}(s)\Vert
			_{\mathcal{L}_{2}^{0}}^{2} \big\}ds+a\sup_{t\in[0,T]}\Vert \hat{p}(t)\Vert _{H}^{2}  \\&\hspace{4em}+\int_{(0,T]}\Vert 
			\hat{\zeta}(s)\Vert _{H}^{2}\Delta F(s)dF(s)+\frac{1}{a}\Big(%
			\int_{(0,T]}\Vert \hat{\zeta}(s)\Vert _{H}d|F|_{v}(s)\Big)^{2} \\
			& \hspace{4em}+\frac{1}{a}\Big(\int_{0}^{T}\big\Vert g(t,p_{s+}^{\prime
			},q_{s+}^{\prime })-g^{\prime }(t,p_{s+}^{\prime },q_{s+}^{\prime })\big\Vert
			_{H}ds\Big)^{2}\Big\},  
		\end{aligned}
	\end{equation}
	with $C$ being a constant independent of $a$ and may vary from line to
	line.
	
	On the other hand, by \eqref{e:QV} and Burkholder-Davis-Gundy inequality, we
	have for some positive constant $D$, 
	\begin{equation}
		\begin{split}
			&\mathbb{E}\Big[\sup\limits_{t\in[0,T]}\Big\vert\int_{t}^{T}\left\langle 
			\hat{q}(s)dw(s),\hat{p}(s)\right\rangle _{H}\Big\vert\Big]\\& \leq D\mathbb{E}
			\Big(\int_{0}^{T}\Vert \hat{q}(t)\Vert _{\mathcal{L}_{2}^{0}}^{2}\Vert \hat{p%
			}(t)\Vert _{H}^{2}dt\Big)^{\frac{1}{2}} \\
			& \leq D\mathbb{E}\Big[\sup\limits_{t\in[0,T]}\Vert \hat{p}(t)\Vert _{H}%
			\Big(\int_{0}^{T}\Vert \hat{q}(t)\Vert _{\mathcal{L}_{2}^{0}}^{2}dt\Big)^{%
				\frac{1}{2}}\Big] \\
			& \leq \frac{1}{8}\mathbb{E}\big[\sup\limits_{t\in[0,T]}\Vert \hat{p}%
			(t)\Vert _{H}^{2}\big]+2D^{2}\mathbb{E}\int_{0}^{T}\Vert \hat{q}(t)\Vert _{%
				\mathcal{L}_{2}^{0}}^{2}dt.
		\end{split}
		\label{EQ-6}
	\end{equation}%
	Then taking supremum over $t\in \lbrack 0,T]$ on both sides of (\ref{EQ-5})
	(for any fixed $\varepsilon >0$), we get 
	\begin{align*}
		& \sup\limits_{t\in[0,T]}\Vert \hat{p}(t)\Vert _{H}^{2}\leq C\Big\{
		\Vert \hat{\xi}(T)\Vert _{H}^{2}+\int_{T}^{T+K}\big(\Vert \hat{\xi}(s)\Vert^{2} _{ H}+\Vert \hat{\eta}(s)\Vert _{\mathcal{L}%
			_{2}^{0}}^{2}\big)ds\\&\hspace{1em}+\Big(\int_{(0,T]}\Vert \hat{\zeta}(s)\Vert
		_{H}d|F|_{v}(s)\Big)^{2}+\int_{(0,T]}\Vert \hat{\zeta}(s)\Vert
		_{H}^{2}\Delta F(s)dF(s) \\
		& \hspace{1em}+\int_{0}^{T}\Vert \hat{p}(s)\Vert
		_{H}^{2}ds+\int_{0}^{T}\Vert \hat{q}(s)\Vert _{\mathcal{L}%
			_{2}^{0}}^{2}ds+\int_{0}^{T}\Vert \hat{p}(s)\Vert _{V}^{2}ds \\
		& \hspace{1em}+\Big(\int_{0}^{T}\big\Vert \mathbb{E}^{\mathcal{F}%
			_{s}}\big[g(s,p_{s+}^{\prime },q_{s+}^{\prime })-g^{\prime }(s,p_{s+}^{\prime
		},q_{s+}^{\prime })\big]\big\Vert _{H}ds\Big)^{2}\Big\} \\
		& \hspace{1em}+\frac{1}{4}\sup_{t\in[0,T]}\Vert \hat{p}(t)\Vert
		_{H}^{2}+2\sup\limits_{t\in[0,T]}\Big|\int_{t}^{T}\left\langle \hat{p}%
		(s),\hat{q}(s)dw(s)\right\rangle _{H}\Big|.
	\end{align*}%
	Taking expectation on both sides, we then obtain by (\ref{EQ-6}) and (\ref{supout}) that 
	\begin{equation}\label{supin}
		\begin{aligned} 
			& \mathbb{E}\Big[\sup\limits_{t\in[0,T]}\Vert \hat{p}(t)\Vert _{H}^{2}
			\Big]\leq (\frac{1}{2}+Ca)\mathbb{E}\big[\sup\limits_{t\in[0,T]}\Vert \hat{p}(t)\Vert _{H}^{2}\big]   \\
			& \hspace{2em}+C\E\Big\{\big[\Vert \hat{\xi}(T)\Vert _{H}^{2}\big]+\int_{T}^{T+K}\big(\Vert \hat{\xi}(s)\Vert^{2} _{ H}+\Vert \hat{\eta}(s)\Vert _{\mathcal{L}%
				_{2}^{0}}^{2}\big)ds\\&\hspace{3em}+\int_{(0,T]}\Vert \hat{
				\zeta}(s)\Vert _{H}^{2}\Delta F(s)dF(s)+(1+\frac{1}{a})\Big(
			\int_{(0,T]}\Vert \hat{\zeta}(s)\Vert _{H}d|F|_{v}(s)\Big)^{2} \\
			& \hspace{3em}+(1+\frac{1}{a})\Big(\int_{0}^{T}\big\Vert g(s,p_{s+}^{\prime
			},q_{s+}^{\prime })-g^{\prime }(s,p_{s+}^{\prime },q_{s+}^{\prime })\big\Vert
			_{H}ds\Big)^{2}\Big\}. 
		\end{aligned}
	\end{equation}
	In view of (\ref{sup-finite}), we get by choosing sufficiently small $a$ that
	\begin{align}
		& \mathbb{E}\Big[\sup\limits_{t\in[0,T]}\Vert \hat{p}(t)\Vert _{H}^{2}%
		\Big]   \leq C\E\Big\{\big[\Vert \hat{\xi}(T)\Vert _{H}^{2}\big]+\int_{T}^{T+K}\big(\Vert \hat{\xi}(s)\Vert^{2} _{ H}+\Vert \hat{\eta}(s)\Vert _{\mathcal{L}%
			_{2}^{0}}^{2}\big)ds\notag\\&\hspace{3em}+\int_{(0,T]}\Vert 
		\hat{\zeta}(s)\Vert _{H}^{2}\Delta F(s)dF(s)+\Big(\int_{(t,T]}\Vert \hat{%
			\zeta}(s)\Vert _{H}d|F|_{v}(s)\Big)^{2} \\
		& \hspace{3em}+\Big(\int_{0}^{T}\big\Vert g(s,p_{s+}^{\prime },q_{s+}^{\prime
		})-g^{\prime }(s,p_{s+}^{\prime },q_{s+}^{\prime })\big\Vert _{H}ds\Big)%
		^{2}\Big\}.  \notag
	\end{align}%
	This together with (\ref{supout}) yields the desired estimate 
	\eqref{bes-difference}. \end{proof}
	
	Now we are ready to prove the well-posedness  for ABSEE (\ref{absee-0}).
	
	\begin{theorem}
		\label{ABSEE-thm} Assuming (B1)-(B5),  ABSEE (\ref{absee-0}) admits a unique solution in $  { \mathscr P}\times L_{\mathbb{F
		}}^{2}(0,T+K; \mathcal{L}_{2}^{0})$ in the sense of Definition \ref{def:ABSEE}. 
	\end{theorem}
	
	\begin{proof}
		The uniqueness follows directly from (\ref{bes-difference}) in Theorem \ref%
		{es-p}. The proof of the existence is divided into the following three steps.
		
		\textit{Step 1. The case $\zeta\equiv 0$.} We shall make use of the so-called
		continuation method (see, e.g.,  \cite{pw99}). For any $\mu \in \lbrack 0,1]$ and $
		f_{0}(\cdot )\in L_{\mathbb{F}}^{1,2}(0,T;H)$, we consider the ABSEE 
		\begin{equation}  \label{absee2}
			\left\{ \begin{aligned} -dp(t)=&\,\,
				\Big\{\mathcal{M}(t)p(t)+\mathcal{N}(t)q(t)+\mu
				\mathbb{E}^{\mathcal{F}_{t}}\big[g(t,p_{t+},q_{t+})\big]+f_{0}(t)\Big\}dt \\
				& \quad -q(t)dw(t),\quad t\in \lbrack 0,T], \\ p(t)=&\xi(t),\
				q(t)=\eta(t),\quad t\in [T,T+K ]. \end{aligned}\right.
		\end{equation}%
		In the following, we shall prove the well-posedness of \eqref{absee2},
		which implies the desired result by setting $\mu =1$ and $f_{0}(\cdot )=0.$
		
		When $\mu =0$, ABSEE \eqref{absee2} is a linear equation, and by a standard argument (see, e.g.,  \cite[Proposition 3.2]{10dm}) one can show that \eqref{absee2} has a unique solution for any  $f_{0}(\cdot
		)\in L_{\mathbb{F}}^{1,2}(0,T;H)$.   This well-posedness result can be
		extended to all $\mu \in \lbrack 0,1]$ as follows.
		
		Suppose that equation \eqref{absee2} admits a unique solution for all $%
		f_{0}(\cdot )\in L_{\mathbb{F}}^{1,2}(0,T;H)$ and some fixed $\mu _{0}\in[0,1)$. Then, for an arbitrary fixed $f_{0}(\cdot )\in L_{\mathbb{F}%
		}^{1,2}(0,T;H)$, any given $(P(\cdot ),Q(\cdot ))\in {\mathscr P}\times L_{\mathbb{F}}^{2}(0,T+K;\mathcal{L}%
		_{2}^{0})$ with $P(t)=\xi(t)$ and $Q(t)=\eta(t)$ for $t\in  [T,T+K]$, and some $\mu \in \lbrack 0,1]$ to be determined, the following
		ABSEE 
		\begin{equation}  \label{absee3}
			\left\{ \begin{aligned} -dp(t)=&\,\,
				\Big\{\mathcal{M}(t)p(t)+\mathcal{N}(t)q(t)+\mu
				_{0}\mathbb{E}^{\mathcal{F}_{t}}\big[g(t,p_{t+},q_{t+})\big]+f_{0}(t) \\ &
				\qquad +(\mu -\mu
				_{0})\mathbb{E}^{\mathcal{F}_{t}}\big[g(t,P_{t+},Q_{t+})\big]\Big\}dt \\ &
				\quad -q(t)dw(t),\quad t\in \lbrack 0,T], \\ 
				p(t)=&\xi(t),\,\,q(t)=\eta(t),\quad t\in [T,T+K], \end{aligned} \right.
		\end{equation}
		admits a unique solution $(p(\cdot ),q(\cdot ))\in {\mathscr P}\times L_{\mathbb{F}}^{2}(0,T+K;\mathcal{L}%
		_{2}^{0})$. By this, we can define the solution mapping $I:{\mathscr P}\times L_{\mathbb{F}}^{2}(0,T+K;\mathcal{L}%
		_{2}^{0})\to {\mathscr P}\times L_{\mathbb{F}%
		}^{2}(0,T+K ;\mathcal{L}_{2}^{0})$ by 
		\begin{equation*}
			(P,Q)\mapsto I(P,Q):=(p,q).
		\end{equation*}%
		Given $(P_{1}(\cdot ),Q_{1}(\cdot )),(P_{2}(\cdot ),Q_{2}(\cdot ))\in {\mathscr P}\times L_{\mathbb{F}}^{2}(0,T+K;\mathcal{L}_{2}^{0})$, it follows from Theorem~\ref{es-p} that
		\begin{align*}
			& \mathbb{E}\Big[\int_{0}^{T}\Vert p_{1}(t)-p_{2}(t)\Vert
			_{V}^{2}dt+\int_{0}^{T}\Vert q_{1}(t)-q_{2}(t)\Vert _{\mathcal{L}%
				_{2}^{0}}^{2}dt\Big] \\
			& \leq C|\mu -\mu _{0}|^{2}\mathbb{E}\Big[\int_{0}^{T}\Vert
			P_{1}(t)-P_{2}(t)\Vert _{V}^{2}dt+\int_{0}^{T}\Vert Q_{1}(t)-Q_{2}(t)\Vert _{%
				\mathcal{L}_{2}^{0}}^{2}dt\Big],
		\end{align*}%
		where $C$ is a positive constant independent of $\mu$. Thus, for $\mu \in
		\lbrack \mu _{0}-\frac{1}{\sqrt{2C}},\mu _{0}+\frac{1}{\sqrt{2C}}]$, the
		solution mapping $I$ is a contraction on ${\mathscr P}\times L_{\mathbb{F}}^{2}(0,T+K;\mathcal{L}_{2}^{0})$, which
		implies the well-posedness of (\ref{absee3}). So starting with $\mu_0=0$ and
		repeating the above procedure, we can prove that there exists a unique
		solution to (\ref{absee2}) for all $\mu\in \lbrack 0,1]$.
		
		\textit{Step 2. The case of $\zeta$ taking values in $V$.}  In this step, we shall use the technique of  solution translation to remove the running terminal condition. More precisely,   denote 
		\begin{equation*}
			\alpha(t)=\int_{(0,t]}\zeta(s)dF(s),\  t\in [0,T]
		\end{equation*}
		and 
		\begin{equation*}
			\bar p(t)=p(t)+\alpha(t),\ t\in [0,T]\ \text{and} \ { \bar p(t)=\xi(t)},\quad t\in
			(T,T+K].
		\end{equation*}
		Then we can rewrite \eqref{absee-0} as 
		\begin{equation*}
			\left\{ \begin{aligned} \bar{p}(t)=&\,\, \xi(T)+
				\alpha(T)+\int_t^T\Big\{\mathcal{M}(s)\bar{p}(s)-\mathcal{M}(s)\alpha(s)+%
				\mathcal{N}(s)q(s)\\
				&+\mathbb{E}^{\mathcal{F}_{s}}\big[g(s,\bar{p}_{s+}-\alpha_{s+},q_{s+})\big]%
				\Big\}ds \\ & -\int_t^Tq(s)dw(s),\quad t\in \lbrack 0,T], \\
				\bar{p}(t)=&\xi(t),\,q(t)=\eta(t),\quad t\in (T,T+K ]. \end{aligned}\right.
		\end{equation*}
		By Step 1, we know that the above equation admits a unique solution $(\bar{p}%
		(\cdot),q(\cdot))\in {\mathscr P}\times L_{\mathbb{F}%
		}^{2}(0,T+K;\mathcal{L}_{2}^{0} )$. Then it is easy to check that $(%
		\bar{p}(\cdot)-\alpha(\cdot),q(\cdot))\in {\mathscr P}
		\times L_{\mathbb{F}}^{2}(0,T+K;\mathcal{L}_{2}^{0})$ is a
		solution to (\ref{absee-0}).
		
		\textit{Step 3. The case of $\zeta$ taking values in $H$.} Consider the
		following approximation equations, for $n\ge 1$, 
		\begin{equation}
			\left\{ \begin{aligned} p^n(t)=&\,\,
				\xi(T)+\int_{(t,T]}\zeta^n(s)dF(s)+\int_t^T\Big\{\mathcal{M}(s)p^n(s)+%
				\mathcal{N}(s)q^n(s)\\ &+\mathbb{E}^{\mathcal{F}_{s}}\big[g(s,p^n_{s
					{+}},q^n_{s {+}})\big]\Big\}ds\\ & -\int_t^Tq^n(s)dw(s),\quad t\in \lbrack
				0,T], \\ p^n(t)=&\,\, \xi(t),\,q^n(t)=\eta(t),\quad t\in (T,T+K], \end{aligned}%
			\right.  \label{absee-01}
		\end{equation}
		where $\zeta^{n}$ belongs to $L_{\mathbb{F},F}^{2}(0,T;V)$ and converges to $\zeta\in L_{\mathbb{F},F}^{2}(0,T;H)$, as $%
		n $ goes to infinity. By Step~2, for each $n$, ABSEE (\ref{absee-01}) has a
		unique solution $(p^{n},q^{n})\in {\mathscr P}\times L_{%
			\mathbb{F}}^{2}(0,T+K;\mathcal{L}_{2}^{0})$. Using (\ref
		{bes-difference}) in Theorem~\ref{es-p}, we have 
		\begin{align*}
			& \mathbb{E}\big[\sup\limits_{t\in[0,T]}\Vert p^{n}(t)-p^{m}(t)\Vert
			_{H}^{2}\big]+\mathbb{E}\int_{0}^{T}\Vert q^{n}(t)-q^{m}(t)\Vert _{\mathcal{L
				}_{2}^{0}}^{2}dt+\mathbb{E}\int_{0}^{T}\Vert p^{n}(t)-p^{m}(t)\Vert_{V}^{2}dt
			\\
			& \leq C\Big\{\mathbb{E}\int_{(0,T]}\Vert\zeta^{n}(t)-\zeta^{m}(t)\Vert _{
				H}^{2}\Delta F(t)dF(t)+\mathbb{E}\Big(\int_{(0,T]}\Vert\zeta^{n}(t)-%
			\zeta^{m}(t)\Vert_{H}d|F|_v(t)\Big)^{2}\Big\} \\
			& \leq CK_{F}\mathbb{E}\int_{(0,T]}\Vert\zeta^{n}(t)-\zeta^{m}(t)\Vert
			_{H}^{2}d|F|_v(t),
		\end{align*}
	 	where the constant $K_F$ is from  assumption (B5). Hence, $p^{n}$ is a Cauchy
		sequence in ${\mathscr P}$ 
		with limit denoted by $p$, and $q^{n}$ is Cauchy sequence in $L_{\mathbb{F}%
		}^{2}(0,{T+K};\mathcal{L}_{2}^{0})$ with limit denoted by $q.$
		
		Finally, we deduce that $(p,q)$ satisfies (\ref{absee-0}) by combining the
		following estimates: for each $t\in [0,T]$, as $n\to \infty$, 
		\begin{align*}
			& \mathbb{E}\Big\Vert\int_{(t,T]}\big(\zeta^{n}(s)-\zeta(s)\big)dF(s)
			\Big\Vert _{H}^{2} \leq K_{F}\mathbb{E}\int_{(t,T]}\Vert\zeta^{n}(s)-\zeta
			(s)\Vert_{H}^{2}d|F|_v(s)\rightarrow0, \\
			&\mathbb{E}\Big\Vert\int_{t}^{T}\big(\mathcal{M}(s)p^{n}(s)-\mathcal{M}(s)p(s)\big)ds%
			\Big\Vert _{V^{\ast}}^{2}  \notag \\
			&\leq T\mathbb{E}\int_{t}^{T}\Vert\mathcal{M}(s)p^{n}(s)-\mathcal{M}
			(s)p(s)\Vert_{V^{\ast}}^{2}ds \leq TK_{1}\mathbb{E}\int_{t} ^{T}\Vert
			p^{n}(s)-p(s)\Vert_{V}^{2}ds\rightarrow0, \\
			&\mathbb{E}\Big\Vert\int_{t}^{T}\big( g(s,p_{s{+}}^{n}(s),q_{s{+}
			}^{n}(s))-g(s,p_{s{+}}(s),q_{s{+}}(s))\big)ds\Big\Vert_{H}^{2} 
			\notag \\
			&\leq T\mathbb{E}\int_{t}^{T}\Vert g(s,p_{s{+}%
			}^{n}(s),q_{s_{+}}^{n}(s))-g(s,p_{s{+}}(s),q_{s{+}}(s))\Vert_{H}^{2}ds 
			\notag \\
			& \leq C\mathbb{E}\int_{t}^{T}\Big\{\Vert p^{n}(s)-p(s)\Vert_{V}^{2}+\Vert
			q^{n}(s)-q(s)\Vert_{\mathcal{L}_{2}^{0}}^{2}\Big\}ds\rightarrow0, \\
			&\mathbb{E}\Big\Vert\int_{t}^{T}\big(\mathcal{N}(s)q^{n}(s)-\mathcal{N}(s)q(s)\big)ds%
			\Big\Vert _{V^{\ast}}^{2}  \notag \\
			&\leq T\mathbb{E}\int_{t}^{T}\Vert\mathcal{N}(s)q^{n} (s)-\mathcal{N}%
			(s)q(s)\Vert_{V^{\ast}}^{2}ds \leq C\mathbb{E}\int_{t}^{T}\Vert
			q^{n}(s)-q(s)\Vert_{\mathcal{L}_{2}^{0}}^{2}ds\rightarrow0,
		\end{align*}
		and 
		\begin{align*}
			\mathbb{E}\Big\Vert\int_{t}^{T}\big(q^{n}(s)-q(s)\big)dw(s)\Big\Vert%
			_{H}^{2}=\mathbb{E}\int_{t}^{T}\Vert q^{n}(s)-q(s)\Vert_{\mathcal{L}%
				_{2}^{0}}^{2}ds\rightarrow0.
		\end{align*}
		The proof is concluded.
	\end{proof}
	
	\begin{remark}\label{rem:GM21}
		{ When $H=V=\mathbb{R}
			^{n}$,  $dF$ induces a finite 
			{(deterministic)} measure, the path dependence on $p$ and $q$ takes the form of an integral with respect to a { prescribed} finite 
			measure,  and the generator $g$ is linear, the equation 
			\eqref{absee-0} reduces to the ABSDE studied in \cite[Theorem 2.4]
			{guatteri2021stochastic} where the well-posedness was established.}
	\end{remark}

	\begin{remark}
		When adding a new term $\zeta$ to the ABSEE in Step 2 of the proof, we
		first consider the case of $V$-valued process $\zeta$, as the operator $%
		\mathcal{M}(t)$ acts only on the space $V$. On the other hand, the arguments in Step 2 remain valid for
		a general $V$-valued process $\alpha\in L_{\mathbb{F}%
		}^{1,2}(0,T;V)$ satisfying $\alpha(T)\in L^{2}(\mathcal{F}_T;H)$.
		Furthermore, if we assume the domain of $\mathcal{M}(t)$ is $H$, the arguments in Step 2 hold for a general $H$-valued process $\alpha\in L_{\mathbb{F}%
		}^{1,2}(0,T;H)$ satisfying $\alpha(T)\in L^{2}(\mathcal{F}_T;H)$. In particular, these extensions apply to the finite-dimensional case (i.e.,  when  $H=V=\mathbb R^n$).
		
	\end{remark}
	
	\section{ 
		{Path derivative and its adjoint operator}}
	
	\label{sec:path-derivative}

	In this section, we study the non-anticipative (or adapted) path
	derivative  and its { adjoint (dual) operator} that will be used in the derivation of
	the maximum principle in Section \ref{sec:SMP}.
	
	For a process $x$  on  $[T_1,T_2]$ with $T_1<T_2$, for $T_1\leq t_1\leq t_2\leq T_2$,  we define a process $x_{t_1,t_2}$  by
	\begin{align*}
		x_{t_1,t_2}(s):=x(t_1)\mathbb I_{[T_1, t_1)}(s) + x(s) \mathbb I_{[t_1, t_2]}(s)+ x(t_2) \mathbb I_{(t_2,T_2]}(s), \ s\in [T_1,T_2].
	\end{align*}

	Let $T>0$ and $K\ge0$ be fixed constants, and $E,F$ be  separable
	Hilbert spaces. 
	For $t\in[0,T]$, we define the subspace of $C(-K,T; E)$:
	\begin{equation}  \label{e:Ct}
		\begin{aligned} C_t(-K, T; E) :=&\Big\{ {x}_{t-K,t}=\big\{
			{x}_{t-K,t}(s), \forall s\in[-K, T]\big\}: ~ x\in C(-K, T;E)\Big\}.
		\end{aligned}
	\end{equation}

	Let $a:[0,T]\times {C(-K,T;E)} \to F$ be a Borel measurable function which is Fr\'echet differentiable in $x\in C(-K,T;E)$. Denote \begin{equation*}
		\hat a(t,x):=a(t, x_{t-K,t}),~ (t,x)\in [0,T]\times {C(-K,T;E)}. \end{equation*} 
	Clearly,  $\hat a(t,x) $ is also Frech\'et differentiable in $x$, and we denote  its derivative operator by 
	\begin{equation}  \label{e:extension}
		\rho_{x,t}(Z):=\partial_x\hat a(t,x)(Z)=\partial_x a(t, x_{t-K,t})(Z_{t-K,t}), ~ Z\in C(-K, T;E).
	\end{equation}
	Then the following \emph{non-anticipative condition} automatically holds for the operator $\rho_{x,t}$: \begin{equation}\label{non-anti}
		\rho_{x,t}(Z)=\rho_{x,t}(Z_{t-K,t}),\, \text{for}\, Z\in C(-K, T;E) \text{ and } t\in[0,T].
	\end{equation} 
	
	In the rest of this section, we often fix a path $x\in C(-K, T;E)$ and shall  omit
	the dependence on $x$ in notations for the sake of simplicity. For instance,
	we denote $\rho_t:=\rho_{x,t}=%
	\partial_x \hat a(t,x)$.

	We define the following right shift operator $\theta_t$ by, for a process $\bar{Z}$ on $[-K,0]$,
	\begin{equation*}
		(\theta _{t}\bar{Z})(s):=\bar{Z}(t-K)\mathbb I_{[-K,t-K)}(s)+\bar{Z}
		(s-t)\mathbb I_{[t-K,t]}(s)+\bar{Z}(0)\mathbb I_{(t,T]}(s), \, \forall s\in[-K,T],
	\end{equation*}
	which is a process on $[-K,T]$, and belongs to $C_t(-K,T;E)$ if $\bar{Z}\in C(-K,0;E).$
	The inverse operator $\theta_{-t}$ is defined by, for a process  ${Z}$ on $[-K,T]$,
	\begin{equation*}
		(\theta _{-t}Z)(s):=Z(s+t),\, \forall s\in[-K,0],
	\end{equation*}
	which is a process on $[-K,0]$, and belongs to $C(-K,0;E)$ if ${Z}\in C(-K,T;E).$
	We set, for $\bar{Z}\in C(-K,0;E)$, \[\bar{\rho}_{t}(\bar{Z}):=\rho _{t}(\theta _{t}\bar{Z})=\partial_x a(t, x_{t-K,t})((\theta _{t}\bar{Z})_{t-K,t})=\partial_x a(t, x_{t-K,t})(\theta _{t}\bar{Z}).\]
	Then $\bar{\rho}_{t}(\theta _{-t}Z)=\rho _{t}(\theta _{t}(\theta _{-t}Z))=\rho
	_{t}(Z)$, for  ${Z}\in C(-K,T;E)$.
	
	For each fixed $t\in[0,T]$, it is direct to see that $\bar{\rho}_{t}$ is a bounded linear operator from $%
	{C(-K,0;E)}$ to $F.$  By  the Dinculeanu-Singer Theorem (see, e.g., p.182 of \cite
	{1977duj}),  there exists a  finitely additive  $\mathcal{L}(E,F)$-valued measure $\nu
	(t,ds):=\nu(x, t, ds)$ on $[-K,0]$, such that\footnote{In general, for $a\leq b$, the integral with respect to a generic measure
		may be different on intervals such as $[a,b],(a,b],[a,b)$ and $(a,b)$. In
		this paper, for notational simplicity, we use the convention $%
		\int_{a}^{b}:=\int_{[a,b]}$.} 
	\begin{equation}  \label{e:barrho}
		\bar{\rho}_{t}(\bar{Z})=\int_{-K}^{0}\bar{Z}(s)\nu (t,ds),~\bar Z\in C(-K,
		0;E),
	\end{equation}
	with
	\begin{equation}\label{e:norm-rho-t}
		\|\bar \rho_t\|_{\mathcal L(C(-K,0;E),F)}=\|\rho_t\|_{\mathcal L(C_t(-K,T;E),F)}=\| \nu\|_v(t,[-K,0]),
	\end{equation}
	where $\|\nu\|_v(t,\cdot)$ is the \emph{semivariation} of $\nu(t,\cdot)$  (see Definition 4 on p.2 and Proposition 11 on p.4 of \cite{1977duj}): for $A\in\mathcal B([-K,0])$,
	\begin{align*}
		\|\nu\|_v(t,A): = &\sup \Big\{ \Big\| \sum_i^n v(t,A_i) x_i \Big\|_{\mathcal{L}(E,F)} :\ x_i\in \mathbb R,\ |x_i| \leq 1,\ \{A_{i},1\leq i\leq n\}\subset 
		\mathcal{B}([-K,0])\\ 
		&\hspace{1cm}\text{is a partition of}\ A,\ n\geq1 \Big\}.  
	\end{align*}
	
	In the remaining of this paper, we shall refer to $\nu(t,\cdot)$  as the ``representing measure'' of $ \rho_t=\partial_x \hat a(t,x)$ and $ \bar\rho_t$.

	Let $|\nu|_v(t,\cdot)$  denote the \emph{variation} of $\nu(t,\cdot)$, which is defined as (see Definition 4 in  \cite{1977duj}): 
	for $A\in \mathcal{B}([-K,0])$, 
	\begin{align*}
	|\nu |_{v}(t,A):=\sup \Big\{	&\sum_{i=1}^{n}\Big\Vert \nu
		(t,A_{i})\Big\Vert _{\mathcal{L}(E,F)}:\{A_{i},1\leq i\leq n\}\subset 
		\mathcal{B}([-K,0])\\& \text{ is a partition of }A,\ n\geq1\Big\}.
	\end{align*}
	For a fixed $t\in[0,T], \|\nu\|_{v}(t,A)\le |\nu |_{v}(t,A)$ for any $A\in \mathcal B([-K,0])$. Moreover, if both $E$ and $F$ are finite-dimensional, we have that $\|\nu\|_{v}(t,[-K,0])<\infty$ if and only if $|\nu|_{v}(t,[-K,0])<\infty$, while this is not the case if the dimension of $E$ or $F$ is infinite.

	Note that the Dinculeanu-Singer Theorem only implies that $\nu(t, \cdot)$ has a bounded semivariation, i.e., $\|\nu\|_v(t,[-K,0])<\infty$. For our purpose, we shall assume the following  uniform  boundedness condition for the variation of $\{\nu(t,\cdot)\}_{t\in[0,T]}$.
	\begin{assumptionp}{(C0)}\label{(C0)} The vector measure 
		$\nu(t,\cdot)$ is $\sigma$-additive  for all $t\in [0,T]$ and satisfying
		\begin{equation}\label{e:N0}
			M_0:=\sup_{t\in[0,T]} |\nu |_{v}(t,[-K,0])
			<\infty. \end{equation} \end{assumptionp}

	Note that the $\sigma$-additivity of  $\nu(t,\cdot)$ implies  that of $|\nu |_{v}(t,\cdot)$  and vice versa (see Proposition 9 on p. 3 of \cite{1977duj}). 
	
	\begin{remark}\label{rem:M0}
		As mentioned above, when $E$ and $F$ are finite-dimensional, the semivariation and variation are equivalent, and hence \eqref{e:N0} is equivalent to, in view of \eqref{e:norm-rho-t}, 
		\begin{equation}\sup_{t\in[0,T]}\|\rho_t\|_{\mathcal L(C(-K,T;E),F)}<\infty.
		\end{equation}
	\end{remark}

	By a limiting argument, we can extend $\bar\rho_t(\bar Z)$ in %
	\eqref{e:barrho} to all $\bar{Z}\in L^{1}_{\nu (t,\cdot )}(-K,0;E)$ such that
	\begin{equation*}
		\bar\rho_t(\bar Z)=\int_{-K}^{0}\bar{Z}(s)\nu (t,ds),~\bar Z\in L^{1}_{\nu
			(t,\cdot )}(-K,0;E),
	\end{equation*}
	where in general we denote, for $a<b$, $p\ge 1$, and an $\mathcal{L}(E,F)$-valued vector measure $\mu$, 
	\begin{align*}
		L^{p}_{\mu}(a,b;E):=\Big\{&f:[a,b]\to E \text{ is a measurable function such
			that }\\& ~\int_a^b \|f(s)\|^p_E |\mu|_v(ds)<\infty\Big\}.
	\end{align*}
	Correspondingly, we define the extension of $\rho _{t}$  by
	\begin{equation}\label{e:rho}
		\rho _{t}(Z):=\bar{\rho}_{t}(\theta _{-t}Z)=\int_{-K}^{0}Z(t+s)\nu (t,ds), 
	\end{equation}
	for $Z\ \text{on}\ [-K,T]$ satisfying  
	$\theta_{-t} Z\in L^{1}_{\nu (t,\cdot)}(-K,0;E)$.
	
	We also assume:
	\begin{assumptionp}{(C1)}\label{(C1)}
		There exists a finite measure $\nu _{0}(\cdot)=\nu_0(x,\cdot)$ on $[-K,0]$
		such that $|\nu |_v(t, \cdot)=|\nu |_v(x,t, \cdot)$ is absolutely continuous
		with respect to $\nu _{0}(\cdot)$ for all $t\in [0,T]$.
	\end{assumptionp}
	
	Assuming \ref{(C1)}, for each fixed $t\in [0,T]$, by Radon-Nikodym theorem for
	operator-valued measures (see, e.g., \cite[Theorem 3.3.2]{14r} and \cite[Theorem 2.5]{ll2017}), there
	exists a weakly measurable ( see \cite
	[Chapter 1]{81k} and  \cite
	[Section 2]{liu2021maximum} for the definition) operator-valued function $k(t, \cdot)=k(x, t,
	\cdot): [ -K,0]\to \mathcal{L}(E, F)$ such that 
	\begin{equation}  \label{e:rd}
		\nu (t,ds)=\frac{\nu (t,ds)}{\nu _{0}(ds)}\nu _{0}(ds)=k(t,s)\nu _{0}(ds).
	\end{equation}
	Then $\rho_t$ can be written as:
	\begin{equation}\label{eq:d-rho}
		\rho _{t}(Z) = \int_{-K}^{0}Z(t+s)k(t,s)\nu
		_{0}(ds).
	\end{equation}
	We note that condition \eqref{e:N0} in Assumption~\ref{(C0)} is equivalent to:
	\begin{assumptionp} {(C0$\hspace{0.3mm}'$)}\label{(C0')}Assume 
		\begin{equation}\label{e:M0}
			M_0:=\sup_{t\in[0,T]} \int_{-K}^0 \|k(t,s)\|_{\mathcal L(E,F)}\nu_0(ds)
			<\infty. \end{equation}
	\end{assumptionp}
	We 
	have the following result.
	\begin{lemma}
		\label{lem:k-wm} The mapping $[0,T]\times \lbrack -K,0]\ni (t,s)\mapsto
		k(t,s)\in\mathcal L(E,F)$ is weakly measurable.
	\end{lemma}
	
	\begin{proof} 
		We define a mapping $G:[0,T]\times C(-K,0;E)\times C(-K,T;E)\rightarrow [0,T]\times 
		{C(-K,T;E)}\times C(-K,T;E)$ by $
		G(t,\bar{Z},x)=(t,\theta _{t}\bar{Z}, x_{t-K,t})$, which is Borel measurable.  We also define   $J:[0,T]\times C(-K,T;E)\times C(-K,T;E)\rightarrow F$ by 
		$ J(t,Z,\bar x)=\partial_x a(t, \bar x)(Z)$, which is also  a
		Borel measurable mapping by noting 
		\begin{equation*}
			\partial_x a(t,\bar  x)(Z)=\lim_{\alpha \rightarrow 0}\frac{a(t,\bar x+\alpha {Z})-a(t,\bar x)}{\alpha}.
		\end{equation*}%
		Then, the composition $\bar{\rho}_{ x,t}({\bar{Z}})=\rho _{x,t}(\theta _{t}\bar{Z}%
		)=\partial_x a(t, x_{t-K,t})(\theta _{t}\bar{Z})=J(G((t,x,\bar{Z})))$ is a Borel measurable mapping from $[0,T]\times C(-K,0;E)\times C(-K,T;E)$ to $F$.  In particular,  $[0,T]\ni t\mapsto 
		\bar{\rho}_{t}({\bar{Z}})=\bar{\rho}_{x,t}({\bar{Z}})\in F$ is measurable for any fixed $(x,\bar{Z})\in C(-K,T;E)\times C(-K,0;E)$. For a bounded measurable $\bar{Z}:[-K,0]\rightarrow E,$ we can find a uniformly bounded
		sequence $\bar{Z}^{n}\in C(-K,0{;E})$ such that $\bar{Z}^{n}\rightarrow $ $%
		\bar{Z}$ in measure $\nu _{0}.$ Then by the dominated convergence theorem,
		for each $t\in \lbrack 0,T],$ we have  as $n\to \infty$ that
		\begin{equation*}
			\bar{\rho}_{t}({\bar{Z}}^{n})=\int_{-K}^{0}k(t,s)\bar{Z}^{n}(s)\nu
			_{0}(ds)\rightarrow \int_{-K}^{0}k(t,s)\bar{Z}(s)\nu _{0}(ds)=\bar{\rho}_{t}(%
			{\bar{Z}}).
		\end{equation*}%
		Hence, $[0,T]\ni t\mapsto \bar{\rho}_{t}({\bar{Z}})\in F$
		is measurable.

		Now, for fixed  $A\in \mathcal{B}([-K,0])$ and $e\in E$, we take ${
			\bar{Z}}(s)=e\mathbb{I}_{A}(s)$,   which is clearly bounded measurable. Then from
		\begin{equation*}
			\bar{\rho}_{t}({\bar{Z}})=\int_{-K}^{0}\mathbb{I}_{A}{(s)}%
			k(t,s)e\nu _{0}(ds)=\nu (t,A)e,
		\end{equation*}%
		we get that $[0,T]\ni t\mapsto \nu (t,A)e\in F$ is measurable, which implies $t\mapsto \nu
		(t,A)$ is weakly measurable. Note that for any $(e,f)\in E\times F$, 
		\begin{equation*}
			\left\langle \nu (t,A)e,f\right\rangle _{F}=\int_{A}\left\langle
			k(t,s)e,f\right\rangle _{F}\nu _{0}(ds),
		\end{equation*}
		which indicates that $s\mapsto \left\langle k(t,s)e,f\right\rangle _{F}$ is the classical
		Radon-Nikodym derivative of the (real-valued signed) measure $\left\langle
		\nu (t,\cdot )e,f\right\rangle _{F}$ with respect to $\nu _{0}$. Hence, $[0,T]\times \lbrack
		-K,0]\ni (t,s)\mapsto \left\langle k(t,s)e,f\right\rangle _{F}\in \mathbb{R}$
		is measurable   by \cite[Theorem 58 in p.52]{1982dm}, and thus $(t,s)\mapsto k(t,s)$ is weakly measurable.\end{proof}
	
	The following lemma allows us to employ the technique of change of variables, which is frequently used in this section.
	\begin{lemma}\label{inte-trans-lem}
		Let $g(t,s):[0,T]\times \lbrack
		-K,0]\rightarrow \mathbb{R}$ be a measurable function.
		Then for any $K'\in [-K,T],$
		\begin{equation}\label{e:change-variable}
			\int_{0}^{T}\int_{-K}^{0}g(t,s)\mathbb{I}_{[K'-t,0]}(s)\nu
			_{0}(ds) dt=\int_{K'}^{T}\int_{-K}^{0}g(u-v,v)\mathbb{I}_{[u-T,u]}(v)\nu
			_{0}(dv)du,
		\end{equation}
		provided that the integral on either side of \eqref{e:change-variable}  is well defined. 
	\end{lemma}
	
	\begin{proof} 
		Let $B\in \mathcal B(\R^2)$ be a bounded Borel measurable set and  denote $\mu (ds,dt)=\nu _{0}(ds)dt$. For the mapping $f:(t,s)\mapsto (u,v):=(t+s,s)$ from $\R^2$ to $\R^2$,   we have
		\begin{align*}
			&(\mu\circ f^{-1})(B) := \mu(f^{-1}(B))\\
			&=\int_{\mathbb{R}}\int_{\mathbb{R}}\mathbb I_{f^{-1}(B)}(t,s)\,dt
			\,\nu _{0}(ds) \\
			&=\int_{\mathbb{R}}\int_{\mathbb{R}}\mathbb I_{B}(t+s,s)\,dt\,\nu _{0}(ds) \\
			&=\int_{\mathbb{R}} \int_{\mathbb{R}}\mathbb I_{B}(u,s)\,du \nu
			_{0}(ds) \\
			&=\int_{\mathbb{R}}\int_{\mathbb{R}}\mathbb I_{B}(u,v)\,\nu _{0}(dv)\,du= \mu(B).
		\end{align*}
		where the fourth equality follows from the translation invariance of the Lebesgue measure. Thus, we have 
		$(\mu\circ f^{-1})(dv,du)=\mu(dv,du)=\nu _{0}(dv)du$. 
		Denoting $G(t,s)=g(t,s)\mathbb I_{[0,T]}(t) \mathbb I_{[-K,0]}(s)\mathbb{I}_{[K'-t,0]}(s)$, we get 
		\begin{align*}
			\int_{0}^{T}\int_{-K}^{0}g(t,s)\mathbb{I}_{[K'-t,0]}(s)\nu
			_{0}(ds)dt&=\int_{\R^2} G(t,s) \mu(ds,dt)\\
			&=\int_{\R^2} (G(u-v, v)\circ f) (t,s) \mu(ds,dt)\\
			&=\int_{\R^2}G(u-v,v) (\mu\circ f^{-1})(du,dv)\\
			& =\int_{ K'}^{T}\int_{-K}^{0}g(u-v,v)\mathbb{I}_{[u-T,u]}(v)\nu _{0}(dv)du.
		\end{align*}
		The proof is complete.
	\end{proof}

	In order to carry on a proper dual analysis on some Hilbert spaces for the path derivative, we shall  regard  $\rho$ as a
	bounded linear operator mapping from $L^2(-K,T;E)$ to $L^2(0,T;F)$. For this purpose, we impose the following assumption.

\begin{assumptionp}{(C2)}\label{(C2)}
	We  assume 
	\begin{equation}\label{e:M}
		M:=\sup_{t\in \lbrack -K,T]}\int_{-K}^{0}\left\Vert k(t-s,s)\right\Vert _{%
			\mathcal{L}(E,F)}\mathbb I_{[t-T, t]} (s) \nu _{0}(ds)<\infty .\end{equation}
\end{assumptionp}

 The following result holds in a more general setting, so we  write $\varrho$ in place of $\rho$.

\begin{proposition}
	\label{Prop1} Let $\nu _{0}$ be a   finite measure on $[-K,0]$ and $%
	k:[0,T]\times \lbrack -K,0]\rightarrow \mathcal{L}(E,F)$ be a weakly
	measurable operator-valued function satisfying Assumptions  \ref{(C0')} and   \ref{(C2)}. For each $t\in \lbrack
	0,T]$, denote  
	\begin{equation*}
		\varrho _{t}(Z)=\int_{-K}^{0}Z(t+s)k(t,s)\nu _{0}(ds),~Z\in L^{2}(-K,T;E)%
		\text{.}
	\end{equation*}%
	Then $\varrho _{t}(Z)$ is a well-defined Bochner integral for almost all $%
	t\in \lbrack 0,T]$, and moreover,  for each $T'\in(0,T]$,  
	\[ \int_0^{T'} \|\varrho_t(Z)\|^2_F dt \le M_0M \int_{-K}^{T'} \|Z(u)\|_E^2 du, \text{ for all } Z\in L^2(-K, T';E), \]
	where the constants  $M_0$ and $M$ are from Assumption \ref{(C0')} and   \ref{(C2)}. In particular,  
	$\varrho:=\{\varrho_t(\cdot), t\in[0,T]\}$   is a bounded linear operator mapping from $L^2(-K, T;E)$ to $L^2(0,T;F)$. 
\end{proposition}

\begin{proof}
	First, we show that for $Z\in L^{1}(-K,T;E),$ $\{\varrho _{t}(Z),t\in
	\lbrack 0,T]\}\in L^{1}(0,T;F)$. By Lemma~\ref{lem:k-wm},  $(t,s)\mapsto k(t,s)$ is weakly measurable and hence $(t,s)\mapsto \Vert k(t,s)\Vert _{
		\mathcal{L}(E,F)}$ is measurable.
	Applying Lemma~\ref{inte-trans-lem},  we get
	\begin{align*}
		& \int_{0}^{T}\int_{-K}^{0}\left\Vert Z(s+t)k(t,s)\right\Vert _{F}\nu
		_{0}(ds)dt \\
		& =\int_{-K}^{T}\left\Vert Z(u)\right\Vert _{E}\int_{-K}^{0}\left\Vert
		k(u-v,v)\right\Vert _{\mathcal{L}(E,F)}\mathbb{I}_{[u-T,u]}(v)\nu _{0}(dv)du
		\\
		& \leq M\int_{-K}^{T}\left\Vert Z(t)\right\Vert _{E}dt<\infty .
	\end{align*}%
	Thus, for almost all $t\in \lbrack 0,T]$, $\int_{-K}^{0}\left\Vert
	Z(s+t)k(t,s)\right\Vert _{F}\nu _{0}(ds)<\infty $ and it follows from
	Theorem 1 in p.133 in \cite{80y} that $\varrho
	_{t}(Z)=\int_{-K}^{0}Z(s+t)k(t,s)\nu _{0}(ds)$ is a well-defined Bochner
	integral.
	
	Next, for $Z\in L^{2}(-K,T;E)$,  by the H\"older's inequality and applying Lemma \ref{inte-trans-lem}  again, we have for each $T^{\prime }\in (0,T]$ that
	\begin{align*}
		&\int_{0}^{T^{\prime }}\Vert \varrho _{t}(Z)\Vert _{F}^{2}dt\\&
		=\int_{0}^{T^{\prime }}\Big\Vert \int_{-K}^{0}Z(t+s)k(t,s)\nu
		_{0}(ds)\Big\Vert _{F}^{2}dt \\
		& \leq \int_{0}^{T^{\prime }}\Big(\int_{-K}^{0}\big\|Z(t+s)k(t,s)\big\|%
		_{E}\nu _{0}(ds)\Big)^{2}dt \\
		& \leq \int_{0}^{T^{\prime }}\int_{-K}^{0}\Vert Z(t+s)\Vert _{E}^{2}\Vert
		k(t,s)\Vert _{\mathcal{L}(E,F)}\nu _{0}(ds)\int_{-K}^{0}\Vert k(t,s)\Vert _{%
			\mathcal{L}(E,F)}\nu _{0}(ds)dt \\
		& \leq M_0\int_{0}^{T^{\prime }}\int_{-K}^{0}\Vert Z(t+s)\Vert _{E}^{2}\Vert
		k(t,s)\Vert _{\mathcal{L}(E,F)}\nu _{0}(ds)dt \\
		& = M_0\int_{-K}^{T^{\prime }}\int_{-K}^{0}\Vert Z(u)\Vert _{E}^{2}\Vert
		k(u-v,v)\Vert _{\mathcal{L}(E,F)} \mathbb I_{[u-T, u]}(v)\nu _{0}(dv)du \\
		& \leq M_0M\int_{-K}^{T'}\left\Vert Z(u)\right\Vert _{E}^{2}du,
	\end{align*}
	which completes the proof.
\end{proof}

As a direct application of Proposition \ref{Prop1},
we have the following result.

\begin{corollary}
	\label{cor1} Assume \ref{(C0)}, \ref{(C1)} and \ref{(C2)} for (the representing measure of) $\rho_t=\partial_x \hat a(t,x)$.
	Then, $\rho=\{\rho_t(\cdot), t\in[0,T]\}$   is a bounded linear operator mapping from $L^2(-K, T;E)$ to $L^2(0,T;F)$.
\end{corollary}

Let $\varrho $ be a bounded linear operator mapping from $L^{2}(-K,T;E)$ to $%
L^{2}(0,T;F)$. Its \emph{adjoint operator} $\varrho ^{\ast }=\big\{\varrho
_{t}^{\ast }(\cdot ),t\in \lbrack -K,T]\big\}:L^{2}(0,T;F)\rightarrow
L^{2}(-K,T;E)$ satisfies, for all $Z\in L^{2}(-K,T;E)$ and $Q\in L^{2}(0,T;F)$%
, 
\begin{equation}
	\int_{0}^{T}\left\langle \varrho _{t}(Z),Q(t)\right\rangle
	_{F}dt=\int_{-K}^{T}\left\langle Z(t),\varrho _{t}^{\ast }(Q)\right\rangle
	_{E}dt.  \label{e:dualilty}
\end{equation}
We have an explicit expression for the adjoint operator $\varrho ^{\ast}$ presented below.

\begin{proposition}
	\label{thm:rho*} Let $\varrho :L^{2}(-K,T;E)\rightarrow L^{2}(0,T;F)$ be a
	bounded linear operator defined as in Proposition \ref{Prop1} and $\varrho
	^{\ast }$ be its adjoint operator.  Then, for $Q\in L^{2}(0,T;F)$, 
	\begin{equation}
		\varrho _{t}^{\ast }(Q)=\int_{-K}^{0}k^{\ast }(t-s,s)Q(t-s)\mathbb I_{[t-T,t]}(s)\nu
		_{0}(ds),~t\in \lbrack -K,T],  \label{e:rho*}
	\end{equation}%
	where $k^{\ast }(t,s)$ is the adjoint operator of $k(t,s)$.  Moreover,  for each $K'\in[-K,T]$, we have  
	\begin{equation}\label{Eq4-13} \int_{K'}^T \|\varrho_t^*(Q)\|_E^2dt 
		\le MM_0 \int_{0\vee K'}^T \|Q(u)\|_F^2du, \text{ for all } Q\in L^{2}(0,T;F).\end{equation}
\end{proposition}

\begin{proof}
	We have, for $Q\in L^{2}(0,T;F)$, 
	\begin{equation*}
		\int_{0}^{T}\left\langle \varrho _{t}(Z),Q(t)\right\rangle
		_{F}dt=\int_{0}^{T}\int_{-K}^{0}\left\langle Z(t+s)k(t,s),Q(t)\right\rangle
		_{F}\nu _{0}(ds)dt.
	\end{equation*}%
	By  Lemma \ref{inte-trans-lem}, we get 
	\begin{align*}
		&\int_{0}^{T}\int_{-K}^{0}\left\langle Z(t+s)k(t,s),Q(t)\right\rangle _{F}\nu
		_{0}(ds)dt \\
		&=\int_{-K}^{T}\int_{-K}^{0}\left\langle Z(u)k(u-v,
		v),Q(u-v)\right\rangle _{F}\mathbb I_{[u-T,u]}(v)\nu _{0}(dv)d
		u \\
		&=\int_{-K}^{T}\Big\langle Z(u),\int_{-K}^{0}Q(u-v)k ^{\ast }(u-v,v)\mathbb I_{[u-T,u]}(v)\nu _{0}(dv)\Big\rangle _{F}d%
		u.
	\end{align*}
	Thus, \eqref{e:rho*} follows directly from the definition of adjoint
	operator.
	
	For $K^{\prime }\in \lbrack -K,T]$, noting that the norm of $k^{\ast }(t,s)$
	coincides with that of $k(t,s)$, we get 
	\begin{align*}
		& \int_{K^{\prime }}^{T}\Vert \varrho _{t}^{\ast }(Q)\Vert
		_{E}^{2}dt=\int_{K^{\prime }}^{T}\Big\Vert \int_{-K}^{0}k^{\ast
		}(t-s,s)Q(t-s)\mathbb{I}_{[t-T,t]}(s)\nu _{0}(ds)\Big\Vert _{E}^{2}dt \\
		& \leq \int_{K^{\prime }}^{T}\Big(\int_{-K}^{0}\big\|k^{\ast }(t-s,s)Q(t-s)%
		\mathbb{I}_{[t-T,t]}(s)\big\|_{E}\nu _{0}(ds)\Big)^{2}dt \\
		& \leq \int_{K^{\prime }}^{T}\Big(\int_{-K}^{0}\big\|k^{\ast }(t-s,s)\Vert _{%
			\mathcal{L}(F,E)}\Vert Q(t-s)\Vert _{F}\mathbb{I}_{[t-T,t]}(s)\nu _{0}(ds)%
		\Big)^{2}dt \\
		& \leq \int_{K^{\prime }}^{T}\int_{-K}^{0}\Vert Q(t-s)\Vert _{F}^{2}\Vert
		k^{\ast }(t-s,s)\Vert _{\mathcal{L}(F,E)}\mathbb{I}_{[t-T,t]}(s)\nu
		_{0}(ds)\\&\hspace{2em}\times\int_{-K}^{0}\Vert k^{\ast }(t-s,s)\Vert _{\mathcal{L}(F,E)}\nu
		_{0}(ds)dt \\
		& \leq M\int_{K^{\prime }}^{T}\int_{-K}^{0}\Vert Q(t-s)\Vert
		_{F}^{2}\Vert k^{\ast }(t-s,s)\Vert _{\mathcal{L}(F,E)}\mathbb{I}%
		_{[t-T,t]}(s)\nu _{0}(ds)dt.
	\end{align*}%
	Then according to Lemma \ref{inte-trans-lem},
	\begin{align*}
		& \int_{K^{\prime }}^{T}\int_{-K}^{0}\Vert Q(t-s)\Vert _{F}^{2}\Vert k^{\ast
		}(t-s,s)\Vert _{\mathcal{L}(F,E)}\mathbb{I}_{[t-T,t]}(s)\nu _{0}(ds)dt \\
		& =\int_{ 0\vee K^{\prime }}^{T}\int_{-K}^{0}\Vert Q(u)\Vert _{F}^{2}\Vert k^{\ast }(u,v
		)\Vert _{\mathcal{L}(F,E)} \mathbb{I}_{[K'-u,0]}(v)\nu _{0}(dv)d%
		u \\
		& \leq \int_{ 0\vee K^{\prime }}^{T}\Vert Q(u)\Vert
		_{F}^{2}\int_{-K}^{0}\Vert k(u,v)\Vert _{%
			\mathcal{L}(E,F)}\nu _{0}(dv)du \\
		& \leq M_0\int_{  0\vee K^{\prime }}^{T}\Vert Q(u)\Vert _{F}^{2}du.
	\end{align*}%
	Combining the above two inequalities, we obtain (\ref{Eq4-13}).
\end{proof}

\begin{remark}
	\label{rem:rho*} Observing that  $\varrho$ defined as in Proposition \ref{Prop1}  is non-anticipative: 
	\begin{equation}  \label{e:rho-rem}
		\varrho_t(Z)=\varrho_t(Z_{t-K,t}),\ \text{for}\ Z\in L^2(-K,T;E).
	\end{equation} As a direct consequence of \eqref{e:rho*}, the adjoint 
	operator $\varrho^*$ is  \emph{anticipative} or
	\emph{non-adapted} in the sense that 
	\begin{equation}  \label{e:rho*-rem}
		\varrho^*_t (Q)=\varrho^*_t (Q_{t, (t+K)\wedge T}), \ \text{for}\ Q\in L^{2}(0,T;F).
	\end{equation}
	This will yield  an \emph{anticipated} BSEE (see \eqref{e-adjoint}) in the derivation of the maximum principle in 
	Section \ref{sec:SMP}.
\end{remark}

\begin{remark}
	\label{Rem:dual-for-u} 
	
	The results in  Propositions \ref{Prop1} and \ref{thm:rho*} also apply to situations beyond  path derivatives. For example, let $b:[0,T]\times E\rightarrow F$ be a measurable  function that is
	Fr\'echet differentiable in $E$ with uniformly bounded  derivatives. Let $\mu$ be a finite measure on $[-K,0]$. Denote 
	\begin{equation}\label{e:x-mu}
		x_{\mu}(t):=\int_{-K}^{0}x(t+s)\mu(ds),~ 
		t\in[0,T],
	\end{equation}
	provided that the integral exists. For  $x\in L_\mu^2(-K,T;E)$, we have
	$$ b(t,x_{\mu}(t)+Z_{\mu}(t))-b(t,x_{\mu}(t)) =\partial_{x}b(t,x_{\mu}(t))Z_{\mu}(t)+o( Z_{\mu}(t)), ~Z\in  L_\mu^2(-K, T;E).$$ Denote
	\begin{equation*}
		\varrho_{t}(Z):=\partial
		_{x}b(t,x_{\mu}(t))Z_{\mu}(t)=\int_{-K}^{0}\partial
		_{x}b(t,x_{\mu}(t))Z(t+s)\mu(ds), ~Z\in L^2(-K, T;E),
	\end{equation*}
	which is well defined by Proposition \ref{Prop1}. Then from Proposition \ref{thm:rho*},  the adjoint operator $\varrho^*$ is characterized by 
	\begin{align*}
		\varrho_t^*(Q)=\int_{-K}^0 \big(\partial_{x}b(t-s,x_{\mu}(t-s))\big)^*Q(t-s)
		\mathbb I_{[t-T,t]}(s) \mu(ds), ~Q\in L^2(0,T; F).  
	\end{align*}
	This will be used in the dual analysis of control delay in Section \ref{sec:SMP}.
\end{remark}

\begin{example}
	\label{example1} 
	Let $\tilde{a}:[0,T]\times E\rightarrow F$ be a measurable function that is 
	Fr\'echet differentiable in $E$  with uniformly bounded derivatives.  Set $a(t,x)=\tilde{a}
	(t,x_{\mu}(t)) $ for  $x\in C(-K,T;E)$,  where $x_{\mu}$ is given by \eqref{e:x-mu}. Since
	\begin{eqnarray*}
		a(t,x+Z)-a(t,x) &=&\tilde{a}(t,x_{\mu}(t)+Z_{\mu}(t))-\tilde{a}(t,x_{\mu}(t))
		\\
		&=&\partial_{x}\tilde{a}(t,x_{\mu}(t))Z_{\mu}(t)+o( Z_{\mu}(t))
		\\
		&=&\partial_{x}\tilde{a}(t,x_{\mu}(t))Z_{\mu}(t)+o(\|Z\|_{C(-K, T; E)}), ~Z\in  C(-K, T;E),
	\end{eqnarray*}
	the Fr\'echet  derivative of $a$ at $x$ is $\partial _{x}a(t,x)(Z)=\partial
	_{x}\tilde{a}(t,x_{\mu}(t))Z_{\mu}(t)$ and the path derivative operator
	\begin{equation}  \label{e:varrho}
		\begin{aligned} \varrho_{x,t}(Z)&=\partial _{x}a(t,x_{t-K,t})(Z_{t-K,t})\\ &=\partial
			_{x}\tilde{a}(t,x_{\mu}(t))Z_{\mu}(t)\\ &= \int_{-K}^{0}Z(t+s)\partial
			_{x}\tilde{a}(t,x_{\mu}(t))\mu(ds), ~Z\in C(-K, T;E). \end{aligned}
	\end{equation}
	
	In this case, we have $\nu (x, t,ds)=\partial_{x}\tilde{a} (t,x_{\mu}(t))\mu(ds)$. Thus, Assumption \ref{(C1)} is satisfied with $\nu_{0}(ds)=\mu(ds)$, and the corresponding Radon-Nikodym derivative $
	k(t,s) =\partial_x \tilde a(t, x_\mu(t))$. Moreover, Assumptions \ref{(C0')} and \ref{(C2)} are fulfilled, and hence  the domain $C(-K,T;E)$ of $\varrho_{x,t}$ in \eqref{e:varrho} can 
	be extended to $L^2(-K,T;E)$ by Corollary~\ref{cor1}. \end{example}

\begin{remark}
	\label{rem:derivatives} In the discussion of path derivatives with respect to the variable 
	$x$, we have assumed that $a(t,x):[0,T]\times {C(-K,T;E)} \to F$ is Fr\'echet differentiable. However, the results in Sections \ref{sec:path-derivative}  and \ref{sec:SMP} remain valid with straightforward modifications for a notion called \emph{non-anticipatively  differentiable}  in the sense that, for any fixed $x\in C(-K,T;E)$,  there exists, for each $t\in[0,T]$, a bounded linear operator 
	\begin{equation*}
		A_{x,t}: C_t(-K, T;E)\to F,
	\end{equation*}
	such that
	\begin{equation*}
		a(t,x+h)= a(t,x) + A_{x,t}(h) + o(\|h\|_{C(-K, T; E)}), ~h\in C_t(-K, T;E).
	\end{equation*}
	If such an operator $A_{x,t}$ exists, then it is unique (in $C_t(-K,T;E)$), and we
	denote it by $D_x a(t,x)$. 
	We define the path derivative operator $\rho_{x,t}$ by
	\begin{equation*}  
		\rho_{x,t}(Z):=D_x a(t, x_{t-K,t})(Z_{t-K,t}), ~ Z\in C(-K, T;E),
	\end{equation*}
	which is automatically  non-anticipative by construction.

	The non-anticipative differentiability is slightly weaker than Fr\'echet differentiability, as  the former requires fewer test elements $Z$.  Thus, if the Fr\'echet derivative $\partial_x a(t,x)$ of $%
	a(t,\cdot)$ at $x$ exists for all $t\in[0,T]$, then it is non-anticipatively differentiable with
	\begin{equation*}
		D_{x}a(t,x)(Z)=\partial_x a(t, x)(Z), ~ Z\in C_t(-K, T;E).
	\end{equation*}

	Clearly, if the function $a$ satisfies the non-anticipative condition, i.e.,  $a(t,x)=a(t,x_{t-K,t})$  for all $x\in {C(-K,T;E)}$, then $\partial_x a(t, x)$ and $D_{x}a(t,x) $ coincide.
\end{remark}
\section{Stochastic maximum principle}\label{sec:SMP}
In this section, we study the recursive optimal control problem for a class of
infinite-dimensional path-dependent systems and derive the Pontryagin's stochastic
maximum principle.

\subsection{Formulation of the control problem}
Suppose that the control domain $U$ is a convex subset of a real separable
Hilbert space $H_{1}$ which is identified  with its dual space.  Consider
the following controlled PSEE
\begin{equation}
	\left\{ \begin{aligned} dx(t)=& \,\,\big[A(t)x(t)+b(t,x_{t-K,t},u_{\mu_1}(t))\big]dt \\ &
		+\big[B(t)x(t)+\sigma (t,x_{t-K,t},u_{\mu_1}(t))\big]dw(t), \quad t\in[0,T], \\ x(t)=&
		\,\,\gamma(t),\,\,u(t)=v(t),\quad t\in \lbrack -K ,0], \end{aligned}\right.
	\label{state}
\end{equation}
where  $\gamma(\cdot )\in  C(-K,0;H)$ and $
v_0(\cdot )\in L^{2}(-K,0;U)$ are given initial paths, 
\begin{equation*}
	(A,B):[0,T]\times \Omega \rightarrow \mathcal L(V;V^{\ast }\times \mathcal{L}_{2}^{0})
\end{equation*}%
are random  unbounded linear
operators,
\begin{equation*}
	(b,\sigma ): \lbrack 0,T]\times\Omega \times  C(-K,T;H)\times
	H_{1}\rightarrow H\times \mathcal{L}_{2}^{0}
\end{equation*}
are  random nonlinear functions,
\begin{equation}\label{e:x-t-k-t}
	x_{t-K,t}(s)=x(t-K)\mathbb I_{[-K,t-K)}(s)+x(s)\mathbb I_{[t-K,t]}(s)+x(t)\mathbb I_{(t,T]}(s),
	\text{ }s\in \lbrack -K,T],
\end{equation}
and for a finite measure $\mu_1$ on $[-K,0]$,
\begin{equation}\label{e:um}
	u_{\mu_1}(t):=\int_{-K}^{0}u(t+s)\mu_1(ds).
\end{equation}

The cost functional is defined by 
\begin{equation*}
	J(u(\cdot )):=y(0),
\end{equation*}%
where $(y(\cdot),z(\cdot))$ solves the following BSDE
\begin{equation}
	\left\{ \begin{aligned} -dy(t)=&\,\, f(t,x_{t-K,t},y(t),z(t),u_{\mu_1}(t)) dt
		-z(t)dw(t),\quad t\in \lbrack 0,T], \\ y(T)
		=&\,\,h\big(x_{\mu_2}(T)\big).
	\end{aligned}\right.  \label{y}
\end{equation}
In \eqref{y}, 
\begin{equation*}
	h:\Omega \times H\rightarrow \mathbb{R}\text{ and }f:[0,T]\times \Omega
	\times C(-K,T;H)\times \mathbb{R}\times \mathcal{L}_{2}^{0}(\mathcal K,\R)\times
	H_{1}\rightarrow \mathbb{R}
\end{equation*}
are the coefficient
functions, and   
\begin{equation}
	x_{\mu_2}(T):=\int_{-K}^{0}x(T+s)\mu_2(ds),  \label{e:x=dv}
\end{equation}%
with $\mu_2$ being a finite measure on $[-K,0]$. 
The admissible control set $\mathcal{U}$ is defined by
\begin{equation*}
	\mathcal{U}:=\Big\{u:[-K,T]\times \Omega \rightarrow U\ \text{satisfying }%
	u|_{[0,T]}\in L_{\mathbb{F}}^{2}(0,T;U)\ \text{and}\ u(t)=v_0(t),\,t\in \lbrack
	-K,0]\Big\}.
\end{equation*}

We aim to find necessary conditions (i.e., the maximum principle) 
for an optimal control $\bar{u}$, i.e., an admissible control $\bar{u}%
(\cdot )$ that minimizes the cost functional $J(u(\cdot ))$ over $\mathcal{U}
$.

Assume the following conditions hold.
\begin{itemize} 
	\item[({H1})] $b(\cdot ,\cdot ,0,0)\in L_{\mathbb{F}}^{2}(0,T;H)$, $%
	\sigma (\cdot ,\cdot ,0,0)\in L_{\mathbb{F}}^{2}(0,T;\mathcal{L}_{2}^{0}).$
	
	\item[({H2})] The operators $A$ and $B$ satisfy  (A2)-(A3).
	
	\item[({H3})]  For each $(x,v)\in  	C(-K,T;H)\times H_{1},$ the
	functions $b(\cdot ,\cdot,x,v)$ and $\sigma (\cdot ,\cdot,x,v)$
	are progressively measurable.  $b$ and $\sigma $ are Fr\'echet differentiable with
	respect to $x$ and $v$ with continuous and uniformly bounded
	derivatives.
	
	\item[({H4})]
	For each $(x,y,z,v)\in 
	C(-K,T;H)\times \mathbb{R}\times \mathcal{L}_{2}^{0}(\mathcal K,\mathbb{R})\times
	H_{1},$ $f(\cdot ,\cdot ,x,y,z,v)$ is progressively measurable and for $x^1\in H$, $h(\cdot ,x^1)$ is $\mathcal{F}_{T}$-measurable. The functions
	$f$ and $h$ are Fr\'echet differentiable with respect to  $(x,y,z,v)$
	and $x^1$, respectively,  with continuous and uniformly bounded derivatives. 
	
	\item[({H5})]  The (representing measures (see (\ref{e:barrho})) of)  path derivatives $\partial
	_{x}\hat b(t,x,v)$, $\partial _{x}\hat \sigma (t,x,v)$ and $\partial
	_{x}\hat f(t,x,y,z,v)$ (see \eqref{e:extension}) satisfy Assumptions \ref{(C0)}, ~\ref{(C1)}  and \ref{(C2)}, with  common bounds  $M_0^{x,\omega ,v,y,z}$ and $M^{x,\omega ,v,y,z}$, uniformly for all $(x,\omega ,v,y,z)$, where  $M_0$ and $M_1$ appear in Assumptions  \ref{(C0)} and \ref{(C2)}, respectively.
	
\end{itemize}

\begin{remark}
	If the dimensions of $\mathcal K$ and  $H$ are finite, in view of Remark \ref{rem:M0}, the conditions for $\partial
	_{x}b$, $\partial _{x}\sigma$ and $\partial
	_{x}f$ in (H3) and (H4) already imply  Assumption \ref{(C0)}, which is assumed in (H5).
\end{remark}

Note that under  {(H1)-(H4)}, {equation}~\eqref{state} admits a unique solution by Theorem \ref{exu}, if we take \[\tilde{b}(t,\omega,x)=b(t,\omega,x_{\cdot\vee (t-K)},u_{\mu_1}(t,\omega)) \ \text{and}\ 
\tilde{\sigma}(t,\omega,x)=\sigma(t,\omega,x_{\cdot\vee (t-K)},u_{\mu_1}(t,\omega)),\] for $(t,\omega,x,u)\in\lbrack 0,T] \times \Omega \times  C(-K,T;H)\times
\mathcal{U},$ in {equaiton}~\eqref{sdee-0}. 

\subsection{Variational equations}
Let $\bar u(\cdot)\in \mathcal{U}$ be an optimal control, and $\bar
x(\cdot)$ and $(\bar y(\cdot), \bar z(\cdot))$ be the corresponding solutions to \eqref{state} and \eqref{y} respectively.  For $\rho \in \lbrack 0,1]$ and 
$u(\cdot )\in \mathcal{U}$, we define the perturbation 
of $\bar u(\cdot)$ by
\begin{equation*}
	u^{\rho }(\cdot )=\bar{u}(\cdot )+\rho (u(\cdot )-\bar{u}(\cdot )).
\end{equation*}
The convexity of $U$ yields that $u^{\rho }(\cdot )\in \mathcal{U}$. Let $x^{\rho}(\cdot )$ and $(y^{\rho }(\cdot ),z^{\rho }(\cdot ))$ be the corresponding solutions of \eqref{state} and \eqref{y} associated with $u^{\rho }(\cdot )$, respectively. 

For the functions $b(t,x,v), \sigma(t,x,v), f(t,x,y, z,v)$ and $h(x^1)$, where $$(x,y,z,v,x^1)\in C(-K,T;H)\times \mathbb{R}\times \mathcal{L}_{2}^{0}(\mathcal K,\mathbb{R})\times
H_{1}\times H,$$ we take the following notations, for $\varphi=b, \sigma$ and $\tau=x,y,z,v$,
\begin{equation}\label{e:abbreviations}
	\begin{aligned}
		\varphi (t)& :=\varphi (t,\bar x_{t-K,t},\bar u_{\mu_1}(t)), \\
		\partial _{\tau }\varphi (t)& :=\partial _{\tau }\varphi (t,\bar x_{t-K,t},\bar{u
		}_{\mu_1}(t)), \\
		f(t)& :=f(t,\bar x_{t-K,t},\bar{y}(t),\bar{z}(t),\bar{u}_{\mu_1}(t)), \\
		\partial _{\tau }f(t)& :=\partial _{\tau }f(t,\bar x_{t-K,t},\bar{y}(t),\bar{z}(t),\bar{u}_{\mu_1}(t)), \\
		h(T)& :=h(\bar{x}_{\mu_2}(T)), \\
		\partial _{x^1}h(T)& :=\partial _{x^1}h(\bar{x}_{\mu_2}(T)).
	\end{aligned}
\end{equation}
We stress that, all the above abbreviated  functions and partial derivatives  are evaluated at the ``optimal quadruple'' $(\bar x(\cdot),\bar{y}(\cdot),\bar{z}(\cdot), \bar u(\cdot))$.

Consider 
\begin{equation}  \label{e-hat-x}
	\left\{ \begin{aligned} d\hat{x}(t)=&\,\,
		\Big[A(t)\hat{x}(t)+\partial_xb(t)(\hat{x}_{t-K,t})+\partial_{v
		} b(t)\big(u_{\mu_1}(t)-\bar{u}_{\mu_1}(t)\big)\Big]dt \\ &
		+\Big[B(t)\hat{x}(t)+\partial_x\sigma(t)(\hat{x}_{t-K,t})+\partial_v \sigma(t)\big(u_{\mu_1}(t)-\bar{u}_{\mu_1}(t)\big)\Big]dw(t),\\&\hspace{2em}t\in \lbrack 0,T], \\ \hat{x}(t)=&\,\, 0,\quad t\in
		\lbrack -K ,0], \end{aligned}\right.
\end{equation}
and 
\begin{equation}  \label{e-hat-yz'}
	\left\{ \begin{aligned} -d\hat{y}(t)=
		&\,\,\Big[\partial _x f(t)(\hat x_{t-K,t})
		+\partial_{y}f(t)\hat{y}(t)+\left\langle \partial_{z}f(t),\hat{z}(t)\right\rangle_{\mathcal
			L_2^0( \mathcal K;\R)}\\& +\left\langle \partial_{v}f(t),u_{\mu_1}(t)-\bar{u}_{\mu_1}(t)\right\rangle _{H_{1}} \Big]dt 
		-\hat{z}(t)dw(t),\quad t\in \lbrack 0,T],\\ \hat{y}(T)= & \,\, \langle \partial_{x^1}h(T),\hat
		{x}_{\mu_2}(T)\rangle_H,\end{aligned} \right.
\end{equation}
which are the  variational equations  along the optimal  quadruple $(\bar{x}(\cdot ),\bar{y}(\cdot ),\bar{z}(\cdot ),\bar{u}(\cdot ))$ for \eqref{state} and 
\eqref{y}, respectively.

We take the following path derivative operators  evaluated at $(\bar x(\cdot ), \bar y(\cdot ), \bar z(\cdot ), \bar u(\cdot ))$ (see \eqref{e:extension}): for $t\in [0,T]$ and $Z\in C(-K,T;H)$, denote
\[ \rho_{b,t}(Z) :=\partial _{x}b(t)(Z_{t-K,t}),\  \rho _{\sigma,t}(Z):=\partial _{x}\sigma (t)(Z_{t-K,t}),\  \rho_{f,t}(Z) :=\partial _{x}f(t)(Z_{t-K,t}),\]
and $\rho
_{b}:=\big\{\rho_{b,t}(\cdot),~ t\in[0,T]\big\}$, $\rho
_{\sigma}:=\big\{\rho_{\sigma,t}(\cdot), ~t\in[0,T]\big\}$, $\rho
_{f}:=\big\{\rho_{f,t}(\cdot),~ t\in[0,T]\big\}$.
Under condition (H5), by Proposition \ref{Prop1} and Corollary \ref{cor1}, $\rho
_{b}$, $\rho _{\sigma}$ and $\rho _{f}$ are bounded linear
operators mapping from $L^{2}(-K,T;H)$ to $L^{2}(0,T;H)$, $L^{2}(0,T;\mathcal{L}_{2}^{0})$ and $L^{2}(0,T;\mathbb{R}),$   respectively, with a bound uniformly in $\omega$. 

Now, the variational equations \eqref{e-hat-x} and \eqref{e-hat-yz}  can be rewritten as
\begin{equation}  \label{e-variational}
	\left\{ \begin{aligned} d\hat{x}(t)=&\,\,
		\Big[A(t)\hat{x}(t)+\rho _{b,t}(\hat x_{t-K,t})+\partial_{v} b(t)\big(u_{\mu_1}(t)-\bar{u}_{\mu_1}(t)\big)\Big]dt \\ &
		+\Big[B(t)\hat{x}(t)+\rho _{\sigma,t }(\hat x_{t-K,t})+\partial_v \sigma(t)\big(u_{\mu_1}(t)-\bar{u}_{\mu_1}(t)\big)\Big]dw(t),\\&\hspace{2em}t\in \lbrack 0,T], \\ \hat{x}(t)=&\,\, 0,\quad t\in
		\lbrack -K ,0],\end{aligned}\right.
\end{equation}
and 
\begin{equation}  \label{e-hat-yz}
	\left\{ \begin{aligned} -d\hat{y}(t)=
		&\,\,\Big[\rho _{f,t}(\hat x_{t-K,t})
		+\partial_{y}f(t)\hat{y}(t)+\left\langle \partial_{z}f(t),\hat{z}(t)\right\rangle_{\mathcal
			L_2^0( \mathcal K;\R)}\\& +\left\langle \partial_{v}f(t),u_{\mu_1}(t)-\bar{u}_{\mu_1}(t)\right\rangle _{H_{1}} \Big]dt 
		-\hat{z}(t)dw(t),\quad t\in \lbrack 0,T],\\ \hat{y}(T)= & \,\, \langle \partial_{x^1}h(T),\hat
		{x}_{\mu_2}(T)\rangle_H,\end{aligned} \right.
\end{equation}

Assuming (H1)-(H4), {equation} \eqref{e-hat-yz} has a unique solution by the classical theory of BSDEs; for  the well-posedness of  \eqref{e-variational}, by Theorem \ref{exu} it suffices to verify (A4), which follows directly from the uniform boundedness of the linear operators $\partial_xb(t)$ and $\partial_x \sigma(t)$ assumed in (H3). Moreover, we can also show under (H1)-(H5) that the coefficient functions of \eqref{e-variational} satisfy (A4$\hspace{0.3mm}'$) in Remark~\ref{rem:A4'}   which also implies the well-posedness of \eqref{e-variational}:  for $(t,\omega )\in \lbrack 0,T]\times \Omega $ and $x,x'\in C(-K,T;H)$,
\begin{align*}
	&\int_{0}^{t}\Big\{\big\Vert \partial_xb(s)(x_{s-K,s})-\partial_xb(s)(x'_{s-K,s})\big\Vert
	_{H}^{2}+\big\Vert \partial_x\sigma(s)(x_{s-K,s})-\partial_x\sigma(s)(x'_{s-K,s})\big\Vert
	_{H}^{2}\Big\}ds\\
	& =\int_{0}^{t}\Big\{\Vert \rho_{b,s}(x-x')\Vert
	_{H}^{2}+\Vert \rho_{\sigma,s}(x-x')\Vert
	_{H}^{2}\Big\}ds\\
	&\leq C\int_{-K}^t\|x(s)-x'(s)\|_H^2ds,
\end{align*}
where the  inequality follows from  Proposition \ref{Prop1}.  

As $\rho$ goes to zero, $u^\rho(\cdot)$ converges to $\bar u(\cdot)$, and formal calculations
suggest that $x^\rho(\cdot)$ (resp. $(y^\rho(\cdot), z^\rho(\cdot))$) converges to $\bar x(\cdot)$
(resp. $(\bar y(\cdot), \bar z(\cdot))$) and $(x^\rho(\cdot)-\bar x(\cdot))/\rho$ (resp. $((y^\rho(\cdot)-\bar
y(\cdot))/\rho(\cdot), (z^\rho(\cdot)-\bar z(\cdot))/\rho)$) converges to the solution $\hat x(\cdot)$ of %
\eqref{e-hat-x} (resp. to the solution $(\hat y(\cdot), \hat z(\cdot))$ of \eqref{e-hat-yz}%
). This is justified by Lemma~\ref{Myle4-1} and Lemma~\ref{Le4-3} below.

\begin{lemma}
	\label{Myle4-1} Let {(H1)}-{(H4)} be satisfied. Then we have, as $\rho\to 0$, 
	\begin{align*}
		& \mathbb{E}\Big[ \sup\limits_{t\in[0,T]}\Vert x^{\rho }(t)-\bar{x}%
		(t)\Vert _{H}^{2}\Big] +\mathbb{E}\int_{0}^{T}\Vert x^{\rho }(t)-\bar{x}%
		(t)\Vert _{V}^{2}dt=O(\rho ^{2}); \\
		& \mathbb{E}\Big[ \sup\limits_{t\in[0,T]}\Vert x^{\rho }(t)-\bar{x}%
		(t)-\rho \hat{x}(t)\Vert _{H}^{2}\Big] +\mathbb{E}\int_0^T\Vert x^{\rho }(t)-%
		\bar{x}(t)-\rho \hat{x}(t)\Vert _{V}^{2}dt=o(\rho ^{2}).
	\end{align*}
\end{lemma}

\begin{proof}
	According to the state equation \eqref{state},
	\begin{equation*}
		\left\{ \begin{aligned} d(x^{\rho }(t)-\bar{x}(t))=& \,\,\Big[A(t)\big(x^{\rho
			}(t)-\bar{x}(t)\big)+\partial^\rho_x b(t)\big(x^{\rho }_{t-K,t}-\bar{x}_{t-K,t}\big)
		\\&\hspace{3em}	+\partial^\rho_vb(t)\rho \big(u_{\mu_1}(t)-\bar{u}_{\mu_1
			}(t)\big)\Big]dt \\ & \hspace{0.5em}+\Big[B(t)(x^{\rho
			}(t)-\bar{x}(t))+\partial^\rho_x\sigma(t)\big(x^{\rho }_{t-K,t}-\bar{x}_{t-K,t}\big)\\
			&\hspace{3em} +\partial^\rho_v\sigma(t)\rho \big(u_{\mu_1}
			(t)-\bar{u}_{\mu_1}(t)\big)\Big]dw(t),\quad t\in \lbrack 0,T], \\ x^\rho
			(t)-\bar x(t)=& \,\,0,\quad t\in \lbrack -K ,0], \end{aligned}\right. 
	\end{equation*}%
	where, for $\varphi =b,\sigma $ and $\tau =x,v$, we denote
	\begin{equation*}
		{\partial^\rho_\tau} \varphi(t):=\int_{0}^{1}\partial _{\tau }\varphi 
		\Big(t,\bar x_{t-K,t}+\lambda (x_{t-K,t}^{\rho }-\bar{x}_{t-K,t}),\bar{u}_{\mu_1}(t)+\lambda \rho (u_{\mu_1}(t)-\bar{u}_{\mu_1}(t))\Big)d\lambda .
	\end{equation*}
	By Theorem \ref{estimate}, we derive that 
	\begin{align*}
		& \mathbb{E}\Big[\sup\limits_{t\in[0,T]}\Vert x^{\rho }(t)-\bar{x}
		(t)\Vert _{H}^{2}\Big]+\mathbb{E}\int_{0}^{T}\Vert x^{\rho }(t)-\bar{x}%
		(t)\Vert _{V}^{2}dt \\
		& \leq C\rho ^{2}\mathbb{E}\int_{0}^{T}\Vert u_{\mu_1}(t)-\bar{u}_{\mu_1}(t)\Vert _{H_{1}}^{2}dt\leq C\rho ^{2},
	\end{align*}%
	which proves the first equality.
	
	Setting for $t\in[-K,T]$,
	$$\tilde{x}^{\rho } (t)=\frac{x^{\rho }(t)-\bar{x}(t)}{\rho }-\hat{x}
	(t),$$ we have 
	\begin{equation}
		\left\{ \begin{aligned} d\tilde{x}^{\rho }(t)=&\,\,
			\Big[A(t)\tilde{x}^{\rho }(t)+\partial^\rho_x b(t)(\tilde{x}^{\rho
			}_{t-K,t})+\big[\partial^\rho_x b(t)-\partial_x b(t)\big](\hat{x}_{t-K,t}) \\ &
			\hspace{1em}+\big[\partial^\rho_v b(t)-\partial_vb(t)\big]\big(u_{\mu_1}(t)-\bar{u}_{\mu_1}(t)\big)\Big]dt \\ &
			+\Big[B(t)\tilde{x}^{\rho }(t)+\partial^\rho_x\sigma(t)(\tilde{x}^{\rho
			}_{t-K,t})+\big[\partial^\rho_x\sigma(t)-\partial_x \sigma(t)\big](\hat{x}_{t-K,t}) \\ &
			\hspace{1em}+\big[\partial^\rho_v\sigma(t)-\partial_v\sigma(t)\big]\big(u_{\mu_1}(t)-\bar{u}_{\mu_1}(t)\big)\Big]dw(t), \\
			\tilde{x}^{\rho }(t)=&\,\, 0,\quad t\in \lbrack -K ,0]. \end{aligned}\right.
	\end{equation}%
	Utilizing the {\it a priori} estimate \eqref{es-x}, we have 
	\begin{align*}
		& \mathbb{E}\Big[\sup\limits_{t\in[0,T]}\big\| \tilde{x}^{\rho }(t)\Vert
		_{H}^{2}\Big]+\mathbb{E}\int_0^T\Vert \tilde{x}^{\rho }(t)\Vert _{V}^{2}dt \\
		&\leq C\mathbb{E}\int_{0}^{T}\Big\{\Big\| \big[\partial^\rho_xb
		(t)-\partial_x b(t)\big](\hat{x}_{t-K,t}) +[\partial^\rho_v b
		(t)-\partial_vb(t)]\big(u_{\mu_1}(t)-\bar{u}_{\mu_1}(t)\big) \Big\|
		_{H}^{2} \\
		& \hspace{2em}+\Big\| \big[ \partial^\rho_x\sigma(t)-\partial_x
		\sigma(t)\big](\hat{x}_{t-K,t}) +[\partial^\rho_v\sigma
		(t)-\partial_v\sigma(t)]\big(u_{\mu_1}(t)-\bar{u}_{\mu_1}(t)\big)%
		\Big\|_{H}^{2}\Big\}dt.
	\end{align*}%
	Then the dominated convergence theorem yields 
	\begin{equation*}
		\lim\limits_{\rho \to 0}\mathbb{E}\Big[\sup\limits_{t\in[0,T]}\Vert \tilde{x}^{\rho }(t)\Vert _{H}^{2}\Big]+\lim\limits_{\rho \to 0}\mathbb{E}\int_0^T\Vert \tilde{x}^{\rho }(t)\Vert _{V}^{2}dt=0,
	\end{equation*}
	and the second equality follows.
\end{proof}

\begin{lemma}
	\label{Le4-3} Assume (H1)-(H4) hold. Then, as 
	$\rho\to0$, 
	\begin{equation*}
		\mathbb{E}\Big[\sup\limits_{t\in[0,T]}|y^{\rho }(t)-\bar{y}(t)-\rho \hat{
			y}(t)| ^{2}\Big]+\mathbb{E}\int_{0}^{T}\|z^{\rho }(t)-\bar{z}(t)-\rho \hat{z}%
		(t)\|^{2}_{\mathcal{L}_2^0(\mathcal K,\mathbb{R})}dt =o(\rho ^{2}). 
	\end{equation*}
\end{lemma}

\begin{proof}
	It suffices to show
	\begin{equation*}
		\lim\limits_{\rho \rightarrow 0}\mathbb{E}\Big[ \sup\limits_{t\in[0,T]}\left\vert \tilde y^\rho(t)\right\vert ^{2}
		\Big] +\lim\limits_{\rho \rightarrow 0}\mathbb{E}\int_{0}^{T}\left\|\tilde z^\rho(t)\right\|^{2}_{\mathcal{L}_2^0(\mathcal K,\mathbb{R})}dt=0,
	\end{equation*}
	where
	\begin{equation*}
		\tilde{y}^{\rho }(t):=\frac{y^{\rho }(t)-\bar{y}(t)}{\rho }-\hat{y}(t)\ \text{and}\ \tilde{z}^{\rho }(t):=\frac{z^{\rho }(t)-\bar{z}(t)}{\rho }-\hat{z}(t), \quad t\in[0,T].
	\end{equation*}
	The pair $(\tilde y^\rho(\cdot), \tilde z^\rho(\cdot))$ solves the following BSDE \begin{align*}
		\left\{ \begin{aligned} -d\tilde{y}^{\rho }(t)=&
			\,\,\Big\{\partial^\rho_{x}f(t)(\tilde{x}^{\rho
			}_{t-K,t})+\big[\partial^\rho_x f(t)-\partial_xf(t)\big](\hat{x}_{t-K,t})+\partial^\rho_{y} f(t)\tilde{y}^{\rho
			}(t)\\&\hspace{0.5em}+\big[\partial^\rho_{y}f(t)-\partial_{y}f(t)\big]\hat{y}(t)+\big<
			\partial^\rho_{z} f(t),\tilde{z}^{\rho }(t)\big>_{\mathcal L_2^0(\mathcal K;\R)} \\ &\hspace{0.5em}+\big<\partial^\rho_{z}f(t)-\partial_{z}f(t),\hat{z}(t)\big>_{\mathcal
				L_2^0(\mathcal K;\R)}\\&\hspace{0.5em}+\big<\partial^\rho_v f(t)-\partial_vf(t),u_{\mu_1}(t)-\bar{u}_{\mu_1}(t)\big>_{H_{1}}\Big\}dt-\tilde{z}^{\rho }(t)dw(t),\quad t\in [0,T],
			\\ \tilde{y}^{\rho }(T)=& \,\,\lla\partial^\rho_{x^1}h(T),\tilde{x}_{\mu_2}^{\rho
			}(T)\rra_{H}+\lla\partial^\rho_{x^1}h(T)-\partial_{x^1}h(T),
			\hat{x}_{\mu_2}(T)\rra_{H},
		\end{aligned}\right.
	\end{align*}
	where
	\begin{align*}
		\partial^\rho_{x^1}h(T):=&\int_{0}^{1}\partial_{x^1}h
		\Big(\bar{x}_{\mu_2}(T)+\lambda
		\big(x_{\mu_2}^{\rho}(T)-\bar{x}_{\mu_2}(T)\big)\Big)d\lambda ,
	\end{align*}
	and for $\tau =x,y,z,v$, 
	\begin{align*}
		\partial^\rho_{\tau }f(t):=& \int_{0}^{1}\partial_{\tau }f\Big(t,\bar{x}_{t-K,t}+\lambda
		\big(x^{\rho }_{t-K,t}-\bar{x}_{t-K,t}\big), \bar{y}(t)+\lambda \big(y^{\rho }(t)-\bar{y}(t)\big),\\
		&\hspace{3em}\bar{z}(t)+\lambda \big(z^{\rho }(t)-\bar{z}(t))\big),\bar{u}_{\mu_1}(t)+\lambda \rho \big(u_{\mu_1}(t)-\bar{u}_{\mu_1}(t)\big)\Big)d\lambda .
	\end{align*}%
	By the {\it a priori} estimate for classical BSDEs and Lemma \ref{Myle4-1}, we have 
	\begin{align*}
		& \mathbb{E}\Big[\sup\limits_{t\in[0,T]}|\tilde{y}^{\rho }(t)|^{2}\Big]+%
		\mathbb{E}\int_{0}^{T}\|\tilde{z}^{\rho }(t)\|^{2}_{\mathcal{L}_2^0(\mathcal K,%
			\mathbb{R})}dt \\
		& \leq C\Big\{\mathbb{E}\int_{0}^{T}\Big|\partial^\rho_{x}f(t)(\tilde{x}^{\rho
		}_{t-K,t})+\big[\partial^\rho_x{f}(t)-\partial_xf(t)\big](\hat{x}_{t-K,t})+\big[\partial^\rho_{y}{f}(t)-\partial_{y}f(t)\big]\hat{y}(t) \\ &\hspace{4em}+\big<\partial^\rho_{z}{f}(t)-\partial_{z}f(t),\hat{z}(t)\big>_{\mathcal
			L_2^0(\mathcal K;\R)}+\big<\partial^\rho_v{f}(t)-\partial_vf(t),u_{\mu_1}(t)-\bar{u}_{\mu_1
		}(t)\big>_{H_{1}}\Big|^2dt \\
		& \hspace{3em}+\mathbb{E}\Big[\Vert \tilde{x}^{\rho }(T)\Vert _{H}^{2}+ \Big|\big<\partial^\rho_{x^1}{h}(T)-\partial_{x^1}h(T),%
		\hat{x}_{\mu_2}(T)\big>_{H}\Big|^2 \Big]\Big\}\rightarrow 0, \text{ as } \rho \to 0.
	\end{align*}%
	The proof is complete.
\end{proof}

\subsection{Maximum principle}

In this subsection, we introduce the anticipative adjoint equation and then derive the stochastic maximum principle.

Let $k(\cdot)$ solve  the following adjoint SDE associated with the cost functional $
y(\cdot)$: 
\begin{equation}
	\begin{cases}
		dk(t)=\partial_{y}f(t)k(t)dt+\partial_{z}f(t)k(t)dw(t),\quad t\in \lbrack
		0,T], \\ 
		k(0)=-1.%
	\end{cases}
	\label{k}
\end{equation}
We consider the adjoint BSEE
\begin{equation}\label{e-adjoint}
	\left\{ \begin{aligned} p(t)=&
		-\int_{I_t}\mathbb{E}^{\mathcal{F}_{s}}\big[k(T)
		\partial_{x^1}h(T)\big]\mu_2(d(s-T))+\int_t^T\Big\{A^{\ast }(s)p(s)+B^{\ast
		}(s)q(s)\\&\hspace{1em}+\mathbb{E}^{\mathcal{F}_{s}}\Big[%
		\rho^*_{b,s}\big( p_{s,s+K}|_{[0,T]}\big)+\rho^*_{\sigma,s}\big( q_{s,s+K}|_{[0,T]}\big)-\rho^*_{f,s}\big( { k_{s,s+K}|_{[0,T]}}\big)\Big]
		\Big\}ds\\ &-\int_t^Tq(s)dw(s),\quad t\in \lbrack 0,T], \\ p(t)=&
		0,\,q(t)=0,\quad t\in (T,T+K ].
	\end{aligned}\right. 
\end{equation}
Here,   $I_t:=(t,T]\cap[T-K,T]$ for $t\in[0,T]$,  $\rho_b^*=\big\{\rho
_{b,t}^{\ast }(\cdot),~t\in [ -K,T]\big\}$, $\rho_\sigma^*=\big\{\rho
_{\sigma,t}^{\ast }(\cdot),~t\in [ -K,T]\big\},  \rho_f^*=\big\{\rho
_{f,t}^{\ast }(\cdot),~t\in [ -K,T]\big\}$ are the adjoint operators of $\rho_b$, $\rho_\sigma$, $\rho_f$, respectively (recalling Proposition \ref{thm:rho*}). Clearly the BSEE \eqref{e-adjoint} is {anticipative} in the sense of \cite{10py} (see also Remark~\ref{rem:rho*}). 

\begin{remark}\label{Rem5-1}  Assume condition (H5). By Proposition~\ref{thm:rho*},  for any fixed $t\in[0,T]$, $\rho_b^*$, $\rho_\sigma^*$ and  $\rho_f^*$ are bounded operators from  $L^{2}(t,T;H)$, $L^{2}(t,T;\mathcal{L}_{2}^{0})$ and $L^{2}(t,T;\mathbb{R})$ to $L^2(t,T; H)$,  respectively. Thus, for $(t,\omega )\in \lbrack 0,T]\times \Omega $, $p,p'\in L^2(0,T+K;H)$, and $q,q'\in L^2(0,T+K;\mathcal{L}_{2}^{0})$, we have from \eqref{e:rho*-rem} that
	\begin{align*}
		&\int_t^T\Big\{\big\|\rho_{b,s}^*(p_{s,s+K}|_{[0,T]})-\rho_{b,s}^*(p'_{s,s+K}|_{[0,T]})\big\|^2_H+\big\|\rho_{ \sigma,s}^*(q_{s,s+K}|_{[0,T]})-\rho_{\sigma,s}^*(q'_{s,s+K}|_{[0,T]})\big\|^2_H\Big\}ds\\&=\int_t^T\Big\{\big\|\rho_{b,s}^*(p|_{[0,T]})-\rho_{b,s}^*(p'|_{[0,T]})\big\|^2_H+\big\|\rho_{\sigma,s}^*(q|_{[0,T]})-\rho_{\sigma,s}^*(q'|_{[0,T]})\big\|^2_H\Big\}ds 
		\\&\leq C\int_t^{T}\Big\{\|p(s)-p'(s)\|_H^2+\|q(s)-q'(s)\|^2_{\mathcal L_2^0}\Big\}ds
		\\&\leq  C\int_t^{T+K}\Big\{\|p(s)-p'(s)\|_H^2+\|q(s)-q'(s)\|^2_{\mathcal L_2^0}\Big\}ds,
	\end{align*}
	which verifies {(B4)}. 
	Then by Theorem \ref{ABSEE-thm}, equation \eqref{e-adjoint} admits a unique solution $(p(\cdot),q(\cdot))\in  \mathscr P\times L_{\mathbb{F%
	}}^{2}(0,T+K; \mathcal{L}_{2}^{0})$.
\end{remark}

The \emph{Hamiltonian} $H:[0,T] \times\Omega\times    { C(-K,T;H)}\times \mathbb{R}
\times \mathcal{L}_{2}^{0}(\mathcal K;\mathbb{R})\times H_{1}\times
H\times \mathcal{L}_{2}^{0}\times \mathbb{R}\rightarrow \mathbb{R}$ is defined by
\begin{equation}
	\begin{aligned} H(t,\omega,x,y,z,v,p,q,k):=&\left\langle b(t,x,v),p\right\rangle _{H}+\lla \sigma (t,x,v),q\rra
		_{\mathcal{L}_{2}^{0}} -f(t,x,y,z,v)k.\end{aligned}
	\label{Hamiltionian}
\end{equation}
Denote
$$
H(t)=H(t,\bar x_{t-K,t},\bar y(t), \bar z(t),\bar u_{\mu_1}(t),p(t),q(t),k(t)),$$
and  for $\tau=x,y,z,v$,
$$\partial_{\tau}H(t)=\partial_{\tau}H(t,\bar x_{t-K,t},\bar y(t), \bar z(t),\bar u_{\mu_1}(t),p(t),q(t),k(t)),$$
where $(p(\cdot),q(\cdot))$ solves \eqref{e-adjoint}, and $k(\cdot)$ solves \eqref{k}.

We are ready to derive the stochastic maximum principle for our control problem.
\begin{theorem}
	\label{SMP} Suppose that  (H1)-(H5) hold. Let $\bar{u}(\cdot )$ be an optimal
	control for the control problem, $\bar{x}(\cdot )$ and $(\bar{y}(\cdot ),%
	\bar{z}(\cdot ))$ be the corresponding solutions to \eqref{state} and 
	\eqref{y}, respectively. Assume that $(p(\cdot ),q(\cdot ))$ is the solution
	of \eqref{e-adjoint} with $k(\cdot )$ being the solution of $\eqref{k}$. Then, 
	\begin{equation}
		\Big<\mathbb{E}^{\mathcal{F}_{t}}\Big[\int_{-K
		}^{0}\partial_vH(t-s)\mu_1(ds)\Big],u-\bar{u}(t)\Big>_{H_{1}}\geq 0,
		\label{ne}
	\end{equation}
	holds for all $u\in U$ and $dt\times dP$-almost all $(t,\omega)\in[0,T]%
	\times \Omega$.
\end{theorem}

\begin{proof}
	Recalling \eqref{e-variational} and \eqref{e-adjoint}, applying It\^{o}'s formula to $\left\langle p(t),\hat{x}(t)\right\rangle _{H}$, and then taking
	expectation, we have 
	\begin{align*}
		& -\mathbb{E }\int_{T-K}^T\big<
		\partial_{x^1}h(T)k(T),\hat x(t)\big>_H\mu_2(d(t-T)) \\
		&=\mathbb{E}\int_{0}^{T}\Big\{\Big<\rho_{b,t}(\hat x_{t-K,t})+\partial_{v
		} b(t)\big(u_{\mu_1}(t)-\bar{u}_{\mu_1}(t)\big),p(t)\big>_{H} \\
		&\hspace{4em} -\mathbb{E}^{\mathcal{F}_{t}}\Big[\rho^*_{b,t}\big(p_{t,t+K}|_{[0,T]}\big)+\rho^*_{\sigma,t}\big( q_{t,t+K}|_{[0,T]}\big)-\rho^*_{f,t}\big(k_{t,t+K}|_{[0,T]}\big)\Big],\hat{x}(t)\Big>_{H} \\
		& \hspace{4em}+\big<\rho_{\sigma,t}(\hat x_{t-K,t})+\partial_{v
		} \sigma(t)\big(u_{\mu_1}(t)-\bar{u}_{\mu_1}(t)\big),q(t)\big>_{{\mathcal{L}_2^0}}\Big\}dt.
	\end{align*}
	Noting $\hat{x}(t) =0$ for $t\in[-K, 0]$, by \eqref{e:dualilty},  \eqref{non-anti}  and Remark \eqref{rem:rho*}  we get the following equalities: 
	\begin{align*}
		\E\int_0^T\big<\rho_{b,t}(\hat x_{t-K,t}),p(t)\big>dt&=\E\int_0^T\Big<\E^{\mathcal F_t}\big[\rho^*_{b,t}(p_{t,t+K}|_{[0,T]})\big],\hat x(t)\Big>dt,\\
		\E\int_0^T\big<\rho_{\sigma,t}(\hat x_{t-K,t}),q(t)\big>dt&=\E\int_0^T\Big<\E^{\mathcal F_t}\big[\rho^*_{\sigma,t}(q_{t,t+K}|_{[0,T]})\big],\hat x(t)\Big>dt,\\
		\E\int_0^T\rho_{f,t}(\hat x_{t-K,t})k(t)dt&=\E\int_0^T\Big<\E^{\mathcal F_t}\big[\rho^*_{f,t}(k_{t,t+K}|_{[0,T]})\big],\hat x(t)\Big>dt.
	\end{align*}
	Similarly, 
	noting that $u(t)-\bar u(t)=0$ for $t\in[-K,0]$,  we also have
	\begin{align*}
		&\E\int_0^T\big<\partial_{v
		} b(t)(u_{\mu_1}(t)-\bar{u}_{\mu_1}(t)),p(t)\big>_{H}dt\\
		&=\E\int_0^T\Big<\int_{-K}^{0}(\partial_vb(t-s))^{\ast }p(t-s) \mathbb I_{[t-T,t]}(s)\mu_1(ds), u(t)-\bar u(t)\Big>_{H_1}dt\\
		&=\E\int_0^T\Big<\int_{-K}^{0}(\partial_vb(t-s))^{\ast }p(t-s)\mu_1(ds), u(t)-\bar u(t)\Big>_{H_1}dt,
	\end{align*}
	where the first step follows from Remark  \ref{Rem:dual-for-u} and the second step  is due to the fact $p(t)=0$ for $t\in(T, T+K]$. 
	Consequently, recalling the notation given by \eqref{e:x=dv},  we have
	\begin{equation}
		\begin{split}
			& -\mathbb{E }[k(T) \langle \partial_{x^1}h(T), \hat x_{\mu_2} (T)\rangle_H]\\&=\mathbb{E}\int_{0}^{T}\Big\{\big<\E^{\mathcal F_t}\big[\rho^*_{f,t}({ k_{t,t+K}|_{[0,T]}})\big],\hat{x}(t)\big>_{H}\\&\hspace{3em}+\Big<\mathbb{E}^{\mathcal{F}_{t}}\Big[\int_{-K}^{0}(\partial_vb(t-s))^{\ast }p(t-s)\mu_1(ds)\\&\hspace{3em} +\int_{-K}^{0}(\partial_v\sigma(t-s)) ^{\ast
			}q(t-s)\mu_1(ds)\Big],u(t)-\bar{u}(t)\Big>_{H_{1}}\Big\}dt.
		\end{split}
		\label{eq:32}
	\end{equation}
	Applying It\^{o}'s formula to $k(t)\hat{y}(t)$ on $[0,T]$, we obtain by
	Proposition \ref{thm:rho*} that 
	\begin{align}\label{eq:33}
			\hat{y}(0)&=-\mathbb{E}\big[k(T)\hat y(T)\big]  -\mathbb{E}\int_{0}^{T}\Big\{\rho _{f,t}(\hat x_{t-K,t})-\left\langle \partial_vf(t),u_{\mu_1}(t)-\bar{u}
			_{\mu_1}(t)\right\rangle _{H_{1}}\Big\}k(t)dt\nonumber\\
			& =-\mathbb{E}\big[k(T)\lla\partial_{x^1}h(T),\hat{x}
			_{\mu_2}(T))\rra _{H} \big]-\mathbb{E}\int_{0}^{T}\Big\{\rho _{f,t}(\hat x_{t-K,t})\\
			&\hspace{3em}-\Big<\mathbb{E}^{\mathcal{F}
				_{t}}\Big[\int_{-K}^{0}(\partial_vf(t-s))^*k(t-s)\mu_1(ds)\Big],u(t)-\bar{u}(t)\Big>
			_{H_{1}}\Big\}dt. \nonumber
		\end{align}

	Then, it follows from  \eqref{eq:32} and \eqref{eq:33} that 
	\begin{equation}\label{Myeq4-11}
		\begin{aligned}  
			\hat{y}(0)=& \mathbb{E}\int_{0}^{T}\Big<\mathbb{E}^{\mathcal{F}_{t}}\Big[\int_{-K
			}^{0}(\partial_vb(t-s))^{\ast }p(t-s)\mu_1(ds) \\
			& \hspace{3em}+\int_{-K }^{0}(\partial_v\sigma(t-s))
			^{\ast }q(t-s)\mu_1(ds) \\
			& \hspace{3em}-\int_{-K }^{0}(\partial_vf(t-s))^*k(t-s)\mu_1(ds)\Big],u(t)-%
			\bar{u}(t)\Big>_{H_{1}}dt.  
		\end{aligned}
	\end{equation}
	On the other hand, by Lemma \ref{Le4-3} and the optimality of  $\bar{u}(\cdot)$, we get 
	\begin{equation*}
		0\leq J(u^{\rho }(\cdot ))-J(\bar{u}(\cdot ))=\rho \hat{y}(0)+o(\rho ).
	\end{equation*}
	This together with (\ref{Myeq4-11}) implies 
	\begin{equation*}
		\hat{y}(0)=\mathbb{E}\Big[ \int_{0}^{T}\Big<\mathbb{E}^{\mathcal{F}
			_{t}}\Big[\int_{-K}^{0}\partial_vH(t-s)\mu_1(ds)\Big],u(t)-\bar{u}(t)
		\Big>_{H_{1}}dt\Big] \geq 0,
	\end{equation*}
	from which we obtain the maximum principle \eqref{ne}.
\end{proof}

\begin{remark}
	In the above proof, $\hat{x}(t)$ is continuous whereas $p(t)$ may not be. So
	the possible jumps of $p(t)$ do not contribute when applying It\^{o}'s formula
	to $\left\langle p(t),\hat{x}(t)\right\rangle_{H}$.
\end{remark}

\begin{remark}\label{Integral-delay-rem}
	Equations \eqref{state} and \eqref{y} are path-dependent in the sense that the coefficients $b,\sigma, f$ depend on the past trajectories of $x$ on $[t-K,t]$  at the present time  $t\in [0,T]$. One typical  path dependence is of the form of an integral with respect to a finite measure.  More specifically,  set  $b(t,x_{t-K,t},u_{\mu_1}(t))=\tilde b(t,x_{\mu}(t),u_{\mu_1}(t)),
	\sigma(t,x_{t-K,t},u_{\mu_1}(t))=\tilde \sigma(t,x_{\mu}(t),u_{\mu_1}(t))$, and $
	f(t,x_{t-K,t},u_{\mu_1}(t),y(t),z(t),u_{\mu_1}(t))=\tilde f(t,x_{\mu}(t),u_{\mu_1}(t),y(t),z(t),\\u_{\mu_1}(t))$,
	where 
	\begin{equation*}
		(\tilde b,\tilde \sigma ): \lbrack 0,T]\times \Omega \times  H\times
		H_{1}\rightarrow H\times \mathcal{L}_{2}^{0},\ \tilde f:[0,T]\times \Omega
		\times H\times \mathbb{R}\times \mathcal{L}_{2}^{0}(\mathcal K,\R)\times
		H_{1}\rightarrow \mathbb{R}
	\end{equation*}
 satisfies standard Lipschitz continuity, measurability, integrability  and  differentiability assumptions, and
	$$x_\mu(t):=\int_{-K}^0x(t+s)\mu(ds)$$
	is an integral delay with respect to a finite measure $\mu$  on $[-K,0]$.  Then, the adjoint equation $\eqref{e-adjoint}$ becomes, in view of  Example \ref{example1} and Remark \ref{Rem:dual-for-u},
	\begin{equation*}
		\left\{ 
		\begin{aligned}
			p(t)=& -\int_{I_t}\mathbb{E}^{\mathcal{F}_{s}}
			\big[k(T)\partial_x h(T)\big]\mu_2(d(s-T))\\
			&+\int_t^T\Big\{A^{\ast }(s)p(s)+B^{\ast }(s)q(s)+\mathbb{E}^{\mathcal{F}_{s}}\Big[\int_{-K}^{0}(\partial_x\tilde b(s-r))^{\ast
			}p(s-r)\mu(dr)\\
			& \hspace{1.5em}+\int_{-K}^{0}(\partial_x\tilde\sigma (s-r))^{\ast
			}q(s-r)\mu(dr) -\int_{-K}^{0}\partial_x\tilde f(s-r)k(s-r)\mu(dr)\Big]
			\Big\}ds  \\
			&-\int_t^Tq(s)dw(s),\quad t\in \lbrack 0,T],\\
			p(t)=& 0,\,q(t)=0,\quad t\in (T,T+K],
		\end{aligned}
		\right. 
	\end{equation*}
	and we can apply Theorem~\ref{SMP} to get the maximum principle.  
	
	Existing literature primarily considers the following  two special forms of  integral delay  (see, e.g., \cite{10cw,11osz,guatteri2021stochastic,16ms} for more details): (a) the pointwise delay $x_{\mu}(t)=x(t-K)$ when $\mu(ds)$ is the Dirac delta measure at $-K$; (b) the moving average delay $x_\mu(t)=\int_{-K}^0 x(t+s) ds$ with respect to the Lebesgue measure $\mu(ds)=ds$.  
\end{remark}

\begin{remark}\label{rem:time-state dual}
	The dual analysis  in our system is performed for pointwise $\omega.$ The duality analysis in \cite{hu1996maximum}   involves expectation and conditional expectation, due to which the  coefficient functions therein were assumed to be deterministic.
\end{remark}
\begin{remark} 
	\label{Extension Rem} In view of the results on PSEEs and ABSDEs established
	in the previous sections, some straightforward adaptions of the proof of
	Theorem \ref{SMP} shall yield a variety of extensions. We list some 
	directions below.
	
	\begin{itemize}
		\item[(i)] The delay measure $\nu_1$ appearing in the SEE \eqref{EQ-1} and
		the  BSDE \eqref{BSDE} can be distinct.

		\item[(ii)]   The measures  $\nu_1,\nu_2$ can be extended to finite \emph{signed} measures on $[T-K,T],$ and furthermore, they can be $\mathbb{R}^d$-valued, for any integer $d>1$.

	\end{itemize}
\end{remark}

\subsection{Sufficient conditions}

In this subsection, we will show that the necessary condition \eqref{ne} for
an optimal control is also sufficient under some convexity conditions.

\begin{theorem}
	\label{sufficient} Suppose that (H1)-(H5) hold. Let $\bar u(\cdot)\in 
	\mathcal{U}$ and $\bar{x}(\cdot )$ and $(\bar{y}(\cdot ),\bar{z}(\cdot ))$
	be the corresponding solutions of \eqref{state} and \eqref{y}, respectively.
	Assume
	
	\begin{enumerate}
		\item[(a)] $h(\cdot)$ is   convex;
		
		\item[(b)] the Hamiltonian $H$ given in \eqref{Hamiltionian} is convex  for each $
		(t,\omega,p,q,k)$ in the sense that for $(x,y,z,v), (x',y',z',v')\in C(-K,T;H)\times \R\times \mathcal L_2^0(\mathcal K;\R)\times H_1$,
		\begin{align*}
			&H(t,\omega,x,y,z,v,p,q,k)-H(t,\omega,x',y',z',v',p,q,k)\\&\geq \partial_x H(t,\omega,x',y',z',v',p,q,k)(x-x')+ \partial_yH(t,\omega,x',y',z',v',p,q,k)(y-y')\\&+\partial_z H(t,\omega,x',y',z',v',p,q,k)(z-z')+ \partial_vH(t,\omega,x',y',z',v',p,q,k)(u-u').
		\end{align*}
		\item[(c)] \eqref {ne} holds for all $u\in U$, a.e., a.s.
	\end{enumerate}
	
	Then $\bar u(\cdot)$ is an optimal control.
\end{theorem}




\begin{proof}
	For an arbitrarily chosen control process $u(\cdot )\in \mathcal{U%
	}$, let $x^{u}(\cdot )$ and $(y^{u}(\cdot ),z^{u}(\cdot ))$ be the
	corresponding solutions of \eqref{state} and \eqref{y}, respectively. We
	denote, for $t\in \lbrack 0,T]$, 
	\begin{align*}
		b^{u}(t)=& b(t,x^{u}_{t-K,t},u_{\mu_1}(t)), \\
		\sigma ^{u}(t)=& \sigma (t,x^{u}_{t-K,t},u_{\mu_1}(t)),
		\\
		f^{u}(t)=& f(t,x^{u}_{t-K,t},y^{u}(t),z^{u}(t),u_{
			\mu_1}(t)).
	\end{align*}
	Applying It\^{o}'s formula to $k(t)(y^{u}(t)-\bar{y}(t))$ and $\left\langle
	p(t),x^{u}(t)-\bar{x}(t)\right\rangle _{H}$ on $[0,T]$, we can derive  that 
	\begin{align*}
		& \mathbb{E}\big[k(T)\big(h(x^{u}_{\mu_2}(T))-h(\bar{x}_{\mu_2}(T))\big)\big]+y^{u}(0)-\bar{y}(0) -\mathbb{E}\big[k(T)\big\langle \partial_{x^1}h(T),x^{u}_{\mu_2}(T)-\bar{x}_{\mu_2}
		(T)\big\rangle _{H}\big]\\
		& =\mathbb{E}\int_{0}^{T}\Big\{f_{y}(t)k(t)\big(y^{u}(t)-\bar{y}
		(t)\big)+\big<f_{z}(t)k(t),z^{u}(t)-\bar{z}(t)\big>_{\mathcal{L}%
			_2^0(K;\mathbb{R})}\\
		& \hspace{2em}+\big<b^{u}(t)-b(t),p(t)\big>_{H}+\big<\sigma ^{u}(t)-\sigma
		(t),q(t)\big>_{ \mathcal{L}_2^0}-\big(f^{u}(t)-f(t)\big)k(t) \\
		& \hspace{2em}-\Big<\mathbb{E}^{\mathcal{F}_{t}}\Big[\rho^*_{b,t}\big(p_{t,t+K}|_{[0,T]}\big)+\rho^*_{\sigma,t}\big(q_{t,t+K}|_{[0,T]}\big)-\rho^*_{f,t}\big(k_{t,t+K}|_{[0,T]}\big)\Big]
		,x^{u}(t)-\bar{x}(t)\Big>_{H}\Big\}dt \\
		& =\mathbb{E}\int_{0}^{T}\Big\{H^{u}(t)-H(t)-\partial_xH(t)\big(x^{u}_{t-K,t}-\bar{x}_{t-K,t}\big) -\partial_yH(t)
		\big(y^{u}(t)-\bar{y}(t)\big) 
		\\
		& \hspace{6em}-\big<\partial_zH(t),z^{u}(t)-\bar{z}(t)\big>_{\mathcal{L}_2^0(K;\mathbb{R})}
	\Big\}dt
		=: A,
	\end{align*}
	where the second equality follow from the definition of $H, \rho^*_b,\rho^*_\sigma,\rho^*_f$ and the duality relationship \eqref{e:dualilty}. 
	By the convexity of $H$, we know that 
	\begin{equation}  \label{e:le01}
		A - B \ge 0
	\end{equation}
	where 
	\begin{equation}\label{e:le02}
		\begin{aligned}  
			B=
			&\mathbb{E}\int_{0}^{T}\big<\partial_vH(t), u_{\mu_1}(t)-\bar{u}_{\mu_1}(t)\big>_{H_{1}}dt \\
			=&\mathbb{E}\int_{0}^{T}\Big<\int_{-K}^{0}
			\partial_vH(t-s)\mu_1(ds),u(t)-\bar{u}(t)\Big>_{H_{1}}dt\ge 0, 
		\end{aligned}
	\end{equation}
	with the nonnegativity following from the assumption \eqref{ne}. Therefore, from the convexity of $h$, the
	fact that $k(T)\le 0$, \eqref{e:le01} and \eqref{e:le02},
	\begin{align*}
		y^{u}(0)-\bar{y}(0) = & (A-B)+B -\mathbb{E}\big[k(T)\big(h(x^{u}_{\mu_2}(T))-h(\bar{x}_{\mu_2}(T))\big)\big] \\
		&\hspace{0.5em}+\mathbb{E}\big[k(T)\big\langle \partial_{x^1}h(T),x^{u}_{\mu_2}(T)-\bar{x}_{\mu_2}
		(T)\big\rangle _{H}\big] \ge 0.
	\end{align*}
	This shows $%
	J(u(\cdot ))-J(\bar{u}(\cdot ))\geq 0$ which yields the optimality of $\bar{u%
	}(\cdot )$.
\end{proof}

\section{Some applications}

\label{sec:APP}

In this section, we apply our result to the optimal control problem of parabolic
SPDEs and the linear quadratic (LQ) problem of SEE. 

\subsection{Optimal control problem of path-dependent SPDEs}

Let $H^{1}$ be the Sobolev space of $W^{1,2}(\mathbb{R}^{d}).$ Set $V=H^{1} $
and $H=L^{2}(\mathbb{R}^{d}).$ Consider the super-parabolic  path-dependent SPDE: 
\begin{equation*}
	\left\{ \begin{aligned} 
		\frac{\partial x(t,\zeta )}{\partial t}& =\sum_{i,j=1}^{n}\partial
		_{\zeta _{i}}[\alpha _{ij}(t,\zeta )\partial _{\zeta _{j}}x(t,\zeta
		)]+\sum_{i=1}^{n}\tilde{\alpha}_{i}(t,\zeta )\partial _{\zeta _{i}}x(t,\zeta
		) +b(t,\zeta,x_{t-K,t}(\zeta
		), u_{\mu_1}(t))\\ & +\Big\{\sum_{i=1}^{n}\beta _{i}(t,\zeta )\partial _{\zeta
			_{i}}x(t,\zeta )+\sigma (t,\zeta ,x_{t-K,t}(\zeta), u_{\mu_1}(t))\Big\}{ \dot W(t,\zeta)}, \\ & (t,\zeta )\in \lbrack 0,T]\times
		\mathbb{R}^{d}, \\ x(t,\zeta )& =\gamma(t,\zeta ),\ u(t,\zeta )=v(t,\zeta
		),\quad (t,\zeta )\in \lbrack -K ,0]\times \mathbb{R}^{d}. \end{aligned}%
	\right.
\end{equation*}%
In the above equation, $\alpha _{ij},\tilde{\alpha}_{i},\beta _{i},b,\sigma $ and  $(\gamma,v)$ 
are coefficient functions and initial values, respectively;  $\dot W(t,\zeta)$ is a space-time white noise; $\mu_1$ is a finite measure   on $[-K,0]$ and $u_{\mu_1}(t):=\int_{-K}^{0}u(t+s)\mu_1(ds)$ with  $
u(\cdot) $ being the control process taking values in a convex subset $U$ of a separable Hilbert space $H_{1}$. Suppose $\gamma\in C(-K ,0;H)$ and $v\in L^{2}(-K,0;U)$. If we denote
$dw(t)=\displaystyle\dot W(t,\cdot)dt$, then $\{w(t),\, t\in[0,T]\}$ is a \emph{cylindrical Wiener process} with $\mathcal K=L^2(\R^d)$ (see Section~\ref{sec:pre}).

Consider the problem
of minimizing the cost functional $
J(u(\cdot ))=y(0)$,
where $y(\cdot)$ is the recursive utility subject to the following BSDE: 
\begin{align*}
	y(t)=&\int_{\mathbb{R}^{d}}h(\zeta ,x(T,\zeta ))d\zeta +\int_{t}^{T}\int_{%
		\mathbb{R}^{d}}f\big(s,\zeta,x_{s-K,s}(\zeta),y(s),z(s),u_{\mu_1}(s)\big)d\zeta
	ds\\&-\int_{t}^{T}z(s)d w(s).
\end{align*}
Take 
\begin{equation*}
	A(t)=\displaystyle\sum_{i,j=1}^{n}\partial _{\zeta _{i}}[\alpha _{ij}(t,\zeta )\partial
	_{\zeta _{j}}]+\sum_{i=1}^{n}\tilde{\alpha}_{i}(t,\zeta )\partial _{\zeta
		_{i}},\ B(t)= \displaystyle\sum_{i=1}^{n}\beta _{i}(t,\zeta )\partial _{\zeta _{i}}.
\end{equation*}
Assume that there exist constants $\alpha \in (0,1)$ and $K>1$ such
that 
\begin{equation*}
	\alpha I_{d\times d}+(\beta _{i}\beta _{j})_{d\times d}\leq 2(\alpha
	_{ij})_{d\times d}\leq KI_{d\times d},
\end{equation*}
and impose  proper regularity conditions on the coefficients $b,\sigma ,h$ and $f$, such that (H1)-(H5) hold. Then,
we can obtain the  maximum principle by  Theorem~\ref{SMP}, and its sufficiency under proper convex assumptions  by Theorem~\ref%
{sufficient}.

\subsection{LQ problem for PSEEs}

\label{4.1} Suppose that the control domain is a separable Hilbert space $
H_{1}$ and  $\mathcal{U}=L_{\mathbb{F}}^{2}(0,T;H_{1})$. In \eqref{state}
and \eqref{y},  for $(x,v)\in C(-K,T;H)\times H_1$ and  $x^1,x^2\in H$, let 
\begin{align*}
	b(t,x,v)&=A_{1}(t)x+C(t)v, \\
	\sigma (t,x,v)&=B_{1}(t)x+D(t)v, \\
	h(x^1)&=\big\langle \Phi x^1,x^1\big\rangle_{H},\\
	f(t,x^2,y,z,v)&=  \left\langle F(t)x^2,x^2\right\rangle
	_{H}+{G}_{1}{(t)y+G}_{2}{(t)z}+\left\langle N(t)v,v\right\rangle _{H_{1}}, 
\end{align*}
where $A_{1}:[0,T]\times \Omega \rightarrow \mathcal{L}\big( C(-K,T;H),H\big)$, $B_{1}:[0,T]\times \Omega \rightarrow \mathcal{L}\big(C(-K,T),\mathcal L_2^0\big)$,  $
C:[0,T]\times \Omega \rightarrow \mathcal{L}(H_{1},H)$, $
D:[0,T]\times \Omega \rightarrow \mathcal{L}(H_{1},\mathcal L_2^0)$, $F:[0,T]\times \Omega \rightarrow \mathcal{L}(H)$, $%
G_{1}:[0,T]\times \Omega \rightarrow \mathbb{R}$,  $
G_{2}:[0,T]\times \Omega \rightarrow \mathcal L\big(\mathcal L_2^0(\mathcal K,\R),\R\big)$, $N:[0,T]\times
\Omega \rightarrow \mathcal{L}(H_{1})$, and $\Phi :\Omega \rightarrow 
\mathcal{L}(H)$.

Then, the control system is as follows: 
\begin{equation}
	\left\{ \begin{aligned}
		dx(t)&=\big[A(t)x(t)+A_1(t)x_{t-K,t}+C(t)u_{\mu_1}(t)\big]dt\\&\hspace{0.5em}+\big[B(t)x(t)+B_1(t)x_{t-K,t}+D(t)u_{\mu_1}(t)\big]dw(t),\quad t\in[0,T],\\
		x(t)&=\gamma(t),\quad u(t)=v(t),\quad t\in[-K,0], \end{aligned}\right.
	\label{lqx}
\end{equation}%
and the recursive utility $y(\cdot )$ is governed by 
\begin{equation}
	\left\{ \begin{aligned} dy(t)=&\,\,\Big\{\left\langle
		F(t)x(t),x(t)\right\rangle_H+ 
		G_1(t)y(t)+G_2(t)z(t)\\&+\left\langle
		N(t)u_{\mu_1}(t),u_{\mu_1}(t)\right\rangle_{H_1}\Big\}dt-z(t)dw(t),\quad\,t\in[0,T],\\
		y(T)=&\,\,\left\langle\Phi x_{\mu_2}(T),x_{\mu_2}(T)\right\rangle_H. \end{aligned}\right.
\end{equation}
We aim to minimize $J(u(\cdot )):=y(0)$ over $\mathcal{U}$. Assume the following conditions.

\begin{itemize}
	\item[(L1)] The operators $A:[0,T]\times \Omega \rightarrow 
	\mathcal{L}(V;V^{\ast })$ and $B:[0,T]\times \Omega \rightarrow \mathcal{L}%
	(V;\mathcal L_2^0)$ satisfy  (A2)-(A3).
	
	\item[(L2)] $\gamma(\cdot )\in C(-K,0;H)$ and 
	$v(\cdot )\in L^{2}(-K ,0;H_{1}) $. The processes $
	A_{1},B_{1},C,D,F, G_1,G_2,N$ are uniformly bounded, $
	A_{1},B_{1},C,D,N$ are  weakly {$\mathbb{F}$-}adapted  (for the definition, see \cite
	[Chapter 1]{81k} and  \cite
	[Section 2]{liu2021maximum}) 
	and $G_{1},G_{2} $ are $\mathbb{F}$-adapted. $\Phi $ is uniformly bounded
	and weakly $\mathcal{F}_{T}$-measurable.
	
	\item[(L3)] $F,$  $\Phi $ are symmetric and nonnegative
	definite for  almost all $%
	(t,\omega )\in \lbrack 0,T]\times \Omega$. Furthermore, $N$ is symmetric and
	uniformly positive definite    for  almost all $%
	(t,\omega )\in \lbrack 0,T]\times \Omega$.
\end{itemize}

The Hamiltonian becomes 
\begin{align*}
	H(t,x,y,z,v,p,q,k) & =\left\langle A_{1}(t)x+C(t)v,p\right\rangle _{H} +\left\langle
	B_{1}(t)x+D(t)v,q\right\rangle _{\mathcal L_2^0}\\&\hspace{2em}- \left\langle F(t)x,x\right\rangle _{H}-G_{1}{(t)y-G}_{2}{(t)z}
	-\left\langle N(t)v,v\right\rangle _{H_{1}}.
\end{align*}
Assume that $\bar{u}(\cdot )$ is an optimal control and $\bar{x}(\cdot )$ is
the corresponding solution of equation \eqref{lqx}. Denote, for $Z\in C(-K,T;H)$, \begin{align*}
	\rho_{b,t}(Z)=A_1(t)Z_{t-K,t}, \,\,\,\rho_{\sigma,t}(Z)=B_1(t)Z_{t-K,t},\,\,\,\rho_{f,t}(Z)=2\langle F(t)\bar x(t),Z(t)\rangle.
\end{align*}
Note that $\rho^*_{f,t}(Q)=2F(t)\bar x(t)Q(t), $ for $Q\in L^{2}(0,T;\mathbb{R}).$ Then the adjoint
equation  is 
\begin{equation*}
	\left\{ \begin{aligned} p(t)=&-\int_{(t,T]\cap [T-K,T]}\mathbb E^{\mathcal F_s}[2k(T)\Phi x_{\mu_2}(T)]\mu_2(d(s-T))-2\int_t^TF(s)\bar x(s)k(s)ds\\&+\int_t^T\Big\{A^*(s)p(s)+B^*(s)q(s)+\mathbb E^{\mathcal
			F_s}\Big[\rho_{b,s}^*( p_{s,s+K}|_{[0,T]})+\rho_{\sigma,s}^*(q_{s,s+K}|_{[0,T]})\Big]\Big\}ds\\&-\int_t^Tq(s)dw(s),\quad \,t\in[0,T],\\
		p(t)&=q(t)=0,\quad t\in(T,T+K], \end{aligned}%
	\right.
\end{equation*}
with $k(\cdot )$ satisfying
\begin{equation*}
	\left\{ \begin{aligned} dk(t)=&\,\,
		G_1(t)k(t)dt+{ G_2(t)k(t)dw(t)},\quad t\in[0,T],\\ k(0)=&\,\,-1. \end{aligned}\right.
\end{equation*}
Now \eqref{ne} in the maximum principle becomes 
\begin{align*}
	& \Big<\mathbb{E}^{\mathcal{F}_{t}}\Big[%
	\int_{-K}^{0}C^{\ast }(t-s)p(t-s)\mu_1(ds) +\int_{-K}^{0}D^{\ast }(t-s)q(t-s)\mu_1(ds)\Big]\\&-2N(t)\bar{u}%
	(t),u-\bar{u}(t)\Big>_{H_{1}}=0.
\end{align*}
From this we can deduce that 
\begin{equation*}
	\mathbb{E}^{\mathcal{F}_{t}}\Big[
	\int_{-K}^{0}C^{\ast }(t-s)p(t-s)\mu_1(ds)+\int_{-K
	}^{0}D^{\ast }(t-s)q(t-s)\mu_1(ds)\Big]-2N(t)\bar{u}(t)=0,
\end{equation*}
and thus, 
\begin{equation*}
	\bar{u}(t)=-\frac{1}{2}N^{-1}(t)\Big\{
	\mathbb{E}^{\mathcal{F}_{t}}\Big[\int_{-K}^{0}C^{\ast
	}(t-s)p(t-s)\mu_1(ds)+\int_{-K}^{0}D^{\ast }(t-s)q(t-s)\mu_1(ds)\Big]\Big\}.
\end{equation*}
By Theorem \ref{sufficient}, we can conclude that $\bar{u}(\cdot )$ defined
above is indeed an optimal control of the LQ problem.

\appendix

\noindent\textbf{Acknowledgments.} We wish to thank Ying Hu for helpful discussions and comments. G. Liu's
Research is partially supported by National Natural Science Foundation of
China (No. 12201315 and  No. 12571479) and the Fundamental Research Funds for the Central
Universities, Nankai University (No. 63221036). J. Song is partially supported by  National Key R\&D Program of China (No. 2023YFA1009200), National Natural Science Foundation of
China (No. 12471142), and the Fundamental Research Funds for the Central Universities.

\end{document}